\documentclass[mathpazo]{cicp}
\usepackage{graphicx}
\usepackage{subfigure}
\usepackage{textcomp}
\usepackage{amssymb}
\usepackage{amsmath}
\usepackage{amsthm}
\usepackage{comment} 
\usepackage{algorithm}
\usepackage{multirow} 
\usepackage{setspace}
\usepackage{float}
\usepackage{algpseudocode}
\usepackage{algorithm}
\usepackage{easybmat}
\usepackage{color}
\usepackage{hyperref}
\usepackage{ulem}
\usepackage{epstopdf}
\usepackage{dotlessi}
\usepackage{cleveref}

\newcommand{\D}{\mathcal{D}}
\newcommand{\B}{\mathfrak{B}}

\renewcommand{\div}{\operatorname{div}}

\newcommand{\bE}{\mathbb{E}}

\def\dx{\,{\rm d}x}
\def\dy{\,{\rm d}y}

\def\ds{\,{\rm d}s}
\def\dt{\,{\rm d}t}
\def\I{\mathcal{I}}
\def\Ih{\hat{\mathcal{I}}}

\usepackage{comment}

\usepackage{subfigure}
\usepackage{cite}

\DeclareMathOperator*{\argmin}{arg\,min}

\definecolor{sxcol}{rgb}{0, 0, 1}
\definecolor{lzcol}{rgb}{1, 0, 0}
\definecolor{xlcol}{rgb}{0, 1, 0}
\newcommand{\sx}[1]{{\color{sxcol} #1}}

\usepackage{ulem} 

\begin{document}
\title{Generalized Rough Polyharmonic Splines for Multiscale PDEs with Rough Coefficients}

 \author[Liu X. e t.~al.]{Xinliang Liu\affil{1}, Lei Zhang\affil{1}\comma\corrauth 
 and
       Shengxin Zhu\affil{2}\comma\affil{3}$^*$
       }
 \address{\affilnum{1}\ Institute of Natural Sciences, School of Mathematical Sciences,
 	and MOE-LSC, Shanghai Jiao Tong University. \\
           \affilnum{2}\    Research Center for mathematics, Beijing Normal University,  Zhuhai 519087, \\
           \affilnum{3}\    Division of Science and Technology, BNU-HKBU United International College, Zhuhai 519087
           }
 \emails{{\tt  lzhang2012@sjtu.edu.cn} (L.~Zhang), {\tt liuxinliang@sjtu.edu.cn} (X.~Liu),
          {\tt Shengxin.Zhu@bnu.edu.cn}, {\tt shengxinzhu@uic.edu.cn} (S.~Zhu)}


\date{\today}



\begin{abstract}
In this paper, we demonstrate the construction of generalized Rough Polyhamronic Splines (GRPS) within the Bayesian framework, in particular, for multiscale PDEs with rough coefficients. The optimal coarse basis can be derived automatically by the randomization of the original PDEs with a proper prior distribution and the conditional expectation given partial information on edge or derivative measurements. We prove the (quasi)-optimal localization and approximation properties of the obtained bases, and justify the theoretical results with numerical experiments. 

\end{abstract}


\maketitle
\section{Introduction}
\def\L{\mathcal{L}}
\def\B{\mathcal{B}}
\def\H{\mathcal{H}}

Problems with a wide range of coupled temporal and spatial scales are ubiquitous in many phenomena and processes of materials science and biology. Multiscale modeling and simulation is essential in underpinning the discovery and synthesis of new materials and chemicals with novel functionalities in key areas such as energy, information technology and bio-medicine. 

There has been many existing work concerning the design of novel numerical methods for multiscale problems and the mathematics to foresee and assess their performance in engineering and scientific applications, such as homogenization \cite{papanicolau1978asymptotic,jikov2012homogenization}, numerical homogenization \cite{dur91,ab05,weh02}, heterogeneous multi-scale methods \cite{ee03,abdulle2014analysis, ming2005analysis,li2012efficient}, multi-scale network approximations \cite{berlyand2013introduction}, multi-scale finite element methods \cite{Arbogast_two_scale_04, eh09, Review,chen2015mixed},
variational multi-scale methods \cite{hughes98, bazilevs2007variational}, flux norm homogenization \cite{berlyand2010flux,owhadi2008homogenization}, rough polyharmonic splines (RPS) \cite{OwhZhaBer:2014}, generalized multi-scale finite element methods \cite{egh12, chung2014adaptiveDG,  chung2015residual}, localized orthogonal decomposition \cite{MalPet:2014,Henning2014,Henning2015,Peterseim2017a}, etc.  Fundamental questions for numerical homogenization are: how to approximate the high dimensional solution space by a low dimensional approximation space with optimal error control, and furthermore, how to construct the approximation space efficiently, for example, whether its basis can be localized on a coarse patch. Surprisingly, those questions have deep connections with Bayesian inference, kernel learning and probabilistic numerics \cite{Owhadi2015,OwhadiMultigrid:2017,owhadi2020kernel,owhadi2019kernel}.

In this paper, we generalize the so-called \textit{Rough Polyharmonic Splines} (RPS) \cite{OwhZhaBer:2014} within the Bayesian framework \cite{Owhadi2015} for the following integral-differential equation
\begin{equation}
\begin{cases}
\mathcal{L} u &= g, \quad \text{on } \Omega, \\
\mathcal{B} u &=0, \quad \text{on } \partial \Omega.
\end{cases}
\label{eqn:ellp}
\end{equation}
where $\L$ and $\B$ are integro-differential operators on $\Omega$ and $\partial\Omega$, such that 
$(\L, \B): \H(\Omega) \to \H_{\L}(\Omega) \times \H_{\B}(\partial \Omega)$, where $\H(\Omega), \H_{\L}(\Omega)\text{ and }\H_{\B}(\partial \Omega)$ are Hilbert spaces of generalized functions on $\Omega$ and $\partial \Omega$, such that $\H(\Omega) \subset L^2(\Omega)\subset \H_{\L}(\Omega)$. 

A prototypical example is the second order divergence form elliptic equation with rough coefficients, such that $\mathcal{L}=- \div (\kappa(x) \nabla \cdot) $, $\B = \mathrm{Id}$, and $\Omega$ is a simply connected domain with piecewise smooth boundary $\partial\Omega$. The \textit{rough coefficient}, $\kappa(x) \in L^{\infty}(\Omega)$, represents multiscale media with high contrast and fast oscillations. We only require $\kappa$ to be uniformly elliptic on $\Omega$, i.e., that $\kappa$ is uniformly bounded from above and below by two strictly positive constants, denoted by $\kappa_{min}$, $\kappa_{max}$. 
For this example, we have  $\H(\Omega)=H^1_0(\Omega)$, and $\H_{\L}(\Omega)=H^{-1}(\Omega)$.

It is well-known that for an arbitrary $\kappa$, solving the elliptic equation with linear or polynomial finite element methods can be arbitrarily slow \cite{BO00}. To tackle with such a challenge, last decades has witnessed the fast development of multiscale finite element methods \cite{HouWu:1997, EfeHou:2009b} and numerical homogenization approaches \cite{OwhZha:2007, MalPet:2014, OwhZhaBer:2014}. One essential component of these methods is the construction of a proper coarse space with desired approximation and localization properties. The Bayesian homogenization approach \cite{Owhadi2015} provides a unified framework for such constructions \cite{OwhZhaBer:2014}. 

Under the Bayesian framework, the \textit{generalized Rough Polyharmonic Splines} (GRPS) space can be identified by the choice of random noise and measurement function. Point and volume measurements have been used in \cite{OwhZhaBer:2014} and \cite{OwhadiMultigrid:2017}, respectively. 
In this paper, we construct two new GRPS spaces based on the edge measurements or derivative measurements, and provide rigorous proof of their approximation and localization properties. It is sometimes natural to use edge based measurements due to the presence of elongated structures such as cracks and channels in heterogeneous media. The derivative based GRPS can be seen as a higher order method. 

We note that our method is different from the so-called edge multiscale finite elements in \cite{fu2019edge}, which forms the multiscale finite element space by solving local Steklov eigenvalue problems and a local harmonic function with constant flux (also appears in the mixed multiscale finite element method, for example \cite{chen2003mixed}). 

The paper is organized as follows: in Section \ref{sec:formulation}, we first introduce the Bayesian homogenization framework and the variational formulation of numerical homogenization (coarse) basis, then we present the details for the construction of such basis. In Section \ref{sec:analysis}, we provide the rigorous error analysis of the corresponding numerical homogenization method. Numerical examples are presented in Section \ref{sec:numerics} to validate the method. We conclude the paper in Section \ref{sec:conclusion}.

\paragraph{Notations} The symbol $C$ (or $c$) denotes generic positive constant that may change from one line of an estimate to the next. The dependence of $C$ will be clear from the context or stated explicitly. We use standard notations $L^2(\Omega)$, $H^1(\Omega)$ for Lebesgue and Sobolev spaces, and $H_{0}^{1}(\Omega):=\left\{u \in H^{1}(\Omega): u=0 \text { on } \partial\Omega\right\}$. For any measurable subset $\omega\subset\Omega$, the $d$ or $(d-1)$ dimensional Lebesgue measure of $\omega $ is denoted by $|\omega|$ and the $L^2$ norm is denoted by $\|\cdot\|_{L^2(\omega)}$. We denote $\#$  the cardinality of a set. 

\section{Formulation}
\label{sec:formulation}

\subsection{Bayesian homogenization framework}
\label{sec:formulation:Bayesian}

In the Bayesian homogenization framework \cite{Owhadi2015}, the coarse space can be identified by a Bayesian inference problem (in particular, Gaussian process regression), through the randomization of the original deterministic problem \eqref{eqn:ellp},
\begin{equation}
\begin{cases}
\mathcal{L} v(x) = \zeta(x), \quad &\text{on } \Omega, \\
\mathcal{B}v =0, \quad &\text{on } \partial \Omega.
\end{cases}
\label{eqn:ranellp}
\end{equation}
where $\zeta(x)$ is a centered Gaussian process on $\Omega$ with covariance $\Lambda(x,y)$. The solution $v(x)$ is also a centered Gaussian process on $\Omega$ with covariance 
\begin{equation}
   \Gamma(x,y): = \bE[v(x)v(y)]= \int_{\Omega\times\Omega}G(x,z)\Lambda(z,z')G(y,z')dzdz', 
   \label{eqn:cov}
\end{equation}
where $G(x,y)$ is the Green's function such that $\mathcal{L}G(x,y) = \delta(x-y)$ on $\Omega$, $\mathcal{B}G = 0$ on $  \partial\Omega$.


Given an index set $\I$ and a set of linearly independent \textit{measurement functions} $\Phi:=\{\phi_i(x)\}_{i\in \I}$ , such that  $\int_{\Omega\times\Omega} \phi_i(x) \Gamma(x,y) \phi_j(x)\dx $ are well-defined, we define the measurements $M:=(m_1, \dots, m_{N})$, where $m_i:=\int_{\Omega} v(x)\phi_i(x)\dx$ for $i\in \I$ and $N=\#\I$. For example, when $\phi_i(x)=\delta(x-x_i)$, $m_i$ is the point value at $x_i$ for any continuous $v$. Note that $M$ is a centered Gaussian vector with covariance matrix $\Theta$, where 
$$
\Theta_{i,j} : = \int_{\Omega\times\Omega} \phi_i(x) \Gamma(x,y) \phi_j(y) \dx \dy.$$

The optimal approximation space can be identified through the conditional expectation of $v(x)$ with respect to measurements $M$,
\begin{equation}
\mathbb{E}[v|M] = \sum_{i\in \I} m_i \psi_i(x),
\end{equation}
where $\psi_i$ has the following explicit representation formula from the conditional expectation of Gaussian process,
\begin{equation}
\psi_i(x): = \sum_{j\in \I} \Theta_{i,j}^{-1}\int_\Omega \Gamma(x,y) \phi_j(y) \dy.
\label{eqn:psi_baye}
\end{equation}

$ \{\psi_i\}_{i\in\I}$ can be regarded as a set of posterior basis with respect to the noise $\zeta$ and the measurement functions $\Phi$. $\Psi:=\mathrm{span} \{\psi_i\}_{i\in\I}$ can be used as a coarse space to approximate the solution of deterministic problem \eqref{eqn:ellp}. 
We note that the above formulation can be seen as a prototypical example for the emerging field of probabilistic numerics, and we refer interested readers to \cite{Owhadi2015,OwhadiMultigrid:2017,owhadi2019operator,cockayne2019bayesian} for more details. 

It is difficult to use \eqref{eqn:psi_baye} for numerical computation since it involves convolutions over the whole domain $\Omega$. We will give a variational formulation for $\psi_i$ in the next Section \ref{sec:formulation:variation}.
Before proceeding, we first discuss the choice of the noise $\zeta$ and the measurement functions $\Phi$.


\paragraph{Choice of the noise $\zeta$}
For the centered Gaussian field $\zeta$, it reduces to the choice of the covariance function $\Lambda(x,y)$ , which in turn determines the regularity of the solution space. There are two natural possibilities: 
\begin{enumerate}
\item (white noise) Taking $\zeta(x)$ as the white noise, i.e. $\Lambda(x,y)= \delta(x,y)$. This is the choice for the RPS formulation in \cite{OwhZhaBer:2014,Owhadi2015}.
\item ($\L$ noise) Taking the covariance operator of $\zeta$ as $\L$, namely, for any $g$ such that $\int g\L g\dx  <\infty$, $\int g(x)\zeta(x) \dx$ is a Gaussian random variable with mean 0 and  variance $\int g\L. g\dx $. This is the choice for the Gamblet formulation in \cite{OwhadiMultigrid:2017}. 
\end{enumerate}

\paragraph{Choice of $\Phi$} 
The RPS formulation uses point value measurement functions $\phi_i(x) = \delta(x-x_i)$, while the Gamblet formulation takes scaled volume characteristic functions as measurement functions, and volume averages as measurements. In this paper, we take edge averages or volume averaged first order derivatives as measurements to construct the approximation space $\Psi$. We postpone the specification of the measurements $\Phi$ after we set up the discretization in Section \ref{sec:formulation:numericalmethod}.

In the following, we refer to the basis with white noise as RPS basis, and the basis with $\L$ noise as GRPS basis. For instance, we name the RPS basis with point measurement as RPS-P basis (or in short, RPS basis), GRPS basis with volume measurements as GRPS-V basis, and so on. 

\subsection{Variational formulation}
\label{sec:formulation:variation}

We introduce the solution space to the original problem \eqref{eqn:ellp} as, 
\begin{equation}
V:=\{ v | \mathcal{L} v \in L^2(\Omega), \mathcal{B} v =0 \text{ on } \partial \Omega \}.
\end{equation}

We define the bilinear form $a:V\times V\to \mathbb{R}$ as
\begin{equation}
    a ( u, v ) = \left\{
    \begin{array}{cc}
        \int_\Omega (\L u) (\L v) \dx, & \text{$\zeta$ is white noise },\\
        \int_\Omega u\L v \dx, & \text{$\zeta$ is $\L$ noise},
    \end{array}
    \right.
\end{equation}
and the norm $\|\cdot\|:=\big(a(\cdot,\cdot)\big)^\frac12$. $[u,v]:= \int_{\Omega} u v \dx$ is the scalar product. 

Instead of using the Bayesian representation formula \eqref{eqn:psi_baye}, we propose the following variational formulation to compute the basis. For any $i\in \I$,
\begin{equation}
\begin{cases}
\psi_i =\argmin\limits_{v\in V} a(v,v) \\
s.t.\ [v, \phi_j] =\delta_{i,j},\ \forall j\in \I.
\end{cases}
\label{eqn:psi}
\end{equation}
Then $\Psi:=\operatorname{span}\{\psi_i\}_{i\in \I}$ is the GRPS space.
The well-posedness of \eqref{eqn:psi} and the variational property of the derived basis are shown in the following proposition.

\begin{proposition}\cite[Prop 4.2]{Owhadi2015}
	The constrained minimization problem \eqref{eqn:psi} is strictly convex and admits a unique minimizer $\psi_i\in V$, which also satisfies the Bayesian formula \eqref{eqn:psi_baye}. Furthermore, $\psi_i$ fulfils the variational property in the sense that, for any $v$ such that $[v, \phi_j] =0,\ \forall j\in \I$, we have $a(\psi_i,v)=0$.
	\label{prp:var}
\end{proposition}

\begin{remark}\label{prop:deriv}
The constrained minimization problems \eqref{eqn:psi} is equivalent to the following saddle point problem \cite{Owhadi2015}, namely,  finding  $\psi_i \in V$ and $\lambda \in \Phi $ such that
\begin{equation}
\begin{cases}
a(\psi_i,v)+[ \lambda,  v ] = 0, \, \forall v \in V,\\
[\mu, \psi_i] =f_i(\mu), \, \forall \mu \in \Phi,
\end{cases}
\end{equation}
where  $f_i$ is a linear functional on $\Phi$ such that $f_i(\phi_j) =\delta_{i,j},\, \forall j\in\I$.
\end{remark}


\subsection{Numerical Method}
\label{sec:formulation:numericalmethod}

In this section, we present the discretization of \eqref{eqn:ellp} using finite element methods. We focus on the second order elliptic operator $\mathcal{L}=- \div (\kappa(x) \nabla \cdot)$ with a rough coefficient $\kappa(x)$, and the Dirichlet boundary condition such that $\mathcal{B} = Id$. 

Let $\Omega\subset\mathbb{R}^{d}$ ($d\geq 2$) be an open, bounded, and connected polyhedral domain with a Lipschitz boundary $\partial \Omega$. Let the coefficient $\kappa(x)\in (L^\infty)^{d\times d}$, $0<\kappa_{min}:= \inf_{x\in\Omega}\lambda_{\min} (\kappa(x))$ and $ \sup_{x\in\Omega} \lambda_{\max} (\kappa(x))=:\kappa_{max} < \infty$. The variational problem corresponding to \eqref{eqn:ellp} is 
\begin{equation}
a(u,v)=[g,v],\  \forall v\in H_0^1(\Omega),
\label{eqn:varellp}
\end{equation} 
where $a ( u, v ):= \int_{\Omega} u \mathcal{L} v \dx =\int_{\Omega}\kappa(x)\nabla u \cdot \nabla v \dx$.

\def \TH{\mathcal{T}_{H}}
\def \Th{\mathcal{T}_{h}}
\def \EH{\mathcal{E}_{H}}
\def \NH{\mathcal{N}_{H}}

Let $\TH$ be a coarse simplicial subdivision of $\Omega$, where $H:=\max_{\tau \in\TH}H_{\tau}$ is the coarse mesh size, and $H_{\tau} := \mathrm{diam}(\tau)$. We assume that $\TH$ is shape regular in the sense that $\max_{\tau\in\TH} \frac{H_{\tau}}{\rho_{\tau}} \leq \gamma$, for a positive constant $\gamma>0$, where $\rho_{\tau}$ is the radius of the inscribed circle in $\tau$. We denote $\NH$ and $\EH$ as the set of all vertices and $d-1$ dimensional faces(or edges) in $\TH$, respectively. A fine mesh $\Th$ with mesh size $h = 2^{-J}H$ can be obtained by uniformly subdividing $\TH$ $J$ times. We refine ${\mathcal{T}_H}$ $J$ times to obtain the fine mesh $\mathcal{T}_h$, namely, $h = 2^{-J} H$. See Figure \ref{fig:mesh} for an illustration of the coarse and fine mesh over a square domain. The finite element space $V_h$ contains continuous piecewise linear functions with respect to $\Th$ which vanish at the boundary $\partial \Omega$. 

\begin{figure}[H]
	\centering
	\subfigure[Coarse mesh, $N_c=2$]{
		\includegraphics[width=0.3\textwidth]{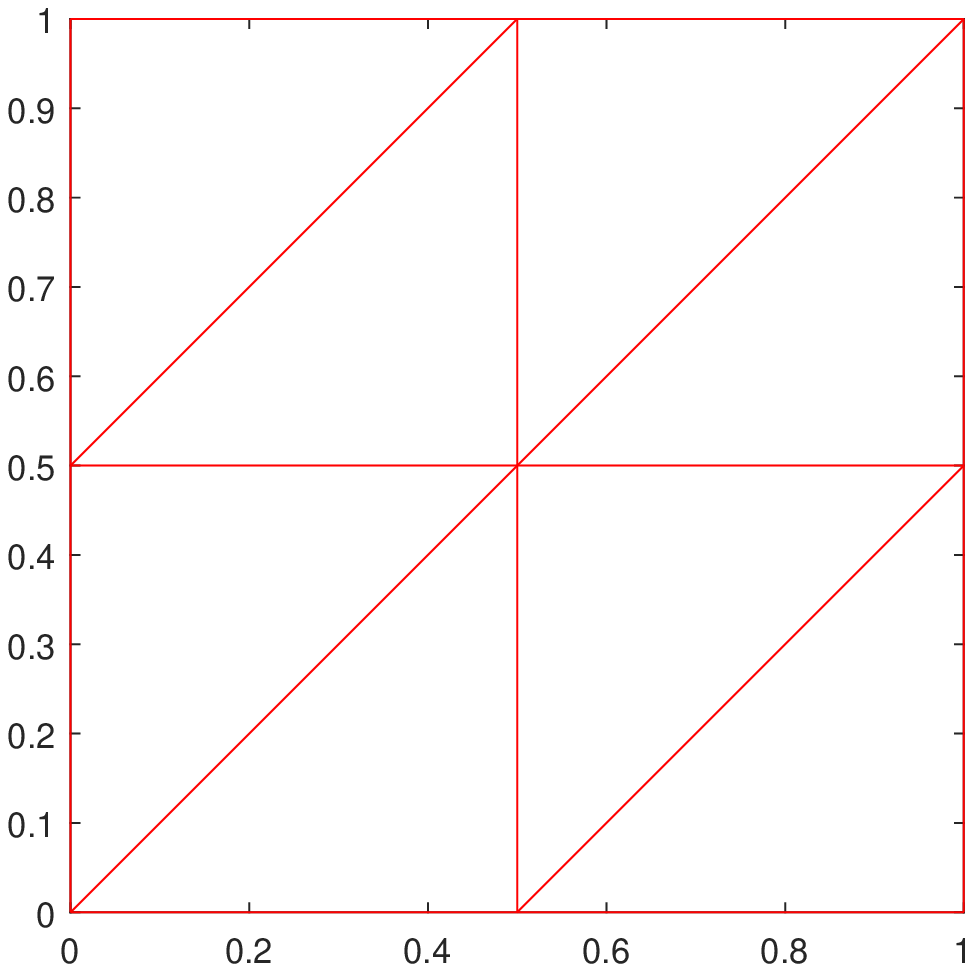}}
		\qquad
	\subfigure[Fine mesh, $N_c=2$,$J=2$]{
		\includegraphics[width=0.3\textwidth]{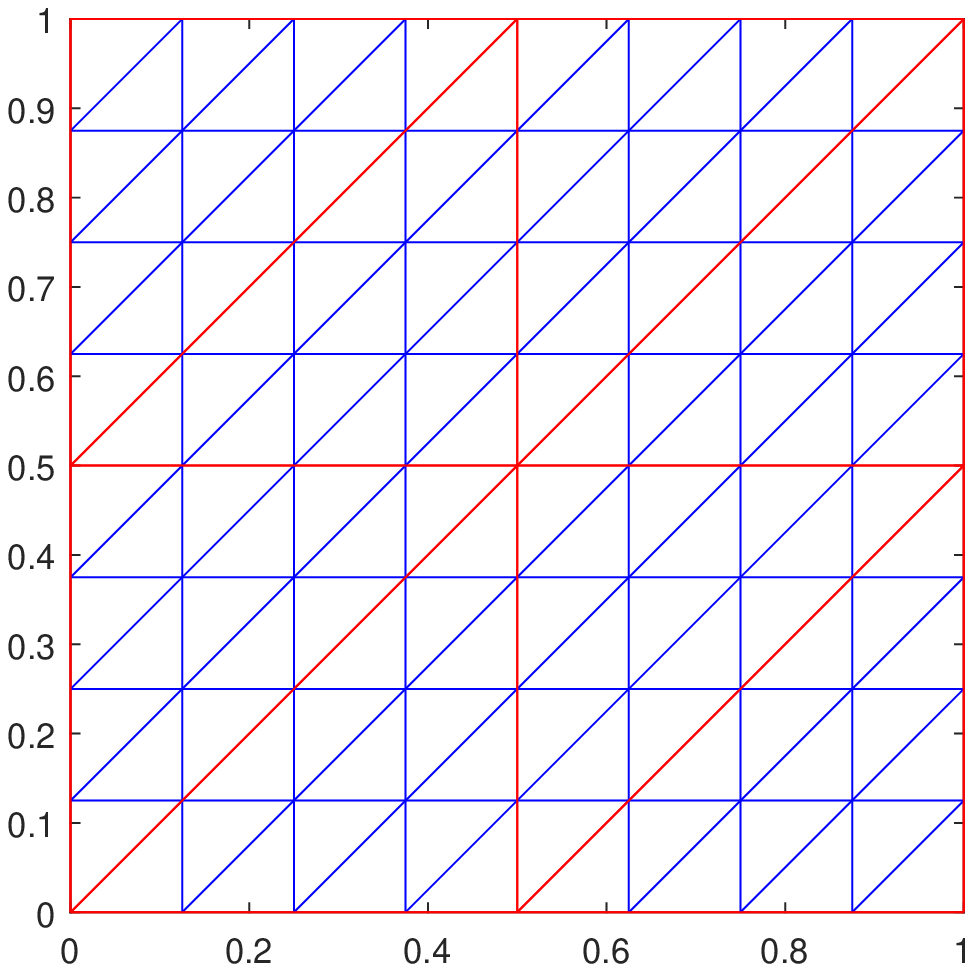}}
	\centering

	\caption{Coarse and fine meshes of the unit square: The regular coarse mesh $\mathcal{T}_H$ is obtained by first subdividing $\Omega$ uniformly into $N_c\times N_c$ squares, then partitioning each square into two triangles along the $(1,1)$ direction. We can further refine the coarse mesh uniformly by dividing each triangle into four similar subtriangles.}
	\label{fig:mesh}
\end{figure}

\subsubsection{Measurement Functions}
\label{sec:formulation:numerics:measurement}

We are now ready to elaborate three different measurement function sets $\Phi$ for $d=2$. The results can be extended to higher dimensions without much difficulty. 
\begin{itemize}
    \item (Case V) Volume measurement function: for $\tau\in \TH$, 
        \begin{equation} 
            \phi_{\tau}: = c_\tau \chi(\tau), \text{ and volume measurement } m_{\tau}: = \int_{\Omega} u \phi_{\tau} \dx = c_{\tau}\int_{\tau} u  \dx , 
          \end{equation}
        where $\chi(\tau)$ is the characteristic function of $\tau$ and $c_\tau:=\sqrt{|\tau|}$ is a scaling factor   (see Proposition \ref{prop:iteration} and Lamma \ref{lemma:psi_i^0} for the effect of the scaling factor).  The collection of such measurement functions forms a feasible choice of $\Phi$ and is denoted by $\Phi_\mathcal{T}:= \{\phi_{\tau}\}_{\tau\in\TH}$. 
    \item (Case E) Edge measurement function: for $e\in \EH$,
        \begin{equation}
            \phi_{e}: = c_e\chi(e), \text{ and edge measurement }  m_{e}:= \int_{\Omega} u \phi_{e} \dx = c_e\int_e u \ds, 
            \label{eqn:edge}
        \end{equation}
        where $\chi(e)$ is a generalized function such that $\int_{\Omega} u \chi(e) \dx= \int_{e} u \ds,\,\forall u \in H^1_0(\Omega)$ and $c_e: =|e|^{\frac{2-d}{2(d-1)}}$ is a scaling factor (see Proposition \ref{prop:iteration} and Lamma \ref{lemma:psi_i^0} for the effect of the scaling factor). The collection of such measurement functions is denoted by $\Phi_{\mathcal{E}}:= \{\phi_{e}\}_{e\in \EH}$.
    \item (Case D') First order derivative measurement function: Given a multi-index  $\alpha \in \mathcal{A}:=\{(\alpha_1,\ldots,\alpha_d)| \alpha_i=0 \text{ or } 1, \text{ and }\sum_{i=1}^{d}\alpha_i = 1\}$, $D^{\alpha} u$ denotes the first order (weak) partial derivatives associated with $\alpha$, e.g., $D^{(1,0,\ldots,0)} u= \frac{\partial}{\partial x_1} u$ (we make use of the notation and definition for weak derivative from \cite{evans1998partial}).  For $\alpha \in \mathcal{A}, \text{ and }  \tau \in \TH$, the first order derivative measurement function
        $$
        \phi_{\tau,\alpha}:= D^{\alpha} \phi_{\tau}, 
        \text{ in the sense that } \int_{\Omega} u \phi_{\tau,\alpha} \dx = -\int_{\Omega} D^{\alpha}u \phi_{\tau} \dx, \, \forall u \in H^1_0(\Omega).
        $$
         The set of measurement functions is the union of $\phi_\tau$ and $\phi_{\tau, \alpha}$, for $\tau\in \mathcal{T}$ and $\alpha\in \mathcal{A}$, namely, 
        \begin{equation}
        \Phi_{\mathcal{\D}'}:= \{\phi_{\tau}\}_{\tau\in \TH}\cup \{\phi_{\tau,\alpha}\}_{\tau\in \TH, \alpha \in \mathcal{A}}.
            \label{eqn:caseDp}
        \end{equation} 
        \item (Case D) Combination of volume and edge measurement functions, 
        \begin{equation}
            \Phi_{\D}:= \{\phi_{e}\}_{e\in \EH}\cup \{\phi_{\tau}\}_{\tau\in \TH}.
            \label{eqn:caseD}
        \end{equation}
\end{itemize}

\begin{remark} \label{rem:deriv}
Derivative measurement function can be represented as a linear combination of edge measurement functions. For $\tau\in \TH$ with $\partial \tau = \cup\{e_1,\,e_2,\,e_3\}$, we have
\begin{equation}
	\begin{aligned}
	m_{\tau,\alpha} :=  \int_{\Omega} u \phi_{\tau,\alpha}  \dx =\int_{\Omega} D^{\alpha}u \phi_{\tau} \dx &= c_\tau( \int_{e_1} u n_{1,\alpha} \ds+\int_{e_2} u n_{2,\alpha} \ds +\int_{e_3} u n_{3,\alpha} \ds )  \\
	&=\frac{c_\tau}{c_e}\int_{\Omega} u (n_{1,\alpha}\phi_{e_1}+n_{2,\alpha}\phi_{e_2}+n_{3,\alpha}\phi_{e_3}) \dx,
	\end{aligned}
	\label{eqn:deriv}
\end{equation}
where $n_{i,\alpha}$  denotes the ${\alpha}$ component of the normal direction $\vec{n}_{i}$ of $e_i$, $i = 1, 2, 3$. Therefore, we have $\phi_{\tau,\alpha}= c_\tau/c_e(n_{1,\alpha}\phi_{e_1}+n_{2,\alpha}\phi_{e_2}+n_{3,\alpha}\phi_{e_3})$.
\end{remark}

\begin{proposition}
The set of measurement function $\Phi_D$ defined in \eqref{eqn:caseD} spans the same linear space as $\Phi_{D'}$ defined in \eqref{eqn:caseDp}, for the Dirichlet boundary condition considered in the paper. Moreover, all the measurement functions in $\Phi_{\D}$ are linearly independent, while the first order derivative measurement functions $\phi_{\tau, \alpha}$ can be linearly dependent. See \ref{sec:prop:span_eq} for the proof.
\label{prop:span_eq}
\end{proposition}

By Proposition \ref{prop:span_eq}, we only need to consider case D instead of case D' to avoid working with linearly dependent measurement functions. We construct the GRPS space $\Psi$ for each case, using the variational formulation \eqref{eqn:psi}.
\begin{itemize}
	\item Case V: 
	$\Psi := \mathrm{span}\{\psi_{\tau}\}_{\tau\in \TH},$ where $\psi_{\tau}$ is the solution of \eqref{eqn:psi} with respect to $\Phi_\mathcal{T}$.
	\item Case E:  $\Psi := \mathrm{span}\{\psi_{e}\}_{e\in \EH},$ where $\psi_{e}$ is the solution to \eqref{eqn:psi} with respect to $\Phi_\mathcal{E}$.
    \item Case D: $\Psi := \mathrm{span}( \{\psi^{\D}_{e}\}_{e\in \EH}\cup \{\psi^{\D}_{\tau}\}_{\tau\in \TH})$. We note that $\psi^\D_{\tau}$ needs to satisfy the constraints such that 
	$[\psi^\D_{\tau}, \phi_{\tau'}] = \delta_{\tau, \tau'}$, and 
    $[\psi^\D_{\tau}, \phi_{e}]=0$, which is different from $\psi_{\tau}$ in case V. $\psi^\D_{e}$ needs to satisfy similar constraints.   
\end{itemize}

In the following, when no confusion arises, we also denote the set of measurement functions by $\Phi:=\{\phi_i\}_{i\in \I}$ and the space of basis by $\Psi:= \mathrm{span}\{\psi\}_{i\in \I}$ using a general index set $\I$ without specifying particular measurements and corresponding bases. We write $N = \# \I$.   

\subsubsection{Localization}
\label{sec:formulation:numerics:localization}

\def \Tc{\mathcal{T}_c}

The basis defined in \eqref{eqn:psi} is globally supported in $\Omega$, which is not practical for applications. In this section, we introduce the notions of local patches and also the formulation of localized bases. 

Let the $0$-th layer patch $\Omega_i^0$ be the smallest subset of $\Omega$ such that $\mathrm{supp}(\phi_i) \subset \Omega_i^0$ and consists of simplices in $\TH$. The $\ell$-th layer patch $\Omega^\ell_i = \cup\{\tau\in\TH: \tau\cap\Omega^{\ell-1}_i\neq \emptyset\}$ for $\ell\geq 1$ can be defined recursively. We refer to Figure \ref{fig:patch} and Figure \ref{fig:edge_patch} to illustrate the local patches for volume measurement and edge measurement, respectively. 

\begin{figure}[H]
	\centering
	\subfigure[$\Omega_i^0$]
	{\includegraphics[width=0.32\textwidth]{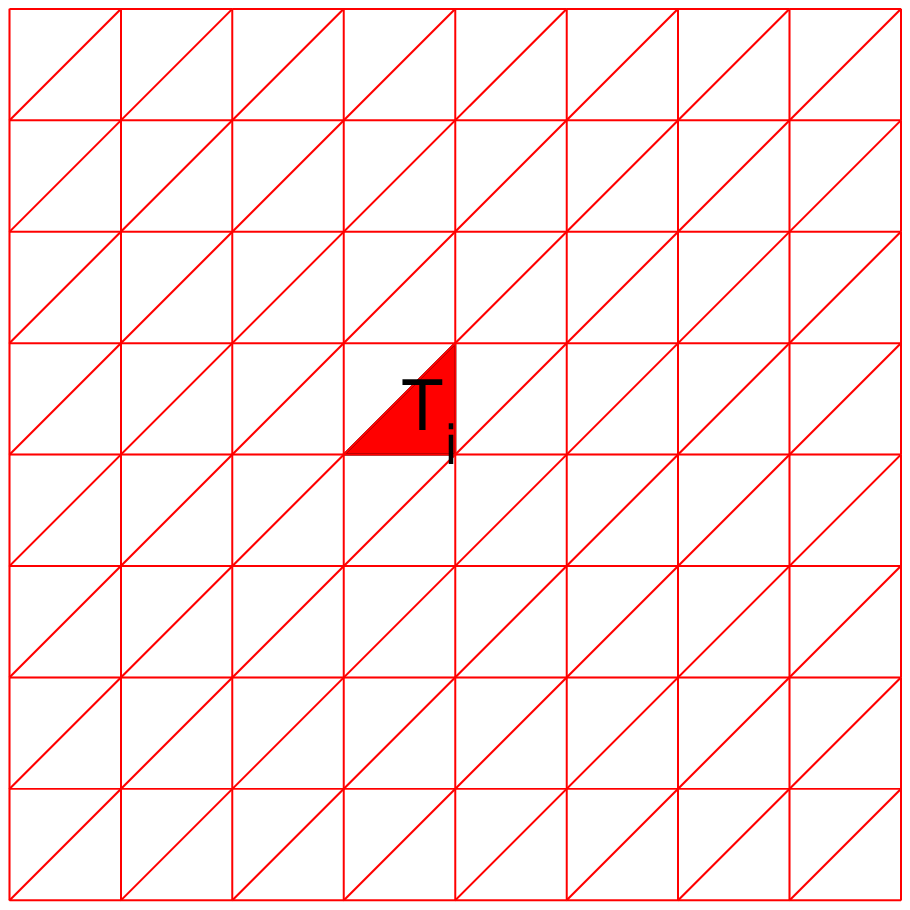}}
	\subfigure[$\Omega_i^1$]
	{\includegraphics[width=0.32\textwidth]{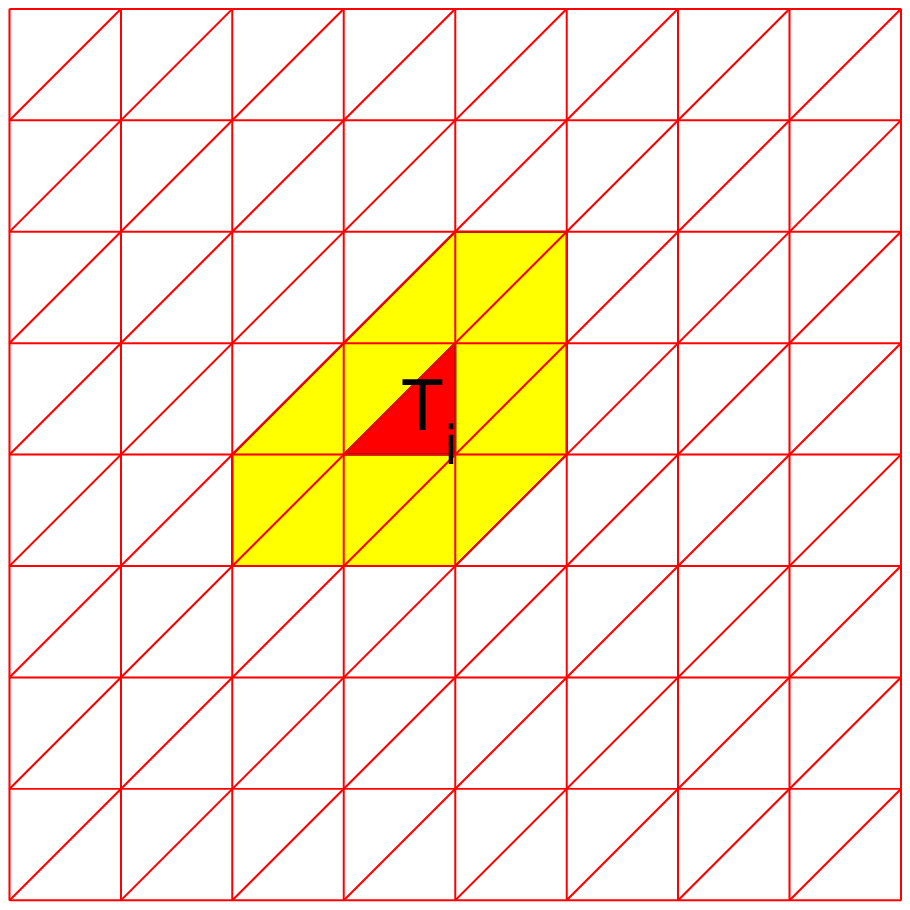}}
	\subfigure[$\Omega_i^2$]
	{\includegraphics[width=0.32\textwidth]{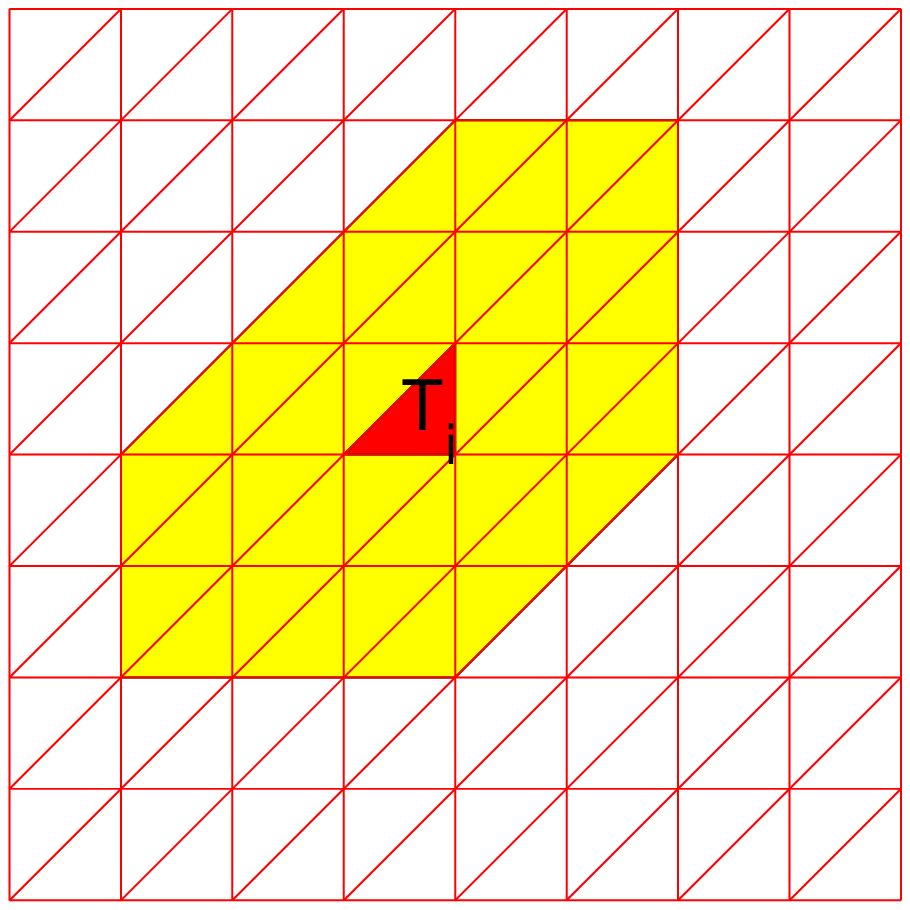}}
	\caption{Local patches for volume measurements.}
	\centering
	\label{fig:patch}
\end{figure}

\begin{figure}[H]
	\centering
	\subfigure[$\Omega_i^0$]
	{\includegraphics[width=0.32\textwidth]{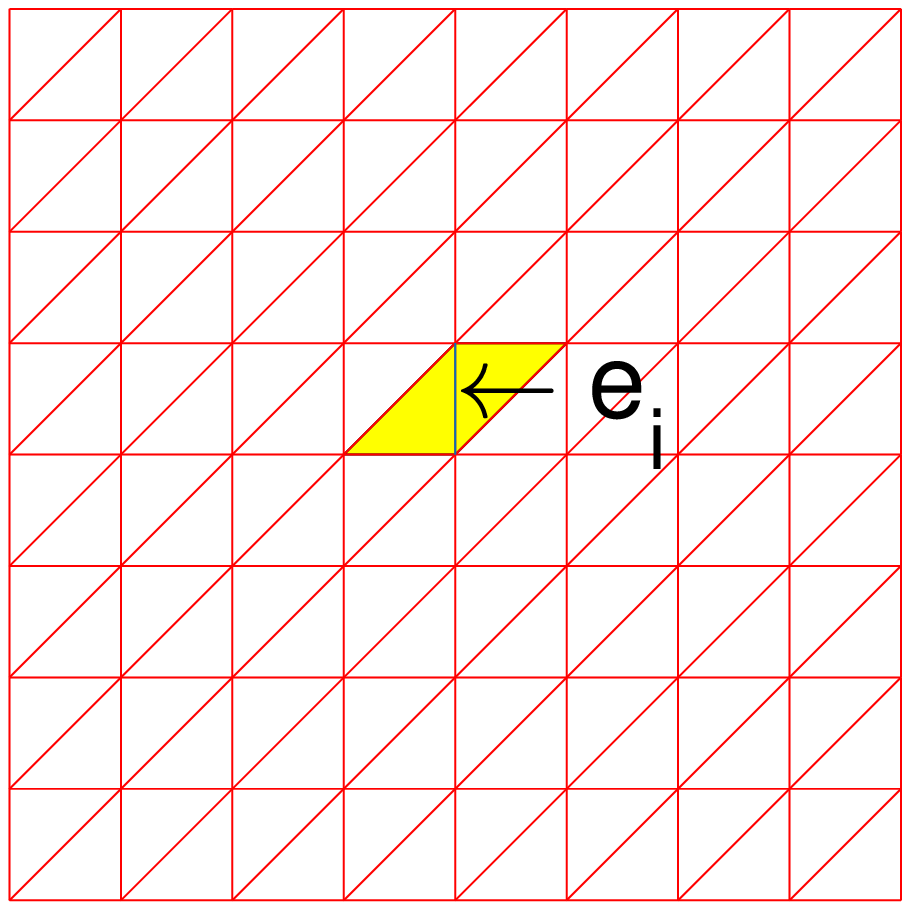}}
	\subfigure[$\Omega_i^1$]
	{\includegraphics[width=0.32\textwidth]{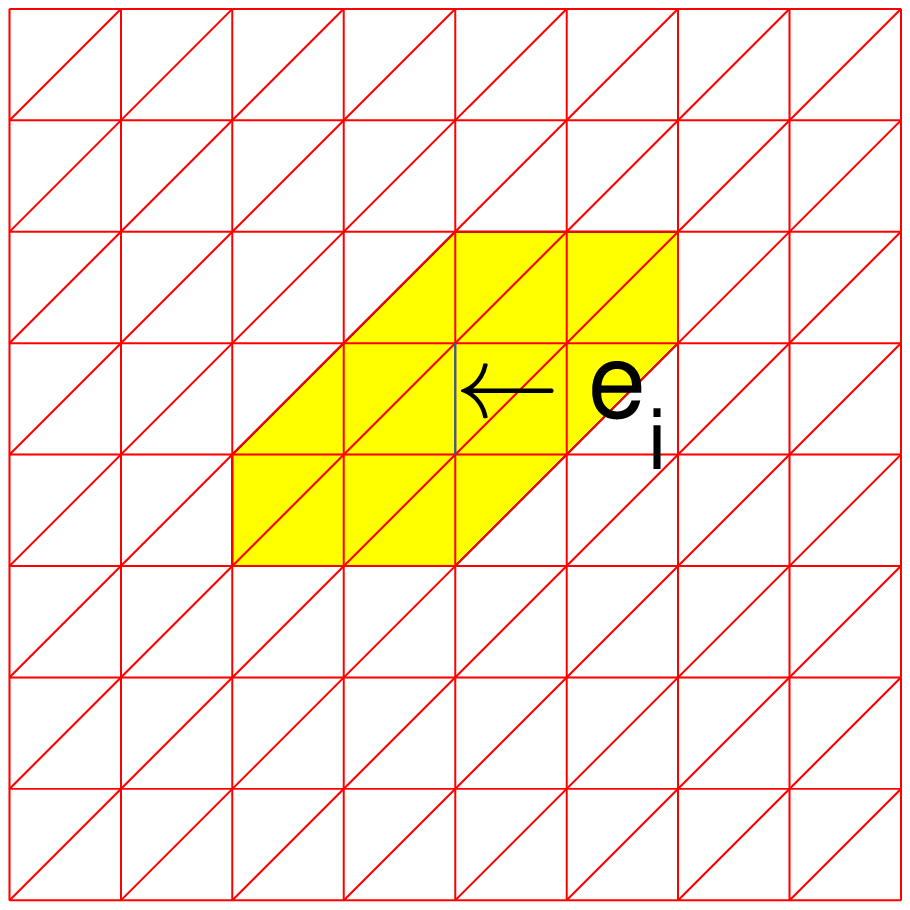}}
	\subfigure[$\Omega_i^2$]
	{\includegraphics[width=0.32\textwidth]{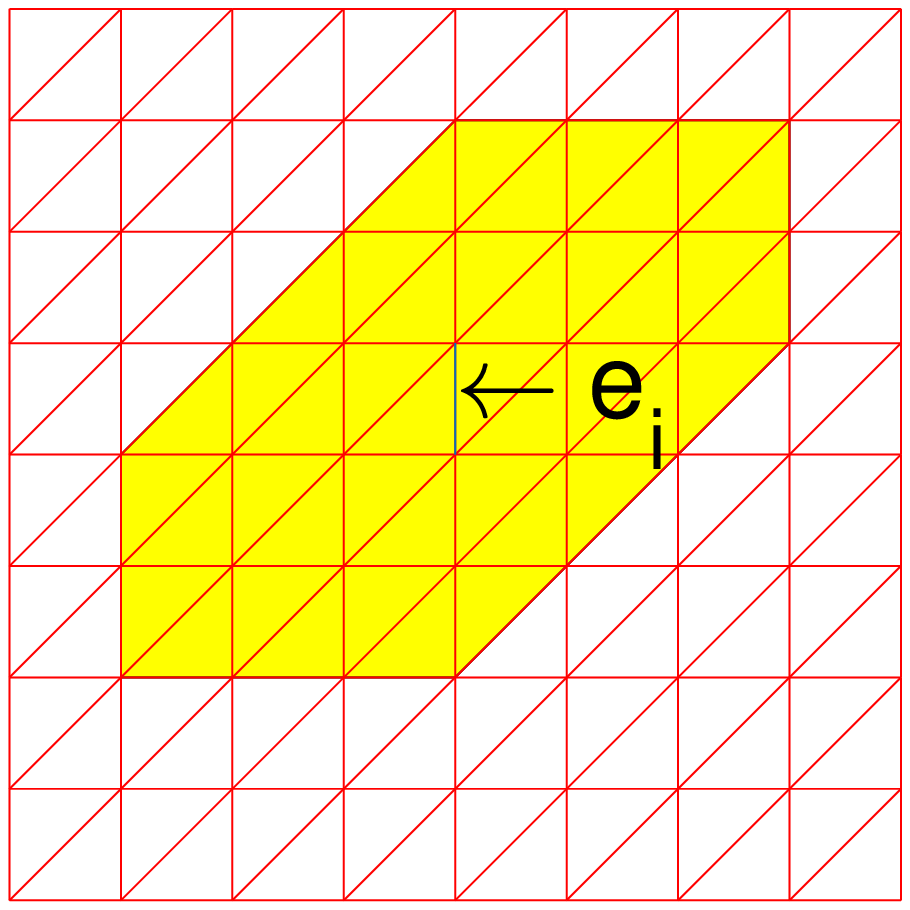}}
	\caption{Local patches for  edge measurements.}
	\centering
	\label{fig:edge_patch}
\end{figure}

To this end, we can localize the computation of $\psi_i$ to a local patch $\Omega^\ell_i$. For any $i\in \I$ and ${\ell}\in \mathbb{N}$,
\begin{equation}
\begin{cases}
\psi_i^{\ell} =\argmin a(v,v) \\
s.t.\ v\in H_0^1(\Omega_i^{\ell})\   and\ [v, \phi_j]=\delta_{i,j},\  \forall j \in \I.
\end{cases}
\label{eqn:psi_i^0}
\end{equation}
The space of localized bases is $\Psi^{\ell} :=\mathrm{span}\{\psi_i^{\ell}\}_{i \in \I}$.

\subsection{Numerical homogenization}

By \textit{numerical homogenization}, we refer to the finite element formulation in the coarse space $\Psi$, namely, to find $u_H \in \Psi$ such that
\begin{equation}
a(u_H,v_H)=[g,v_H],\ \forall v_H\in \Psi.
\label{eqn:varellpFEM}
\end{equation} 

In practice, the  equation \eqref{eqn:varellpFEM} is solved in the space of localized bases $\Psi^{\ell}$, and we write $u_H^{\ell}\in \Psi^{\ell}$ the solution to 
\begin{equation}
a(u_H^{\ell},v_H)=[g,v_H],\ \forall v_H\in \Psi^{\ell}.
\label{eqn:local_varellpFEM}
\end{equation}

For the analysis, we assume that the coarse bases in $\Psi$ or $\Psi^\ell$ are exact solutions of the variational formulation \eqref{eqn:psi} or \eqref{eqn:psi_i^0}. In the numerical experiments, we compute the basis on a sufficiently fine mesh $\Th$, and assume that the discretization error is negligible.

\section{Analysis}
\label{sec:analysis}

In this section, we first present an error analysis in Section \ref{sec:analysis:globalbasis} for the proposed two-level multiscale methods with the basis defined in \eqref{eqn:psi} using edge or first order derivative measurement functions, such that the optimal convergence rate $\mathcal{O}(H)$ for $\| u-u_H \| $ holds (Theorem \ref{thm:convergense of global}), where $u$ is the true solution and $u_H$ is the finite element solution of \eqref{eqn:varellpFEM} with global GRPS basis functions. In Section \ref{sec:analysis:localization}, we propose to compute the basis on a localized patch, and show the exponential decay of the truncation error between the localized and global basis functions in Theorem \ref{thm:decay}. We conclude Section \ref{sec:analysis:localbasis} with our main Theorem \ref{thm:local}, the error estimate for the multiscale method with localized bases, which states that: 

Let $\ell$ indicates the number of layers for the support of the localized basis, and $u_H^{\ell}$ be the solution to \eqref{eqn:local_varellpFEM} with the local bases. The solution error can be controlled by 
$$
\|u -u_H^{\ell}\| \leq \| u -u_H\| +\| u_H- u_H^{\ell}\|.
$$

The first term $\| u -u_H\|$ is of order $\mathcal{O}(H)$, and the second term $\| u_H -u_H^{\ell}\|$ depends on the truncation error decaying exponentially with respect to $\ell$. 


We recall the simplex-wise trace theorem and zero mean boundary type Poincar\'{e} inequality which will be used in the following analysis.

\begin{lemma}[trace inequality \cite{verfurth1999error}]\label{lem:trace}
	For any $\tau \in \mathcal{T}_{H},$ any $e \subset \partial \tau$  and any $v \in H^{1}(\tau)$ we have
	\begin{equation}
	\|v\|_{L^2(e)} \leq\left\{d \frac{|e|}{|\tau|}\right\}^{1 / 2}\left\{\|v\|_{L^2(\tau)}+H_{\tau}\|\nabla v\|_{L^2(\tau)}\right\}	
	\end{equation}	
\end{lemma}

\begin{lemma}[zero mean boundary type Poincar\'{e} inequality \cite{NazRep:2015}]\label{lem:zero_mean}
	Let $\tau$ be a shape regular triangle, then there exists a constant $C_{\tau}$ depending on the diameter of $\tau$ such that
	\begin{equation}
	\|w\|_{L^2(\tau)} \leq C_{\tau}\|\nabla w\|_{L^2(\tau)},\ \forall w\in \tilde{H^1}(\tau):=\left\{w\in H^1(\tau)| \int_{\partial \tau} w = 0\right\},
	\end{equation}
\end{lemma}

\subsection{Accuracy of Global Basis}
\label{sec:analysis:globalbasis}

In this section, we will prove that the finite element solutions to \eqref{eqn:varellpFEM}, with respect to spaces $\Psi$ of global bases derived from case V (volume), case E (edge) or case D (derivative), achieve $\mathcal{O}(H)$ convergence rate.

\begin{theorem}\label{thm:convergense of global}
	Let $u$ be the solution of \eqref{eqn:varellp}, then  $u_H := \sum_{i\in \I} m_i \psi_i$, with $ m_i:=[u,\phi_i]$, is the unique finite element solution to \eqref{eqn:varellpFEM}.  And we have
	\begin{equation}
	\|u-u_H\|_{H_0^1} \leq \kappa_{min}^{-1}C  H\|g\|_{L^2}
	\end{equation}	
\begin{proof}
	Uniqueness is a direct result of the coerciveness of $a(\cdot,\cdot)$. The variational property of $\Psi$, in Proposition \ref{prp:var}, suggests that $a(u_H-u,v)=0,\, \forall v\in \Psi$. Hence  $u_H$ is the unique solution to \eqref{eqn:varellp} over $\Psi$.

	Let $r := u-u_H$, recalling the Galerkin orthogonality $a(r, v) = 0, \forall v\in \Psi$, we have
	\begin{equation}
	\kappa_{min} \|r\|_{H_0^1}^2 \leq a(r,r)=[g,r]\leq \|g\|_{L^2}\|r\|_{L^2}\sx{.} \label{eqn:coercive}
	\end{equation}
	Noting that $[r,\phi_i]=0$ for any $i$, it is sufficient to show 
	\begin{equation}
	    \|r\|_{L^2} \leq C_2 H\|r\|_{H_0^1}
	\label{eqn:poincare}
	\end{equation} for all three cases. 
	
	For case V, the Poincar\'{e} inequality implies $\|r\|_{L^2} \leq C_1 H\|r\|_{H_0^1}$. 
	
	For case E, \eqref{eqn:poincare} can be verified using Lemma \ref{lem:zero_mean}.  
	
	For case D, a  combination of the Poincar\'{e} inequality and Lemma \ref{lem:zero_mean} leads to \eqref{eqn:poincare}.   
\end{proof}
\end{theorem}

\subsection{Localization}
\label{sec:analysis:localization}

The global basis cannot be used directly in practice, therefore, we propose to use the localized basis defined in \eqref{eqn:psi_i^0}. It is crucial to know the level of localization $\ell$ a priori, given accuracy and complexity constraints. In this section, we show that the corrector $\psi_i - \psi^\ell_i $ in all cases V, E and D decays exponentially with respect to $\ell$, enlightened by the idea of subspace decomposition addressed in \cite{Kornhuber2016e,owhadi2019operator}.

First, we introduce a partition of unity.  We use $\Ih$ to denote the index set of the partition of unity to distinguish it from the index set $\I$ for the set of measurement functions. For each $ x_{\hat{\dotlessi}} \in \NH$, let $\omega_{\hat{\dotlessi}}:=\cup\{\tau\in \TH | x_{\hat{\dotlessi}} \in \tau\}$, and $\eta_{\hat{\dotlessi}}$ be the piecewise linear function associated with $x_{\hat{\dotlessi}}$, such that $\eta_{\hat{\dotlessi}}(x_{\hat{\dotlessj}})=\delta_{\hat{\dotlessi},\hat{\dotlessj}}$. Then  $\{\eta_{\hat{\dotlessi}}\}_{\hat{\dotlessi} \in \Ih}$ forms a partition of unity  and
$$
	H_0^1(\Omega)=\sum_{\hat{\dotlessi}\in \Ih} H_0^1(\omega_{\hat{\dotlessi}})
$$
and $\forall v\in H_0^1(\Omega)$,
$$
v = \sum_{\hat{\dotlessi}\in \Ih} v_{\hat{\dotlessi}},
$$
with $v_{\hat{\dotlessi}}= v\eta_{\hat{\dotlessi}}$.

We define  
\begin{equation}
\Phi^{\perp}:=\{u\in H_0^1(\Omega)\,\big|[u,\phi_i]=0, \ \forall i \in \I\},
\end{equation}
and $\Phi_{\hat{\dotlessi}}^{\perp}:=H_0^1(\omega_{\hat{\dotlessi}})\cap \Phi^{\perp}$. 
Let $P_{\hat{\dotlessi}}:H_0^1(\Omega)\to \Phi_{\hat{\dotlessi}}^{\perp}$ be the a-orthogonal projection, such that for any $v \in H_0^1(\Omega)$,
$$
a(P_{\hat{\dotlessi}} v,w)=a(v,w),  \  \forall \, w \in \Phi_{\hat{\dotlessi}}^{\perp}.
$$
The additive subspace decomposition operator $P:\Phi^{\perp}\rightarrow \Phi^{\perp}$ is defined as 
\begin{equation}
    P:= \sum_{\hat{\dotlessi}\in \Ih}P_{\hat{\dotlessi}}.
\end{equation}
We note that the support of $P_{\hat{\dotlessi}}v$ is $\omega_{\hat{\dotlessi}}$. For any $v \in H_0^1(\Omega_i^{\ell})$, the support of $Pv$ is contained in $\cup\{\omega_{\hat{\dotlessi}}| \omega_{\hat{\dotlessi}}\cap\Omega_i^{\ell}\neq \emptyset,  \hat{\dotlessi}\in \Ih \}$. Namely, applying $P$ on $v$ expands the its support by one layer, and $Pv\in H_0^1(\Omega_i^{{\ell}+1})$. The additive subspace decomposition operator $P$ can be utilized as a preconditioner to iteratively approximate any $\chi \in \Phi^{\perp}$. Noticing that the corrector $\psi_i^\ell - \psi_i\in \Phi^{\perp}$, the following proposition shows that if $\mathrm{cond(P)} < \infty$, the truncation error $\|\psi_i^{\ell}-\psi_i\| $ decays exponentially with respect to ${\ell}$.

\begin{lemma}
	If $\mathrm{cond}(P) < \infty$, then 
	\begin{equation}
	\|\psi_i^{\ell}-\psi_i\|\leq \big(\frac{\mathrm{cond}(P)-1}{\mathrm{cond}(P)+1}\big)^{\ell} \|\psi_i^0\|.
	\end{equation} 
	\label{prop:iteration}
	The proof is very similar to \cite{owhadi2019operator}, see \ref{Sec:lemma subspace decomposition} for details.
\end{lemma}

\subsubsection{Condition Number of P}

This subsection is dedicated to an analysis of the condition number of $P$, in particular, for case V, E and D.\\

\begin{lemma} \cite{Kornhuber2016e}\label{lemma:cond(P)}
Let $K_{max}$ be the smallest constant such that
\begin{equation}
\|\chi\|^2 \leq K_{max}\sum_{\hat{\dotlessi}\in \Ih} \|\chi_{\hat{\dotlessi}}\|^2
\end{equation}
holds for any $\chi = \sum_{\hat{\dotlessi}\in \Ih} \chi_{\hat{\dotlessi}}$ with $\chi_{\hat{\dotlessi}} \in \Phi_{\hat{\dotlessi}}^{\perp}$. Let $K_{min}$ be the largest number such that for any $\chi \in \Phi^{\perp} $, there exits a decomposition  $\chi = \sum_{\hat{\dotlessi}\in \Ih} \chi_{\hat{\dotlessi}}$ with $\chi_{\hat{\dotlessi}}\in \Phi_{\hat{\dotlessi}}^{\perp}$ such that
\begin{equation}
K_{min}\sum_{\hat{\dotlessi}\in \Ih}\|\chi_{\hat{\dotlessi}}\|^2\leq\|\chi\|^2.
\label{eqn:decomposition}
\end{equation}
The shape regularity of $\TH$ implies the existence of the overlapping number $n_{\max}$, namely, the maximum number of non-vanishing $\chi_i$ on any given element of $\TH$. We note that $n_{\max}$ depends on $\gamma$ and $d$. It holds true that 
	\begin{equation}
	K_{min} \leq \lambda_{min}(P), \quad  \lambda_{max}(P) \leq K_{max}\leq n_{max}.
	\end{equation}	
\end{lemma}

For the existence of $K_{min}$, in the following, we will construct a specific decomposition for any $\chi \in \Phi^{\perp} $ satisfying \eqref{eqn:decomposition}. 

Note that 
\begin{equation}
\chi = \sum_{\hat{\dotlessi}\in \Ih} v_{\hat{\dotlessi}},
\label{eqn:predecompositon}
\end{equation}
with $v_{\hat{\dotlessi}}=\chi\eta_{\hat{\dotlessi}}$, forms a decomposition of $\chi\in\Phi^{\perp}$; however, $v_{\hat{\dotlessi}}$ does not necessarily belong to $\Phi_{\hat{\dotlessi}}^{\perp}$. We construct the following correction operator $\tilde{P}$ such that $ v_{\hat{\dotlessi}} - \tilde{P} v_{\hat{\dotlessi}} \in \Phi_{\hat{\dotlessi}}^{\perp}$. $\tilde{P}$ varies by the choice of measurement functions $\Phi$. 

The $0$-th layer patch $\Omega_i^0$ is the smallest subset of $\Omega$, consisting of simplices in $\Th$, such that $\mathrm{supp}(\phi_i) \subset \Omega_i^0$.
\begin{itemize}
	\item Case V: $\Omega_i^0= \tau_i$;
	\item Case E: $\Omega_i^0= \omega_{e_i}$, where $\omega_{e_i} := \cup\{\tau\in\TH: \tau\cap e_i\neq \emptyset\}$;
	\item Case D: $\Omega_i^0= \tau_i$ for $\tau_i \in \Phi_{\D}$ and  $\Omega_i^0= \omega_{e_i}$ for $e_i \in \Phi_{\D}$.
\end{itemize}
Let $\psi_i^0$ be the localized basis defined in \eqref{eqn:psi_i^0} with respect to $\Omega_i^0$ .

Let $\tilde{P}: H_0^1(\Omega) \rightarrow H_0^1(\Omega)$ be the linear operator defined by 
\begin{equation}
\tilde{P}v:= \sum_{i\in\I} \psi_i^0 [\phi_i, v],\ for\ v\in H_0^1(\Omega),
\label{eqn:tildeP}
\end{equation}
Although $\phi_i$ and $\psi_i^0$ varies by cases, noticing that $\tilde{P}v_{\hat{\dotlessi}} \in H_0^1(\omega_{\hat{\dotlessi}})$ and $[\psi^0_i,\phi_j] = \delta_{i,j}$, we have $v_{\hat{\dotlessi}}-\tilde{P}v_{\hat{\dotlessi}} \in \Phi_{\hat{\dotlessi}}^{\perp}$. Moreover, for $\chi\in \Phi^{\perp}$, we have $\tilde{P}\chi=0$. It follows that 
\begin{equation}
\chi=\chi-\tilde{P}\chi=\sum_{\hat{\dotlessi}\in \Ih}(v_{\hat{\dotlessi}}-\tilde{P}v_{\hat{\dotlessi}})=\sum_{\hat{\dotlessi}\in \Ih}\chi_{\hat{\dotlessi}},
\label{eqn:decomposition_construction}
\end{equation}
with $\chi_{\hat{\dotlessi}}:=v_{\hat{\dotlessi}}-\tilde{P}v_{\hat{\dotlessi}} \in \Phi^{\perp}_i$, which is the desired decomposition. The next three Lemmas are dedicated to prove that the decomposition is stable in the sense that there exists a constant $K_{min}$, only depending on $d, \gamma,  \kappa_{min}, \kappa_{max}$, such that \eqref{eqn:decomposition} is satisfied.

\begin{lemma}	\label{lemma:predecomposition}
	There exists a constant $C>0$, depending on $d, \gamma,  \kappa_{min}, \kappa_{max}$ only, such that for any $\chi \in \Phi^{\perp}$ the decomposition \eqref{eqn:predecompositon} satisfies 
	\begin{equation}
	\sum_{\hat{\dotlessi}\in \Ih}\|v_{\hat{\dotlessi}}\|^2\leq C \|\chi\|^2 .
	\end{equation}
	See \ref{Sec:lemma:predecomposition} for the proof.
\end{lemma}

\begin{lemma}	\label{lemma:stability}
Let $\tilde{P}$ be defined as \eqref{eqn:tildeP},  there exists a constant $C$, depending on $d, \gamma,  \kappa_{min}, \kappa_{max}$ only, such that for any $\chi\in \Phi^{\perp}$  
\begin{equation}
	\sum_{\hat{\dotlessi}\in \Ih}\|\tilde{P}v_{\hat{\dotlessi}}\|^2\leq C \|\chi\|^2
\end{equation}
with $v_{\hat{\dotlessi}} = \chi \eta_{\hat{\dotlessi}}$, furthermore, there exists a  decomposition $\chi = \sum_{\hat{\dotlessi}\in \Ih} \chi_{\hat{\dotlessi}} $ for any $\chi \in \Phi^{\perp}$ with $\chi_{\hat{\dotlessi}}\in \Phi_{\hat{\dotlessi}}^{\perp}$ such that 
\begin{equation}
	\sum_{\hat{\dotlessi}\in \Ih}\|\chi_{\hat{\dotlessi}}\|^2\leq 2C \|\chi\|^2 .
\end{equation}
See \ref{sec:lemma:stability} for the proof.
\end{lemma}

\begin{lemma}\label{lemma:psi_i^0}
For case V, case E and case D,
it holds true that 
\begin{equation}
\|\psi^0_j\| \leq C_1 H^{-1}  \text{ for any } j\in \I,
\end{equation} 
where $C_1$ only depends on $\gamma$, $\kappa_{max}$ and $d$. See \ref{sec:lemma:psi_i^0} for the proof.
\end{lemma}

\begin{theorem}
\label{thm:decay}
	 It holds that 
	\begin{equation}
	\|\psi_i-\psi_i^{\ell}\| \leq C_1 H^{-1} e^{-{\ell}/C_2}.
	\end{equation}
\end{theorem}
\begin{proof}
	Lemma \ref{lemma:cond(P)} and \ref{lemma:stability} 
	imply that there exists a constant $C'$ depending on  $\gamma, d, \kappa_{max}, \kappa_{min}$ only, such that $1/C' \leq K_{min} \leq \lambda_{min}(P)$. 
	According to Lemma \ref{lemma:cond(P)}, $\mathrm{cond}(P)$ has an upper bound $C_2$  depending on $\gamma, d, \kappa_{max}, \kappa_{min}$ only. Combining Lemma \ref{prop:iteration} and Lemma \ref{lemma:psi_i^0}, we draw the conclusion.
\end{proof}

\subsection{Accuracy of Localized Basis}
\label{sec:analysis:localbasis}

Due to the exponential decay of the truncation error, we can use the localized basis $\psi_i^{\ell}$ instead of the global basis $\psi_i$ to reduce computational cost. The following theorem shows that we can preserve the $O(H)$ convergence rate for the global basis in Theorem \ref{thm:convergense of global} if the localization level (number of layers in the localization patch) ${\ell} \simeq O(\log(1/H))$.
\begin{theorem}
\label{thm:local}
	Let $u_H^{\ell}\in \Psi^{\ell} $ be the solution to \eqref{eqn:local_varellpFEM}. For ${\ell}\geq C_2\log(1/H)$ we have 
	\begin{equation}
	\|u-u_H^{\ell}\| \leq CH \|g\|_{L^2(\Omega)},
	\end{equation}
	where C depends on $\kappa_{min},\kappa_{max},d,\Omega$, and $\gamma$.
\end{theorem}

\begin{proof}
	Let $u_{\psi}^{\ell}=\Sigma_{i=1}^N m_i \psi_i^{\ell}$, where $m_i=[\phi_i,u],\ for \ i=1,...,N$. Recalling that $u_H = \Sigma_{i=1}^N m_i \psi_i$, we have 
	\begin{equation}
	\|u-u_{\psi}^{\ell}\| \leq \|u-u_H\|+\|u_H-u_{\psi}^{\ell}\|.
	\end{equation}
	Theorem \ref{thm:convergense of global} implies that$\|u-u_H\|\leq CH\|g\|_{L^2}$. To derive an estimate of the second term, let $\chi=u_H-u_{\psi}^{\ell}$. As $\chi \in \Phi^{\perp}$, it can be decomposed as 
	\begin{equation}
		\chi=\sum_{\hat{\dotlessi}\in \Ih}\chi_{\hat{\dotlessi}},
	\end{equation}
	with $\chi_{\hat{\dotlessi}}:=v_{\hat{\dotlessi}}-\tilde{P}v_{\hat{\dotlessi}} \in \Phi^{\perp}_{\hat{\dotlessi}}$ as in \eqref{eqn:decomposition_construction}. By Lemma \ref{lemma:stability}, it follows that 
	\begin{equation}
	\sum_{\hat{\dotlessi}\in \Ih}\|\chi_{\hat{\dotlessi}}\|^2\leq C \|\chi\|^2,
	\label{eqn:stable_decom}
	\end{equation}
	where $C$ depends on $d, \gamma,  \kappa_{min}, \kappa_{max}$ only.
	Note that 
	\begin{equation}
	 \|\chi\|^2 = \left(\sum_{\hat{\dotlessi}\in \Ih}\chi_{\hat{\dotlessi}},  \sum_{i\in\I} m_i (\psi_i-\psi_i^{\ell}) \right).
	 \label{eqn:product}
	\end{equation}
	$\chi_{\hat{\dotlessi}} \in \Phi^{\perp}_{\hat{\dotlessi}} $ implies $(\chi_{\hat{\dotlessi}}, \psi_i)=0$ and $(\chi_{\hat{\dotlessi}}, \psi_i^{\ell})=0$ for $\omega_{\hat{\dotlessi}} \subset  \Omega_i^{\ell}$. Therefore,	for any pair $(\hat{\dotlessi},i)\in \{(\hat{\dotlessi},i) \in \Ih \times \I| \omega_{\hat{\dotlessi}} \subset  \Omega_i^{\ell} \text{ or } \omega_{\hat{\dotlessi}} \cap  \Omega_i^{\ell}=\emptyset\} $, it holds that
	$$
	(\chi_{\hat{\dotlessi}},\psi_i-\psi_i^{\ell})=0.
	$$ 
     For any nonzero term, by Young's inequality, we have 
	\begin{equation}	
	m_i(\chi_{\hat{\dotlessi}},\psi_i-\psi_i^{\ell}) \leq \frac{1}{2}((C{C_{ol}})^{-1}\|\chi_{\hat{\dotlessi}}\|^2+CC_{ol} m_i^2\|\psi_i-\psi_i^{\ell}\|^2)
	\label{eqn:nonzero}
	\end{equation}
	We note that the number of nonzero terms in \eqref{eqn:product}, for each $\chi_{\hat{\dotlessi}}$ or $\psi_i-\psi_i^{\ell}$, is bounded by a constant $C_{ol}\sim {\ell}^{(d-1)}$ depending on the shape regularity. Combining \eqref{eqn:stable_decom}, \eqref{eqn:product} and \eqref{eqn:nonzero}, it follows that
	\begin{equation}
	\begin{aligned}
	\|\chi\|^2 = &\sum_{\omega_{\hat{\dotlessi}} \not\subset  \Omega_i^{\ell}\text{ and }\omega_{\hat{\dotlessi}} \cap  \Omega_i^{\ell} \neq \emptyset} m_i(\chi_{\hat{\dotlessi}},\psi_i-\psi_i^{\ell})\\
	\leq & \sum_{\omega_{\hat{\dotlessi}} \not\subset  \Omega_i^{\ell}\text{ and }\omega_{\hat{\dotlessi}} \cap  \Omega_i^{\ell} \neq \emptyset} \frac{1}{2}((C{C_{ol}})^{-1}\|\chi_{\hat{\dotlessi}}\|^2+CC_{ol} m_i^2\|\psi_i-\psi_i^{\ell}\|^2) \\
	\leq & \sum_{\hat{\dotlessi}\in \Ih }C_{ol} \frac{1}{2}((C{C_{ol}})^{-1}\|\chi_{\hat{\dotlessi}}\|^2)+ \sum_{i\in \I }C_{ol}(CC_{ol} m_i^2\|\psi_i-\psi_i^{\ell}\|^2)\\
	\leq & \frac{1}{2}\|\chi\|^2 + \sum_{i\in \I}CC_{ol}^2  m_i^2\|\psi_i-\psi_i^{\ell}\|^2.
	\end{aligned}
	\end{equation}
	Hence 
	\begin{equation}
	    \|\chi\|^2 \leq 2\sum_{i\in \I}CC_{ol}^2  m_i^2\|\psi_i-\psi_i^{\ell}\|^2.
	    	\label{eqn:temp1}
	\end{equation}

	By Lemma \ref{lem:trace} and the Poincar\'{e} inequality, it suffices to show 
	\begin{equation}
	\sum_{i \in \I} m_{i}^2 =\sum_{i\in \I} [\phi_i,u]^2\leq C \|u\|^2,
	\label{eqn: coef stable}
	\end{equation}
	for different cases. We have
	
	\begin{itemize}
	    \item Case V:  $\sum_{i\in \I}  [\phi_i,u]^2 \leq \sum_{i\in \I}  \|\phi_i\|_{L^2}^2 \|u\|_{L^2(\tau_i)}^2 \leq \sum_{i\in \I}  \|u\|_{L^2(\tau_i)}^2 = \|u\|^2$.
	    \item Case E: $\sum_{i\in \I}  [\phi_i,u] ^2 = \sum_{i\in \I}(\int_{e_i}|e_i|^{\frac{2-d}{2(d-1)}} u \ds)^2 
	\leq \sum_{i\in \I} CH \|u\|_{L^2(e_i)}^2
	\leq C \sum_{i\in \I}  \|u\|_{L^2(\tau_i)}^2 \leq C \|u\|^2$, where $C$ depends on $d$, $\Omega$, and $\gamma$ only.
	    \item Case D: a combination of case V and case E.
	\end{itemize}

	Plugging \eqref{eqn: coef stable} into \eqref{eqn:temp1} and applying Theorem \ref{thm:decay}, we have
	\begin{equation}
	\|\chi\|^2 \leq 2  CC_{ol}^2 e^{-{\ell}/C_2}H^{-1}\|g\|_{L^2}^2.
	\end{equation}
	For ${\ell}\geq C_2\log(1/H)$, we have $\|\chi\| \leq CH \|g\|_{L^2(\Omega)}$, where $C$ only depends on $\kappa_{min},\kappa_{max},d,\Omega$, and $\gamma$.
\end{proof}

\section{Numerics}
\label{sec:numerics}

In this section, we justify our theoretical results through a few examples. We demonstrate the localization property of GRPS basis and the convergence of localized GRPS for benchmark problems with multiscale coefficients. Furthermore, we validate the GRPS method for wave equations in heterogeneous media. 




\subsection{Multiscale Trigonometric Example} 
\label{sec:numerics:mstrig}

The multiscale trigonometric (mstrig) coefficient $\kappa(x_1,x_2)$ is given by,

\begin{multline}
\kappa(x_1,x_2):=\frac{1}{6}(\frac{1.1+\sin(2\pi x_1/\epsilon_1)}{1.1+\sin(2\pi x_2/\epsilon_1)}+\frac{1.1+\sin(2\pi x_2/\epsilon_2)}{1.1+\cos(2\pi x_1 /\epsilon_2)}
+\frac{1.1+\cos(2\pi x_2/\epsilon_3)}{1.1+\sin(2\pi x_1 /\epsilon_3)}\\
+\frac{1.1+\sin(2\pi x_2/\epsilon_4)}{1.1+\cos(2\pi x_1 /\epsilon_4)}
+\frac{1.1+\cos(2\pi x_1/\epsilon_5)}{1.1+\sin(2 \pi x_2/\epsilon_5)}+\sin(4x_1^2x_2^2)+1)
\label{eqn:benchmark1}
\end{multline}

where $\epsilon_1=1/5,\, \epsilon_2=1/13,\, \epsilon_3=1/17,\, \epsilon_4=1/31,\,\epsilon_5=1/65$. $\kappa(x_1,x_2)$ is highly oscillatory with non-separable scales on the unit square $\Omega=[0,1]\times [0,1]$. Figure \ref{fig:MTE} illustrates $\kappa(x_1,x_2)$. $g(x_1,x_2) = \sin(x_1)$.
\begin{figure}[H]
	\centering
	{\includegraphics[width=0.5\textwidth]{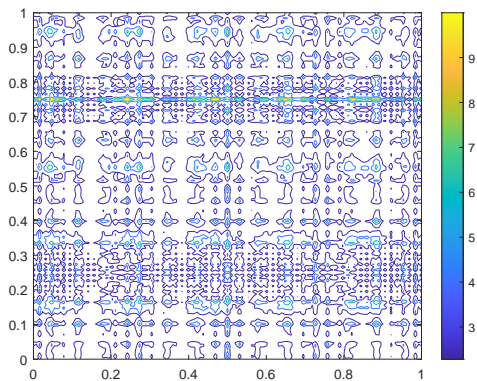}}
	\caption{Multiscale Trigonometric Example,  $\mathrm{contrast}\approx 33.4$}
	\label{fig:MTE}
\end{figure}

\subsubsection{Localization of GRPS basis}

We first show the exponential decay of GRPS basis function. For a fixed coarse mesh with $N_c\times N_c$ nodes, the degrees of freedom for the RPS, GRPS-V, GRPS-E and GRPS-D bases are $(N_c-1)^2, 2N_c^2, 3N_c^2-2N_c$ and $5N_c^2-2N_c$, respectively. These numbers indicate that for the same mesh, the density (the degree of freedoms) of GRPS-D is the largest, and then GRPS-E, GRPS-V and RPS.

We illustrate RPS, GRPS-E, GRPS-V, GRPS-D basis functions in Fig. \ref{fig:PC} to Fig. \ref{fig:DS}. It seems that GRPS-E and GRPS-D bases decay more rapidly compared with GRPS-V and RPS bases. The GRPS-E and GRPS-D bases seem to be more spiky than the GRPS-V basis, and all those three bases are more localized than the RPS basis. 

\begin{figure}[http]
	\centering
	\subfigure[RPS, contour\label{fig:PC}] {\includegraphics[width=0.23\textwidth]{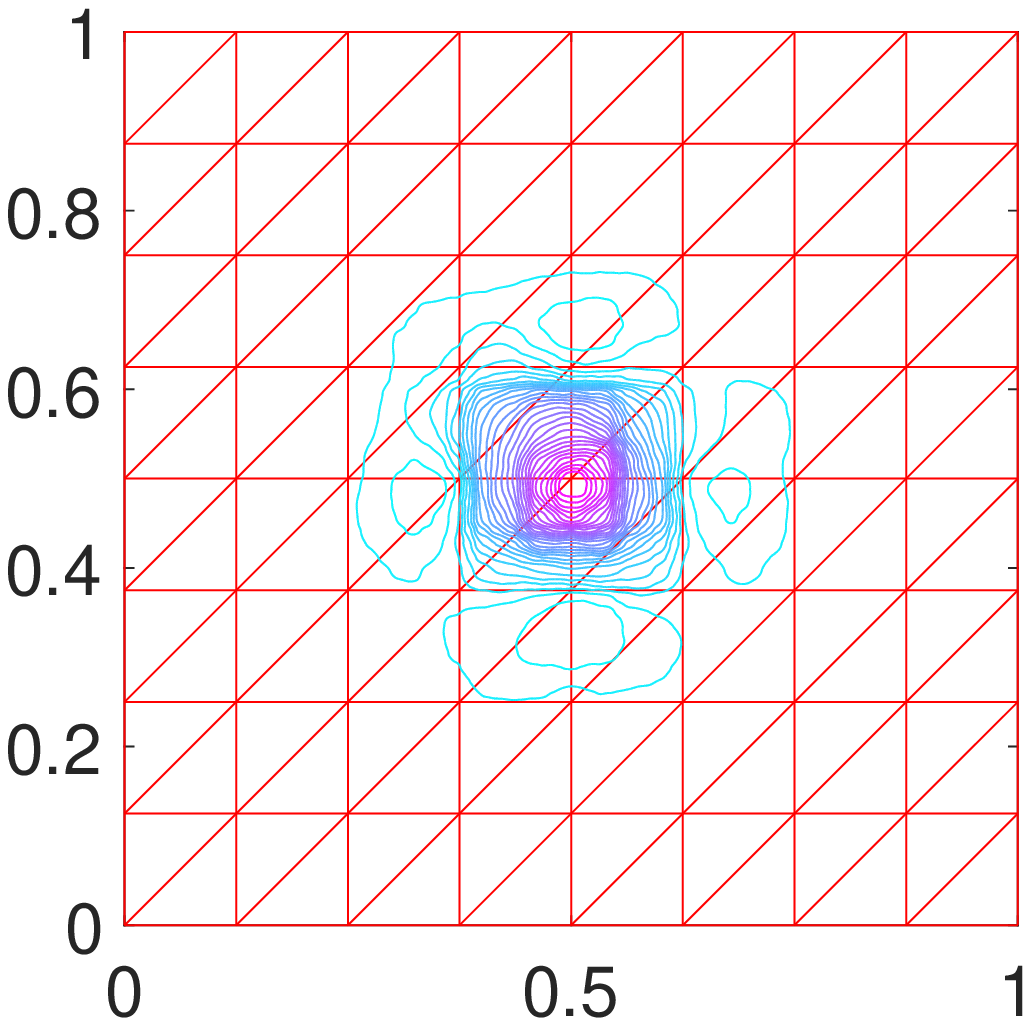}} \quad
	\subfigure[GRPS-V, contour\label{fig:VC}]{\includegraphics[width=0.23\textwidth]{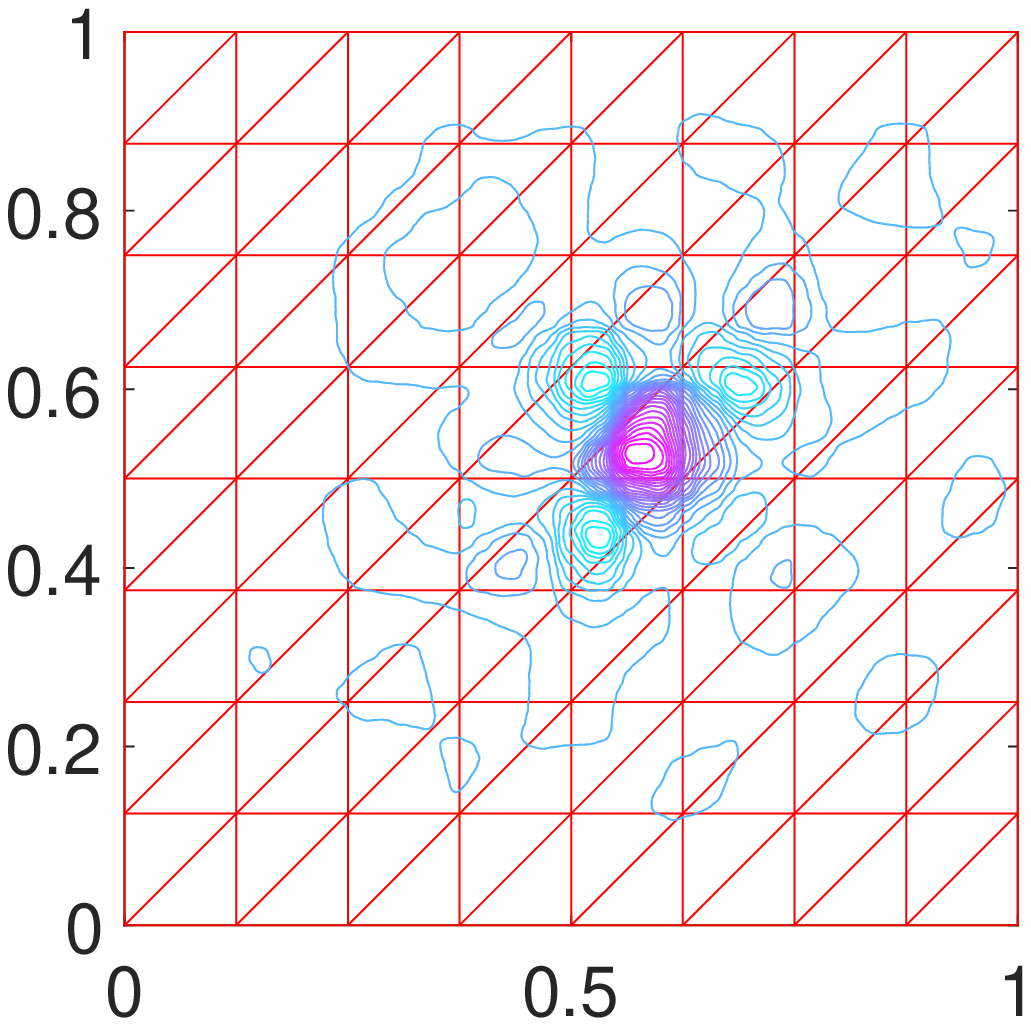}}
	\subfigure[GRPS-E, contour\label{fig:EC}] {\includegraphics[width=0.23\textwidth]{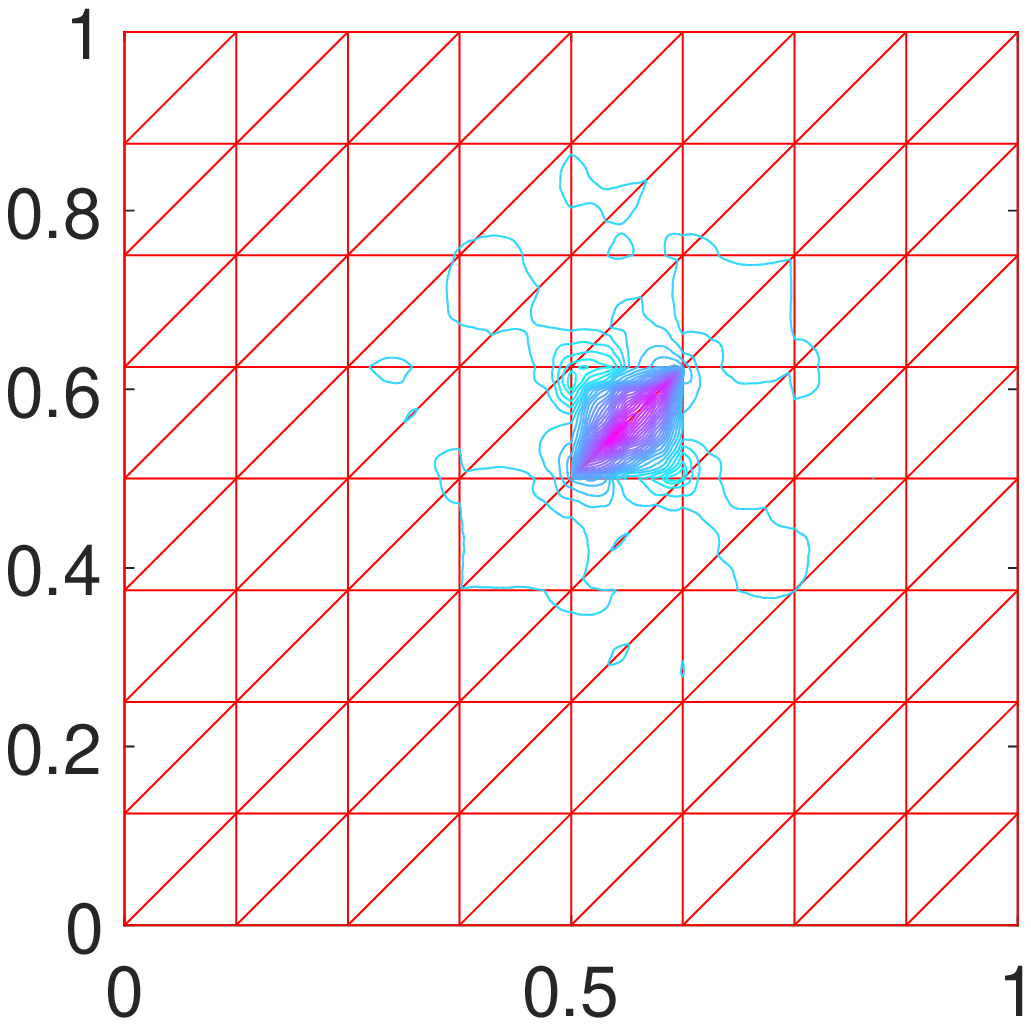}}\quad
	\subfigure[GRPS-D, contour\label{fig:DC}]{\includegraphics[width=0.23\textwidth]{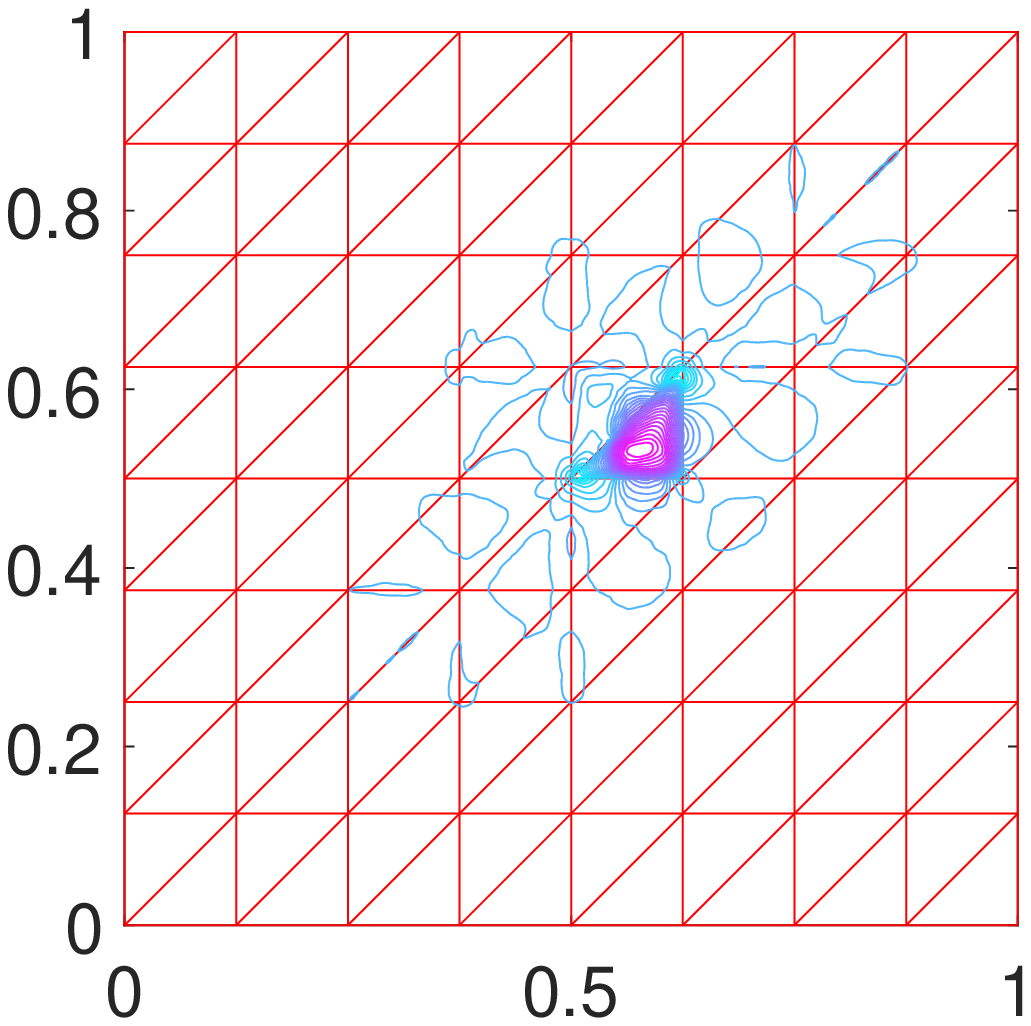}}
	
	\subfigure[RPS, surface\label{fig:PS}] {\includegraphics[width=0.230\textwidth]{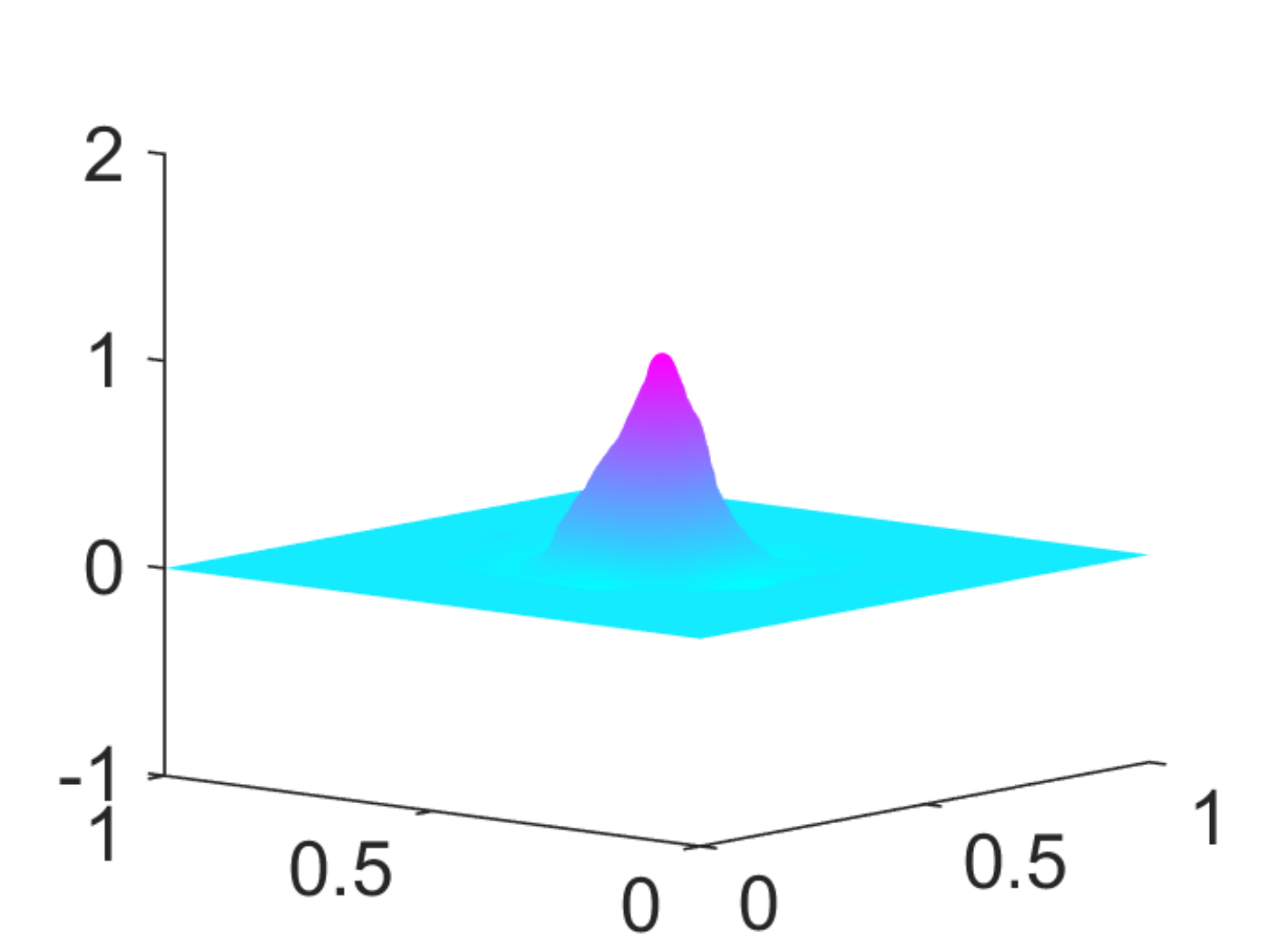}}\quad
	\subfigure[GRPS-V, surface\label{fig:VS}]{\includegraphics[width=0.230\textwidth]{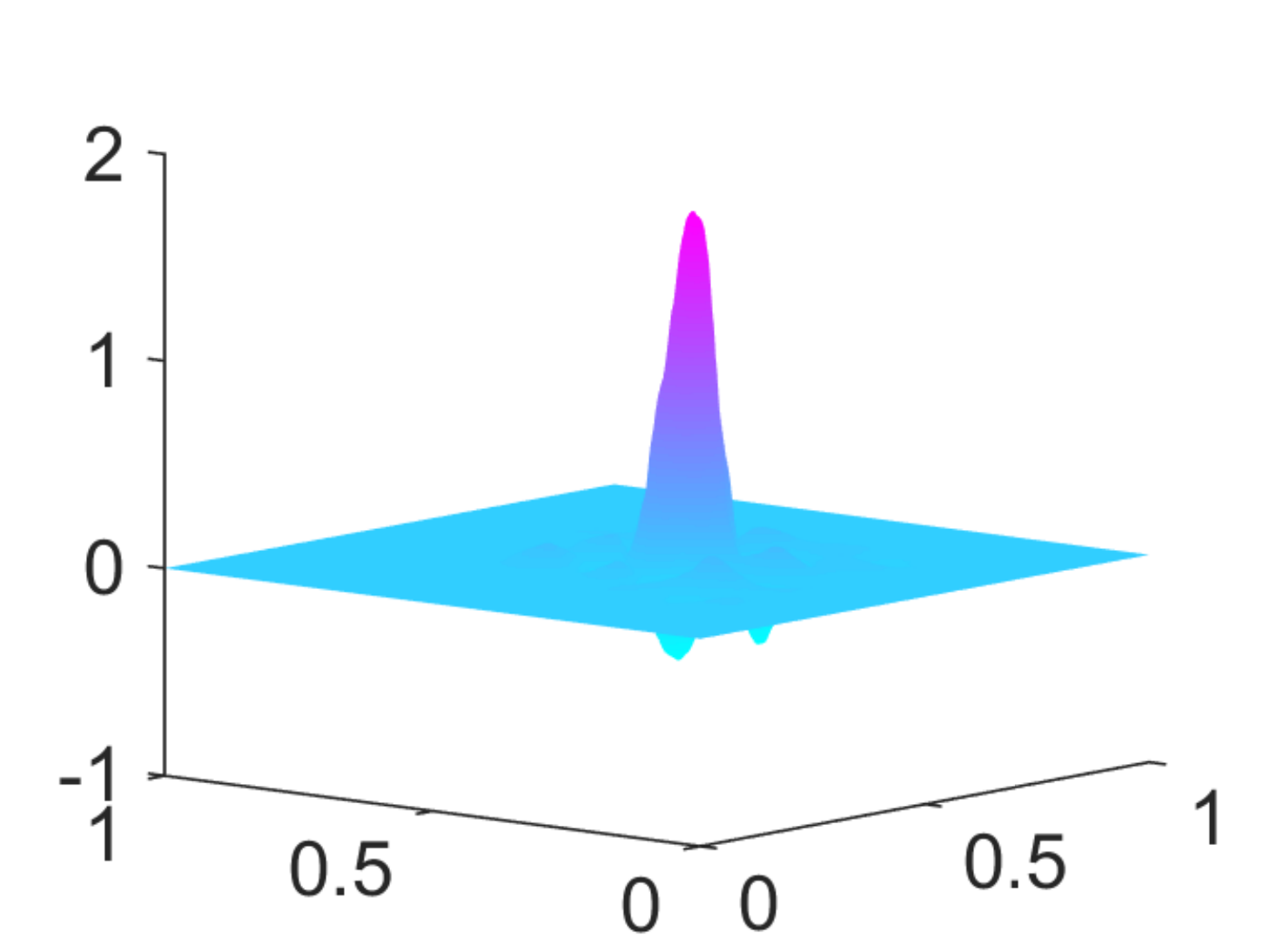}}
	\subfigure[GRPS-E, surface\label{fig:ES}] {\includegraphics[width=0.230\textwidth]{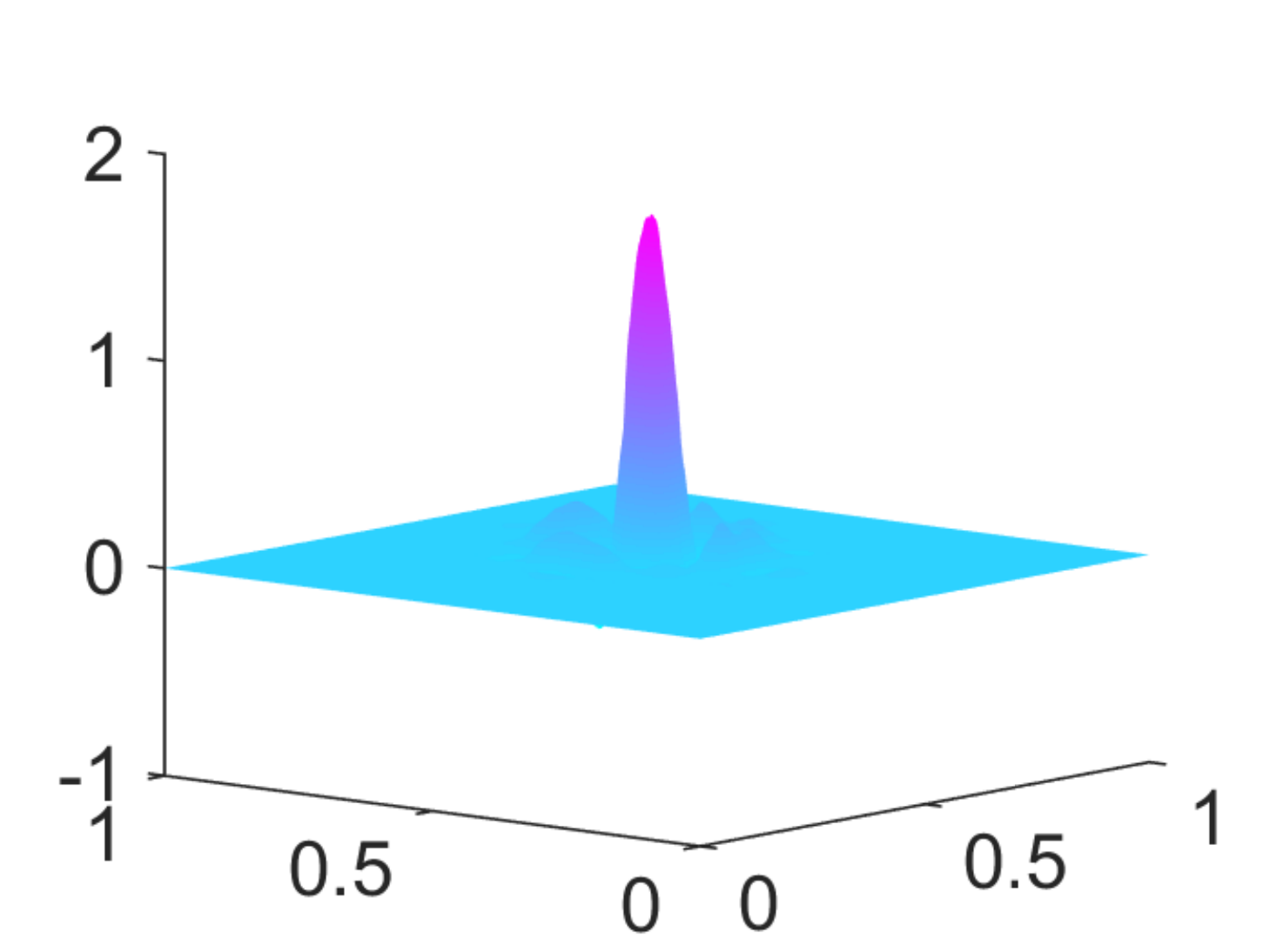}}\quad
	\subfigure[GRPS-D, surface\label{fig:DS}]{\includegraphics[width=0.230\textwidth]{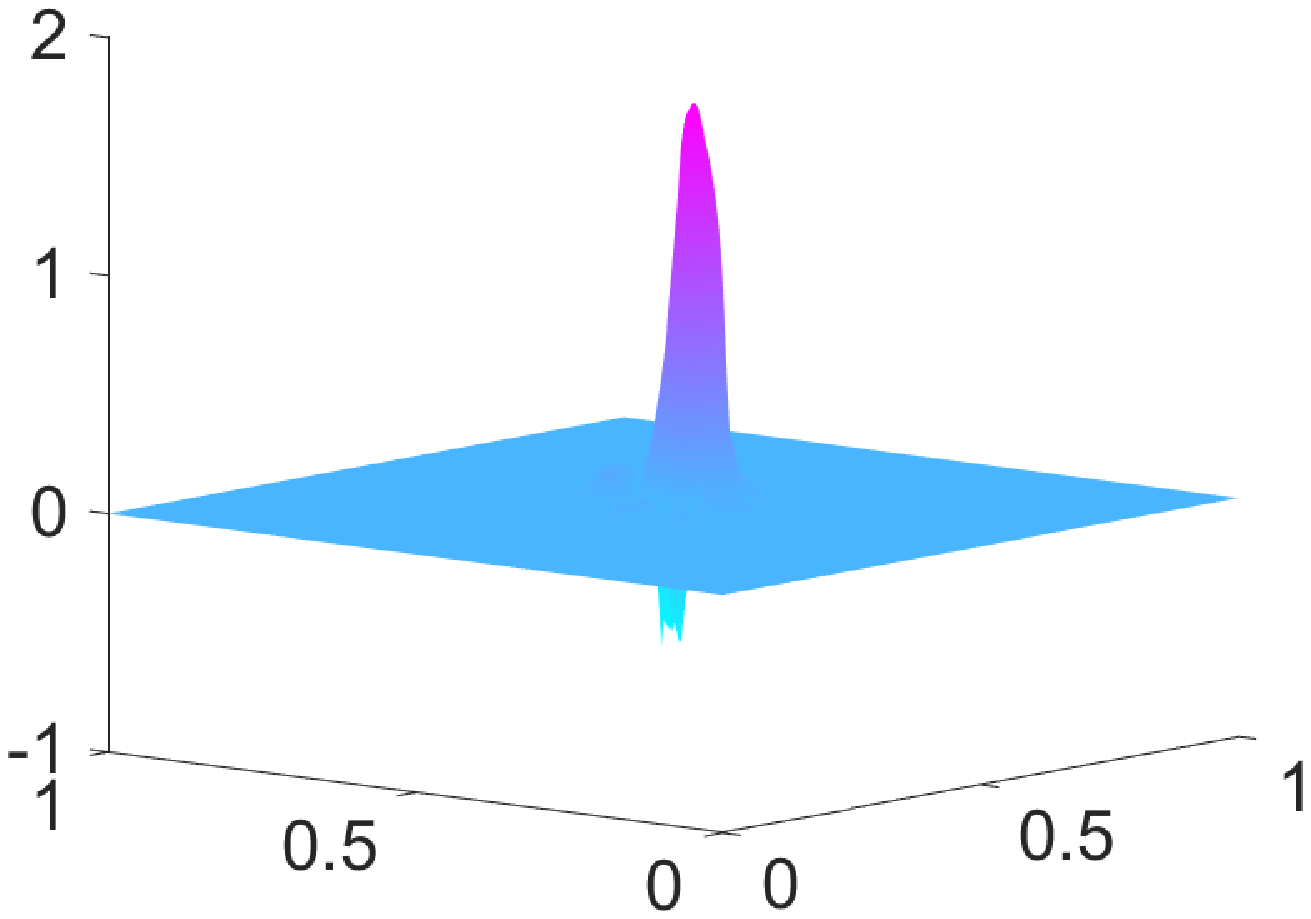}}
	
	\subfigure[RPS, $x_2=x_1$\label{fig:Pslice}] {\includegraphics[width=0.230\textwidth]{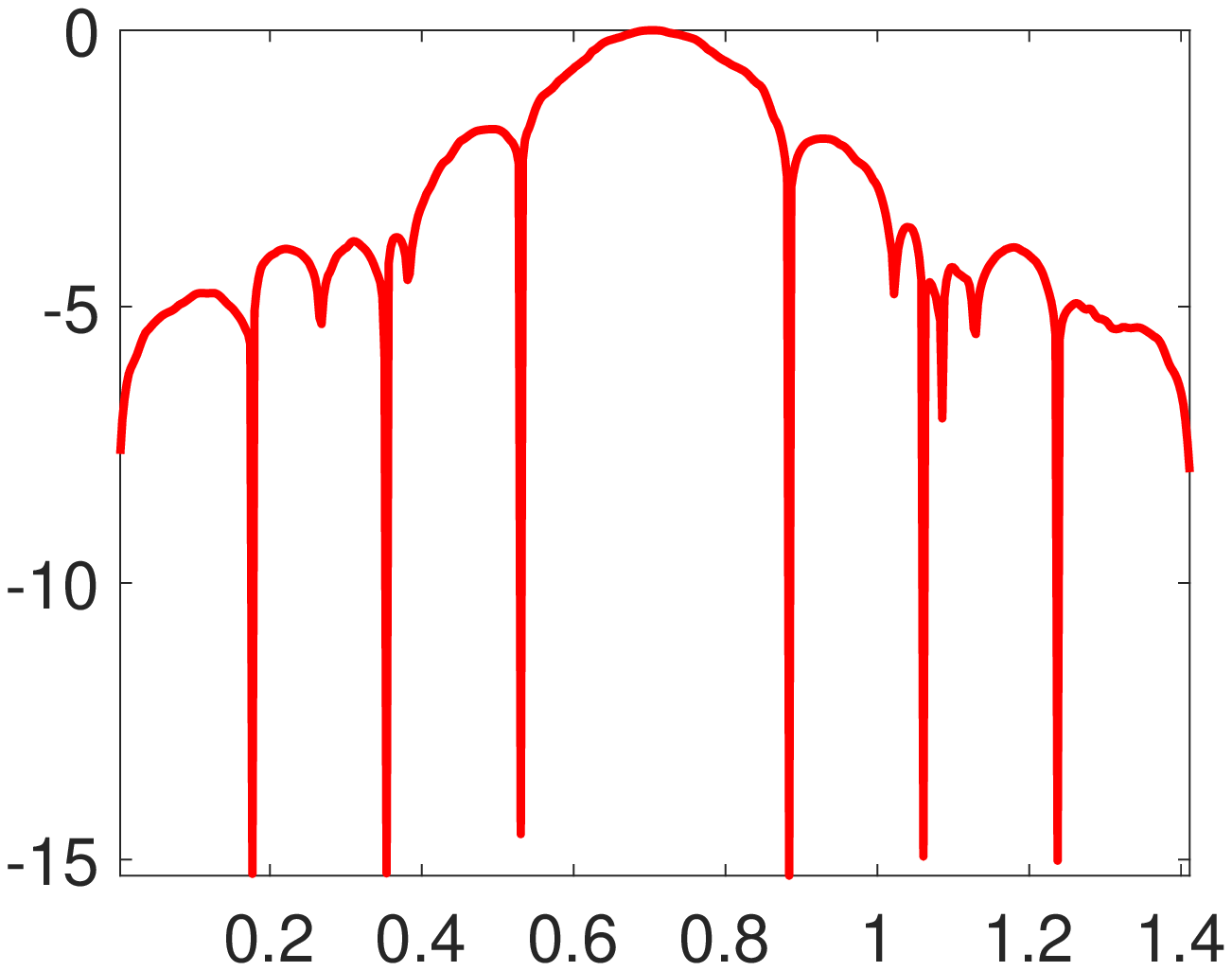}}\quad
	\subfigure[GRPS-V, $x_2=x_1$ \label{fig:Vslice}]{\includegraphics[width=0.230\textwidth]{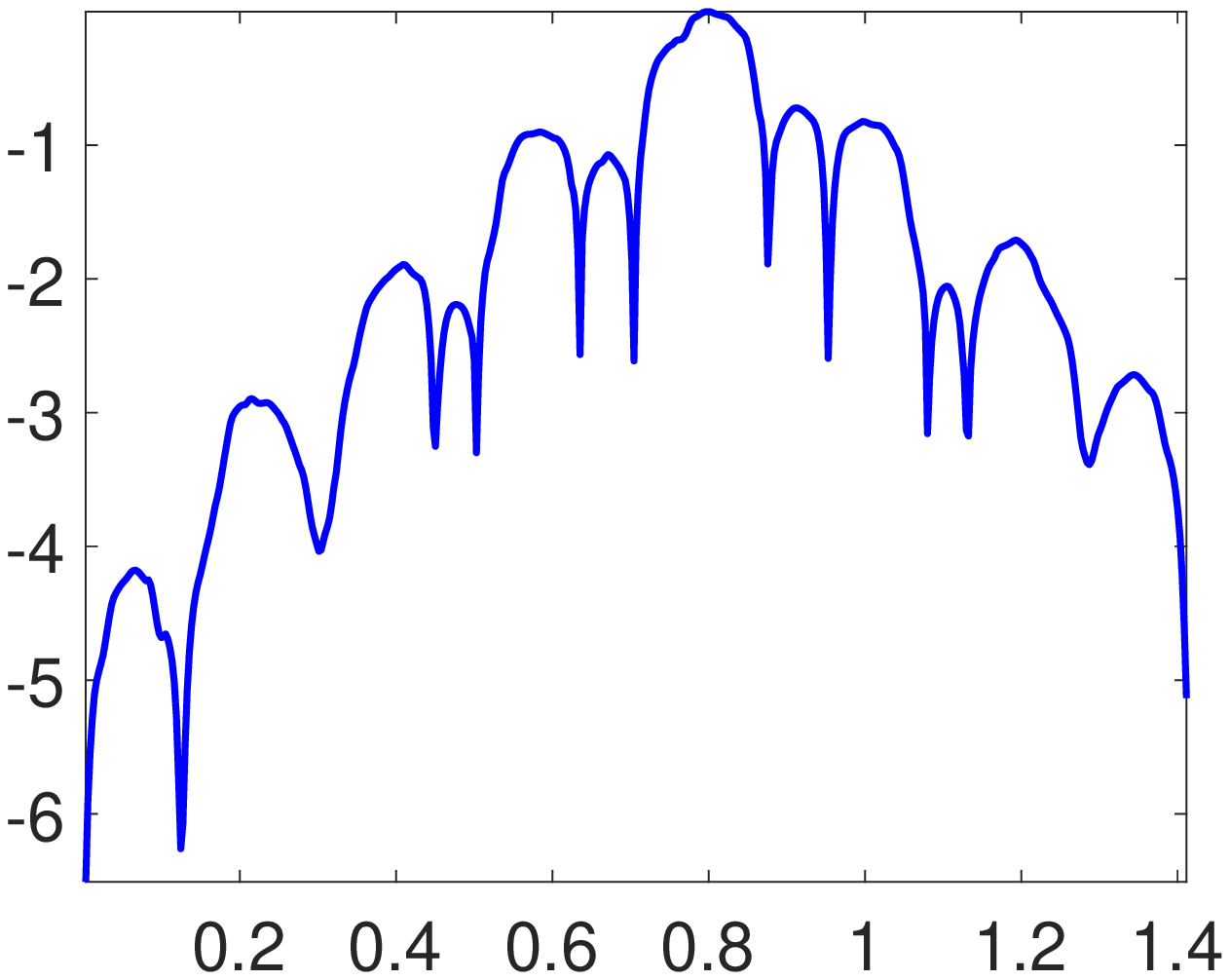}}
	\subfigure[GRPS-E, $x_2=x_1$\label{fig:Eslice}] {\includegraphics[width=0.230\textwidth]{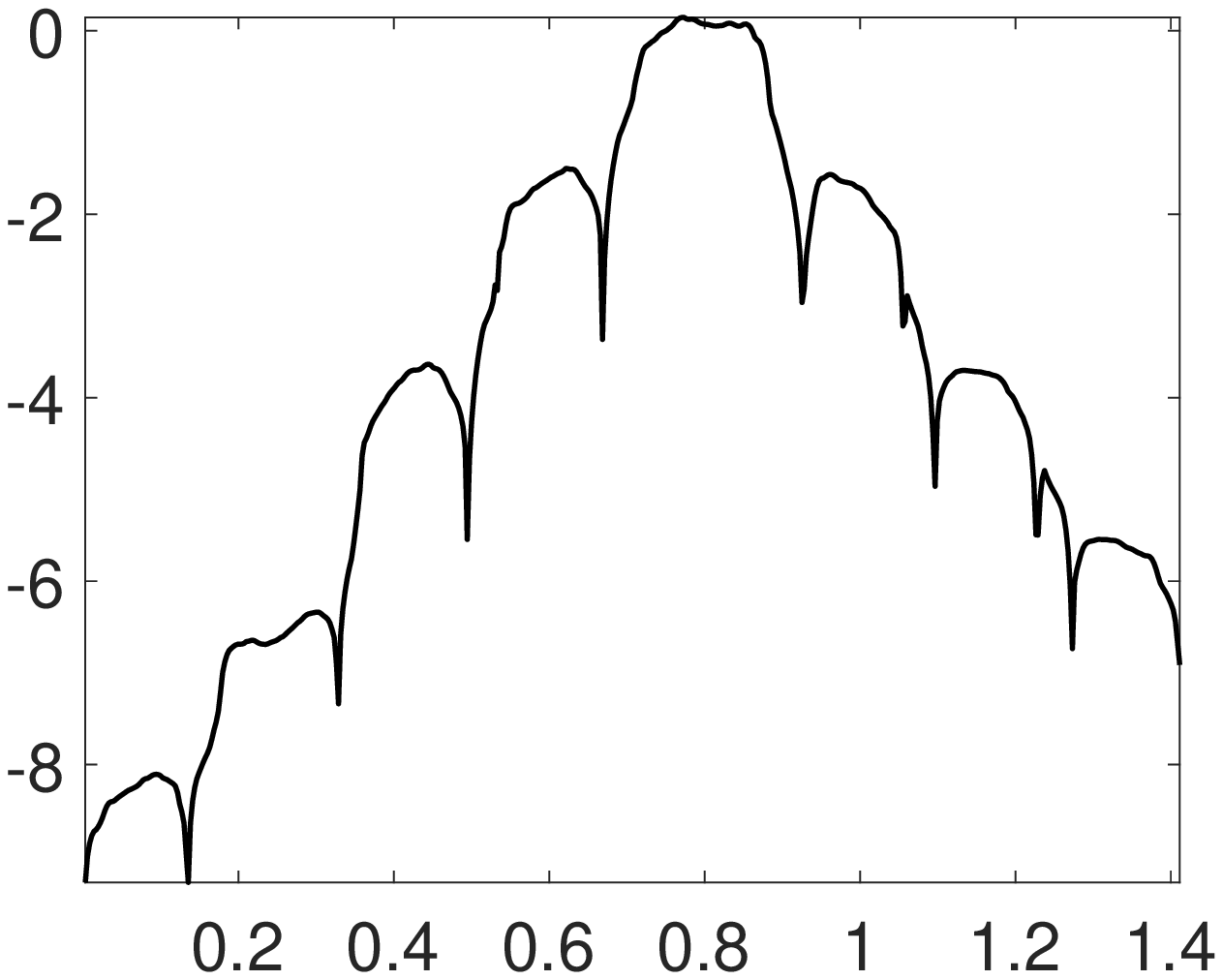}}\quad
	\subfigure[GRPS-D, $x_2=x_1$ \label{fig:Dslice}]{\includegraphics[width=0.230\textwidth]{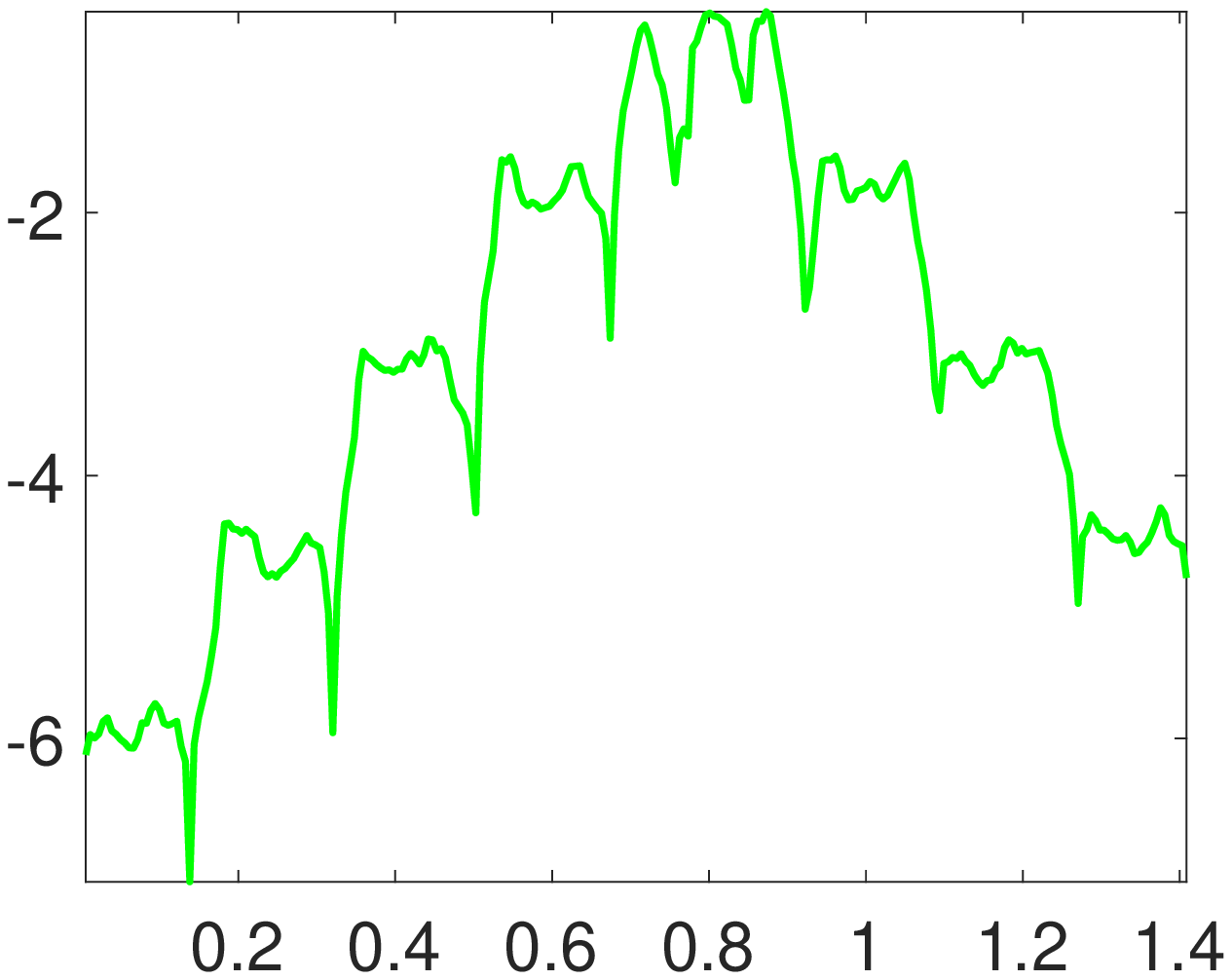}}
	
	\subfigure[RPS, $x_1+x_2=\frac{9}{8}$\label{fig:Pslice2}] {\includegraphics[width=0.230\textwidth]{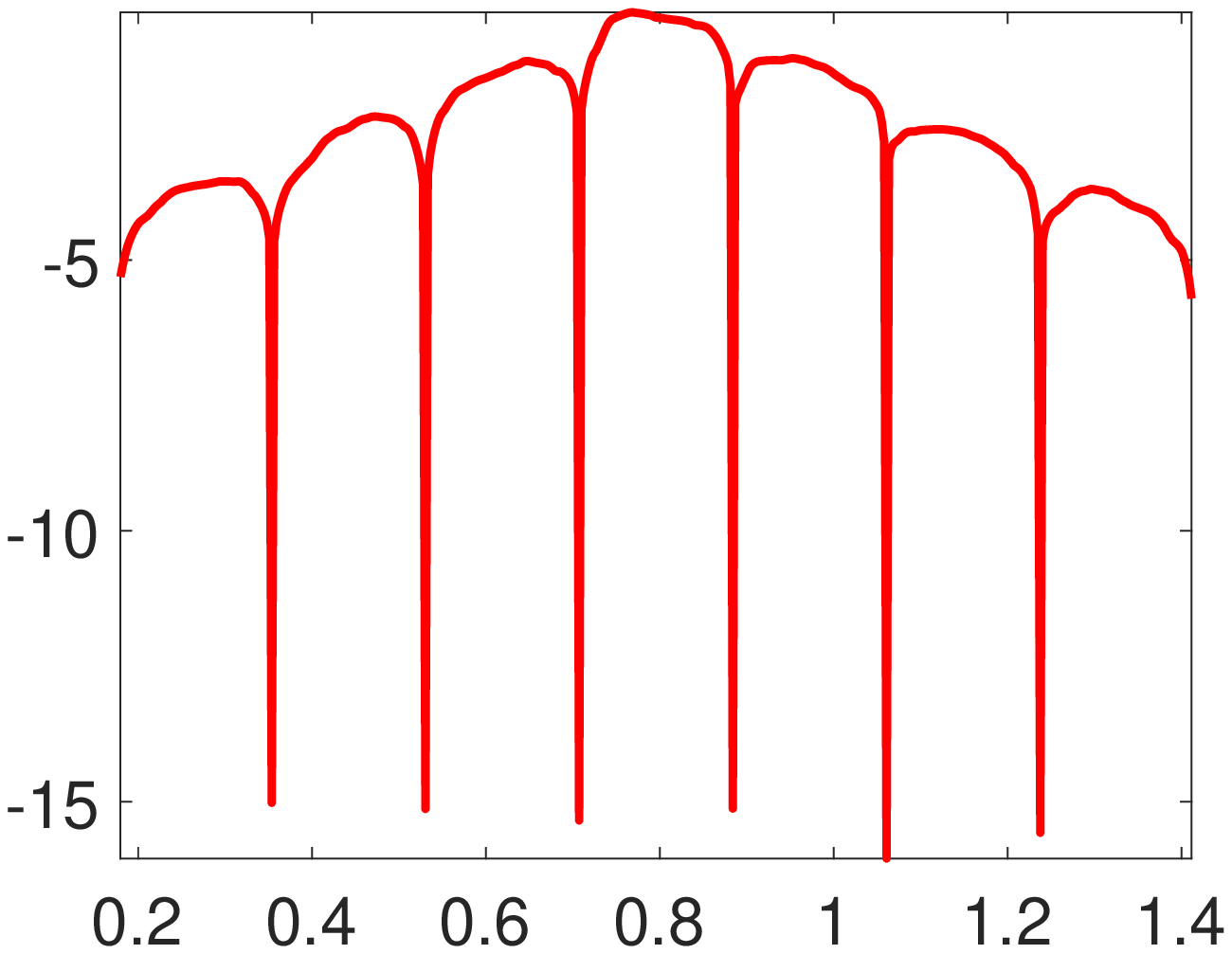}}\quad
	\subfigure[GRPS-V, $x_1+x_2=\frac{9}{8}$ \label{fig:Vslice2}]{\includegraphics[width=0.230\textwidth]{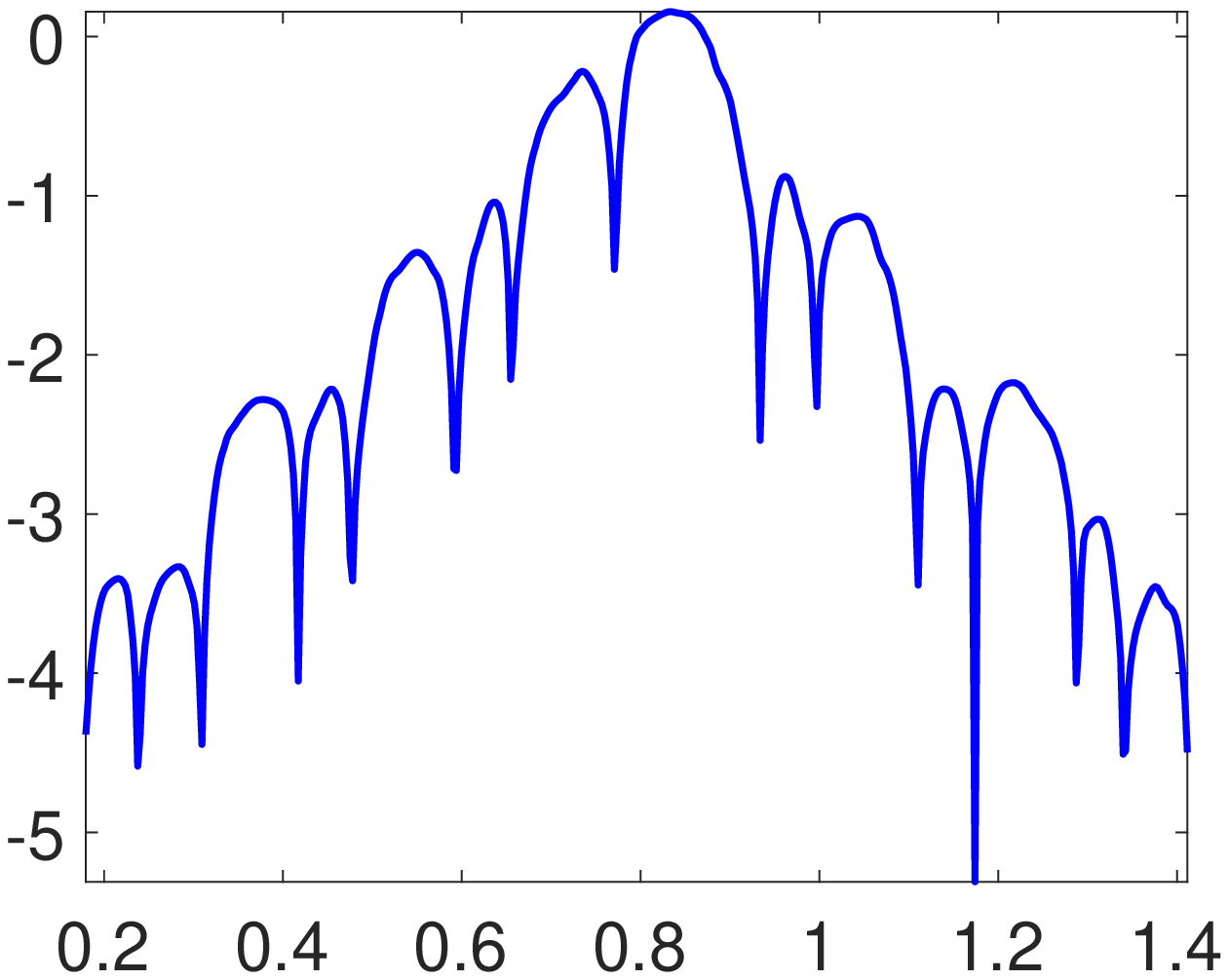}}
	\subfigure[GRPS-E, $x_1+x_2=\frac{9}{8}$\label{fig:Eslice2}] {\includegraphics[width=0.230\textwidth]{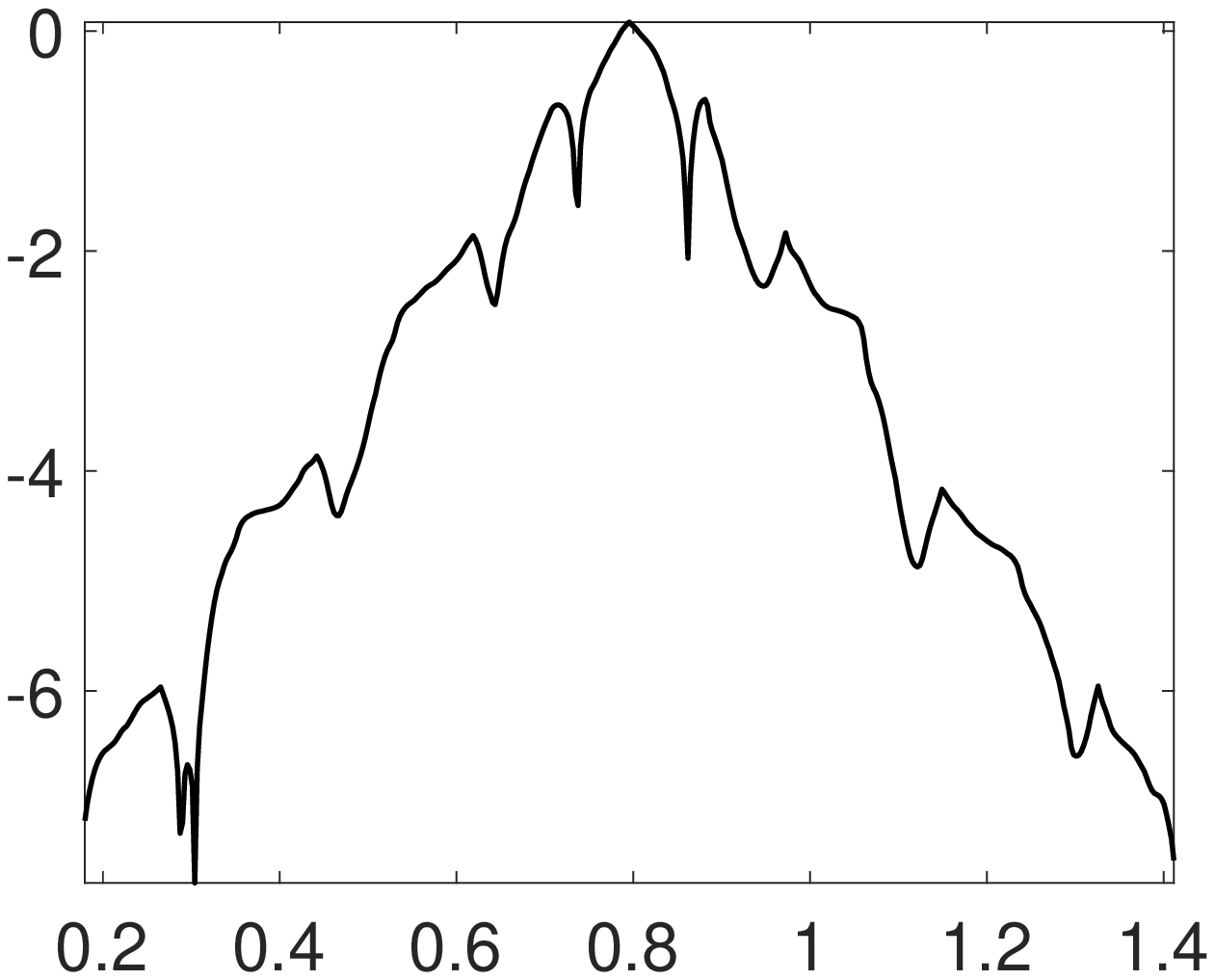}}\quad
	\subfigure[GRPS-D, $x_1+x_2=\frac{9}{8}$ \label{fig:Dslice2}]{\includegraphics[width=0.230\textwidth]{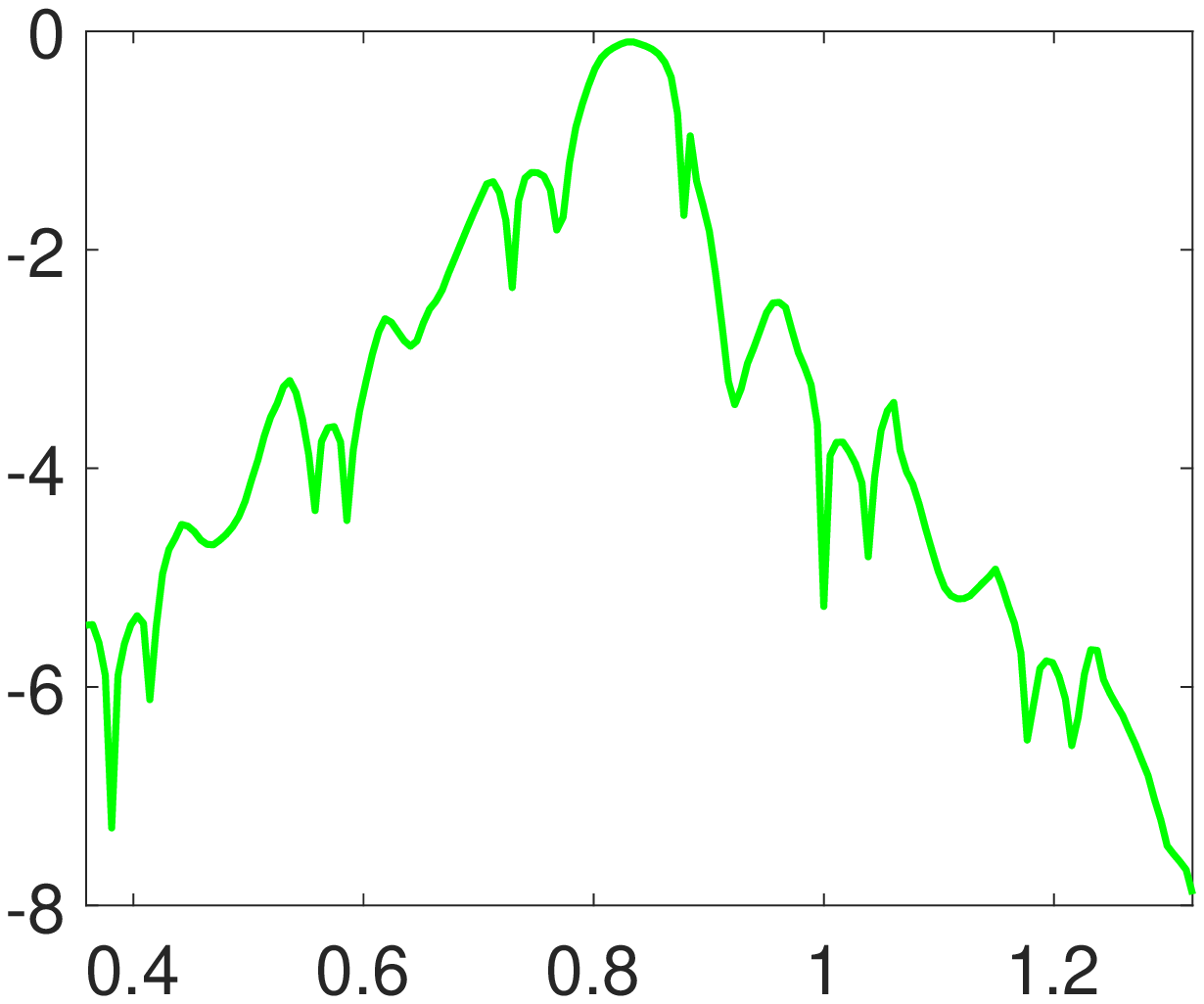}}
	\caption{Contour plots (a)-(d), surface plots (e)-(h), 1d slice plots (in the $\log_{10}$-scale) along $x_2=x_1$ (i)-(l), and $x_1+x_2=9/8$ (m)-(p), for RPS, GRPS-V, GRPS-E, GRPS-D basis functions (\textit{mstrig} exmaple). For the coarse mesh $N_c=8$, and for the fine mesh $N_f=512$.}
\end{figure}

\subsubsection{Convergence}
We investigate the convergence property of GRPS bases and RPS bases for the \textit{mstrig} example. We use a fixed fine mesh with $h=2^{-8}$ and coarse meshes with sizes $H=2^{-3},\,2^{-4},\,2^{-5}$, respectively. In Figure \ref{fig:comp_phi}, we compare convergence curves of four different basis functions, with localization levels $\ell=2,\,3,\,4,\,5,\,6$. The x-axis stands for the degree of freedom of the basis functions, and the y-axis stands for the error $\|u_{H}^{\ell}-u_h\|_{H_0^1}$ in $\log_{10}$-scale, where $u_h$ is the finite element reference solution to \eqref{eqn:varellp} over $V_h$. Each curve stands for the convergence rate vs. degrees of freedom plot with respect to a fixed $\ell$.
\begin{figure}[http]
	\centering
	\subfigure[RPS, $H=2^{-3}$ \label{fig:Pc1}]
	{\includegraphics[width=0.24\textwidth]{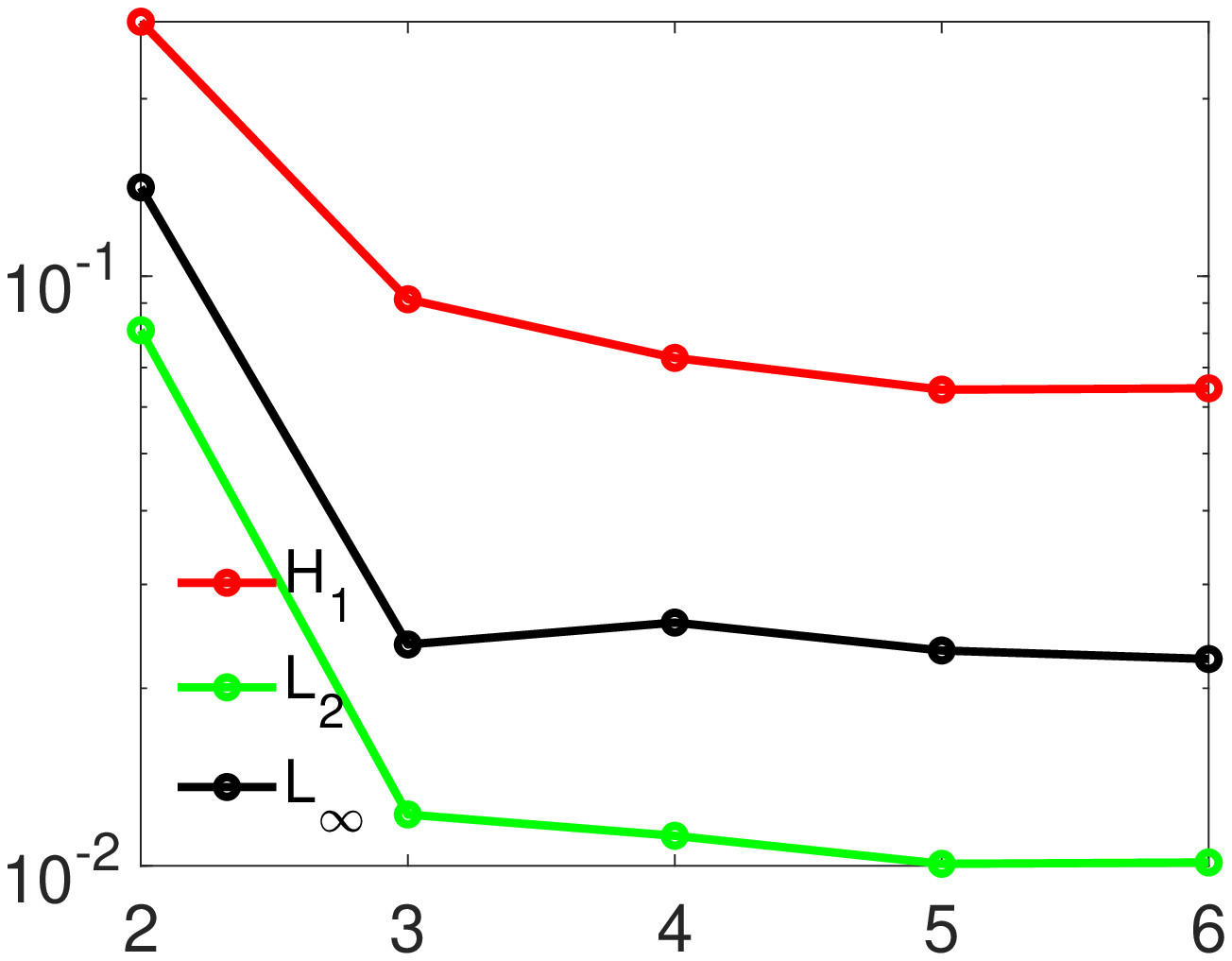}}
	\quad
	\subfigure[RPS, $H=2^{-4}$ \label{fig:Pc2}]
	{\includegraphics[width=0.24\textwidth]{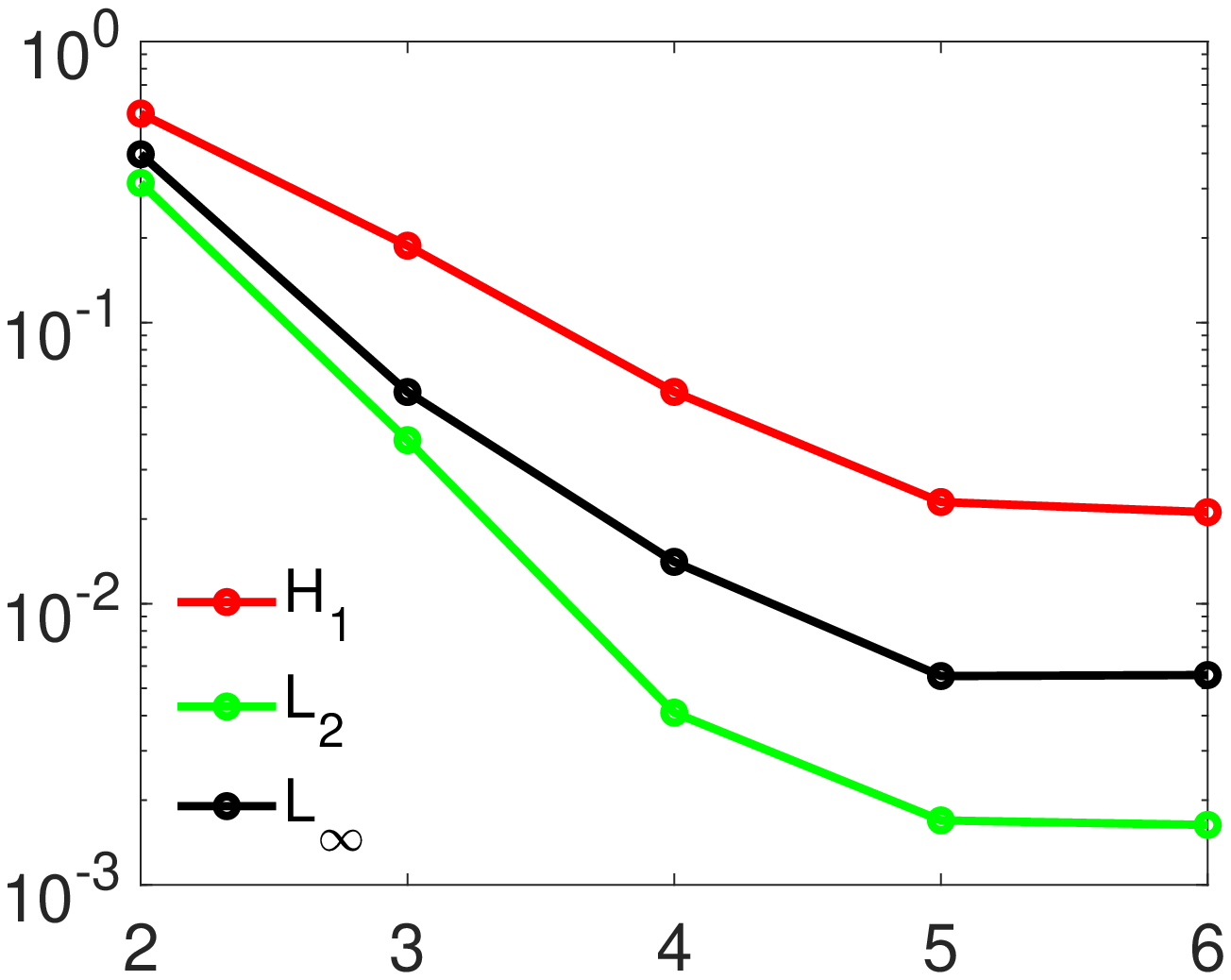}}
	\quad
	\subfigure[RPS, $H=2^{-5}$ \label{fig:Pc3}]
	{\includegraphics[width=0.24\textwidth]{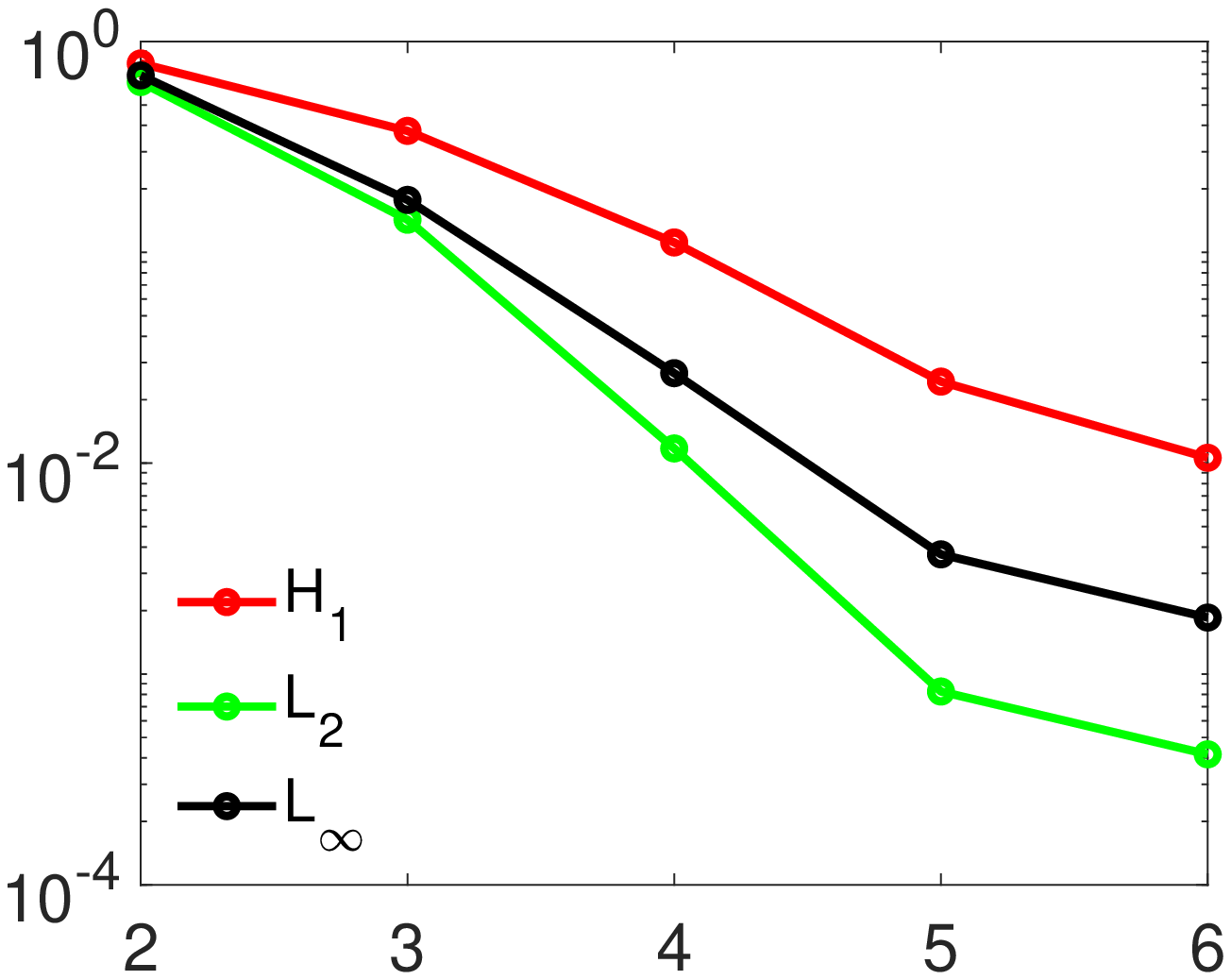}}
	
	\subfigure[GRPS-V, $H=2^{-3}$ \label{fig:PT1}]
	{\includegraphics[width=0.24\textwidth]{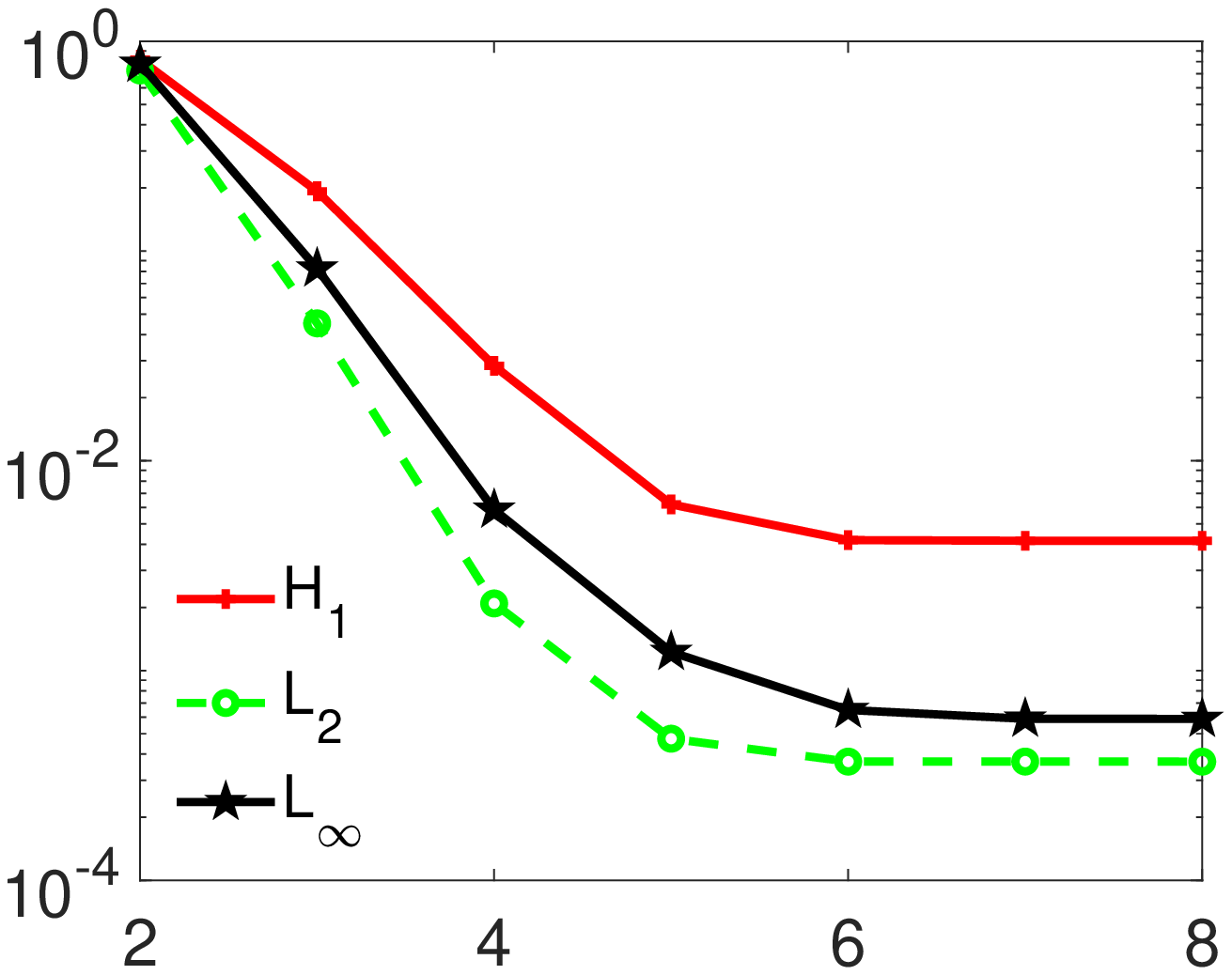}}
	\quad
	\subfigure[GRPS-V, $H=2^{-4}$ \label{fig:PT2}]
	{\includegraphics[width=0.24\textwidth]{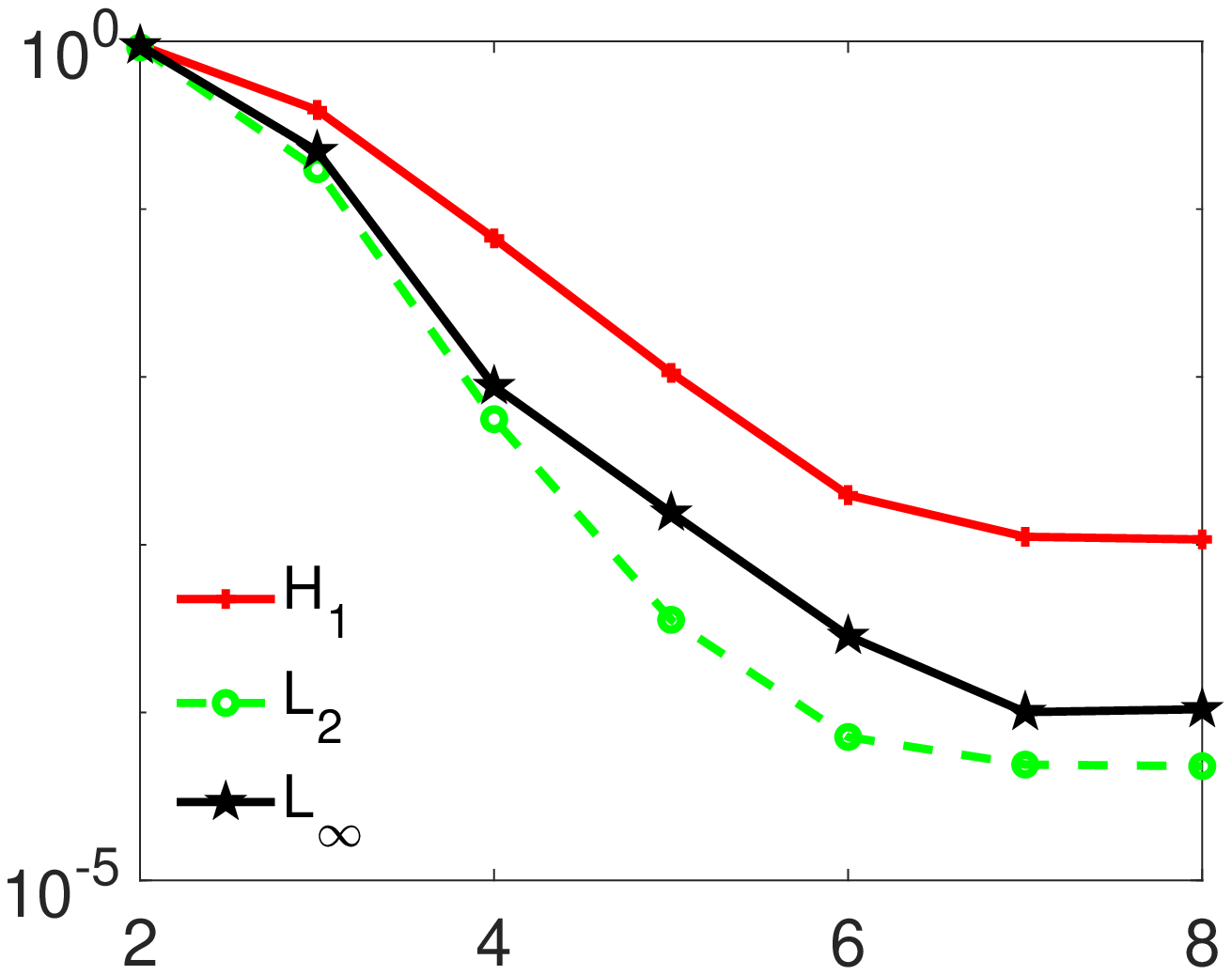}}
	\quad
	\subfigure[GRPS-V, $H=2^{-5}$ \label{fig:PT3}]
	{\includegraphics[width=0.24\textwidth]{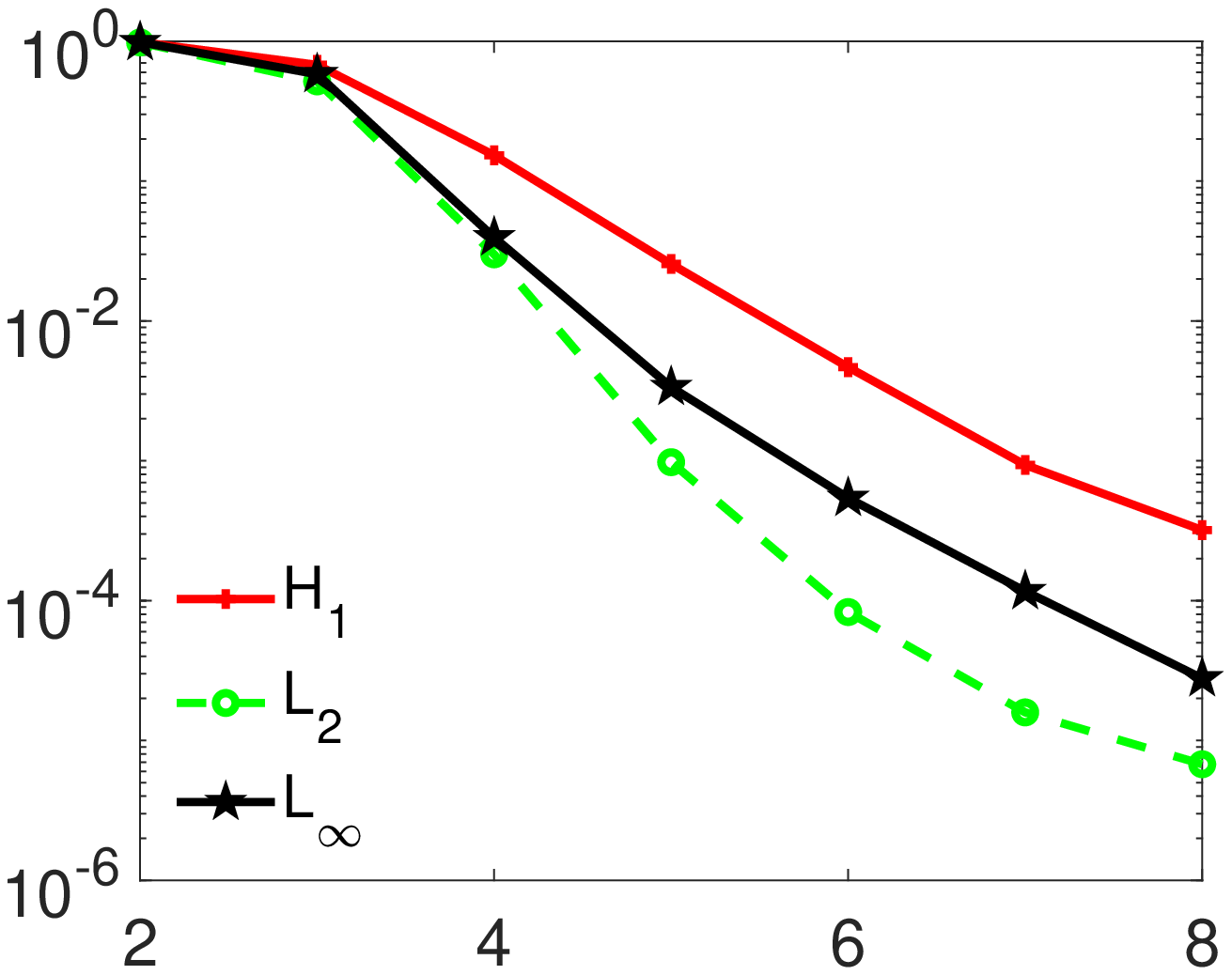}}
	
	\subfigure[GRPS-E, $H=2^{-3}$ \label{fig:PE1}]
	{\includegraphics[width=0.24\textwidth]{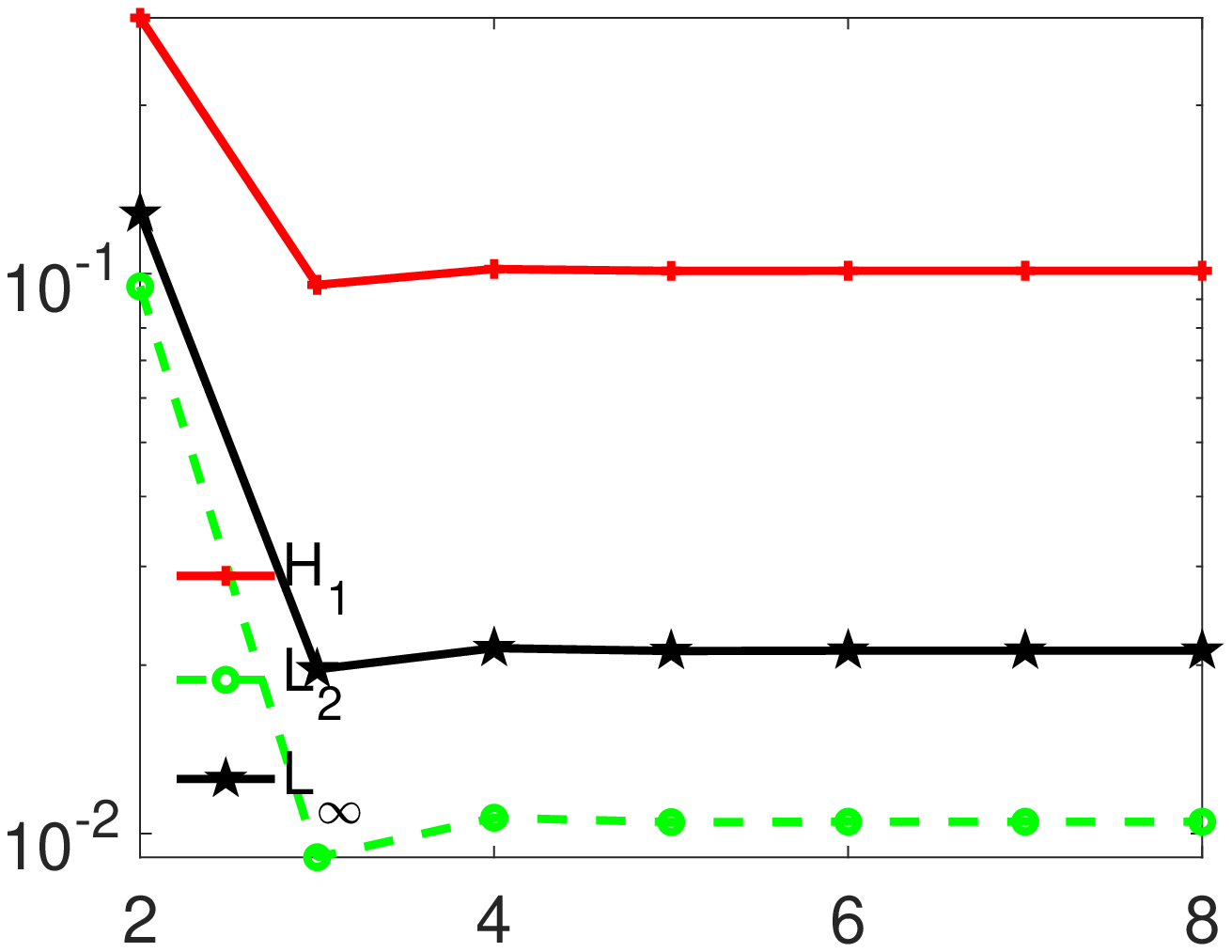}}
	\quad
	\subfigure[GRPS-E, $H=2^{-4}$ \label{fig:PE2}]
	{\includegraphics[width=0.24\textwidth]{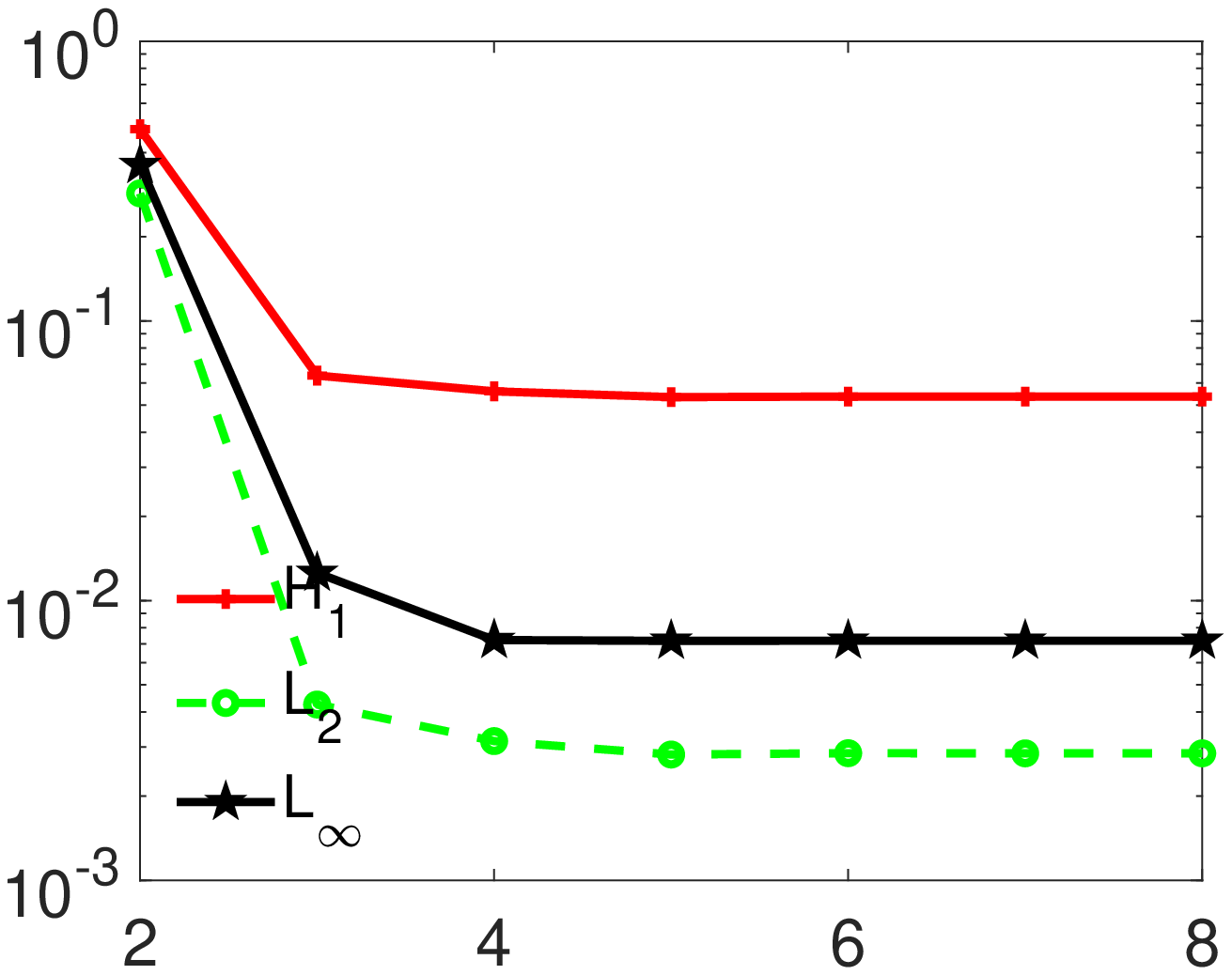}}
	\quad
	\subfigure[GRPS-E, $H=2^{-5}$ \label{fig:PE3}]
	{\includegraphics[width=0.24\textwidth]{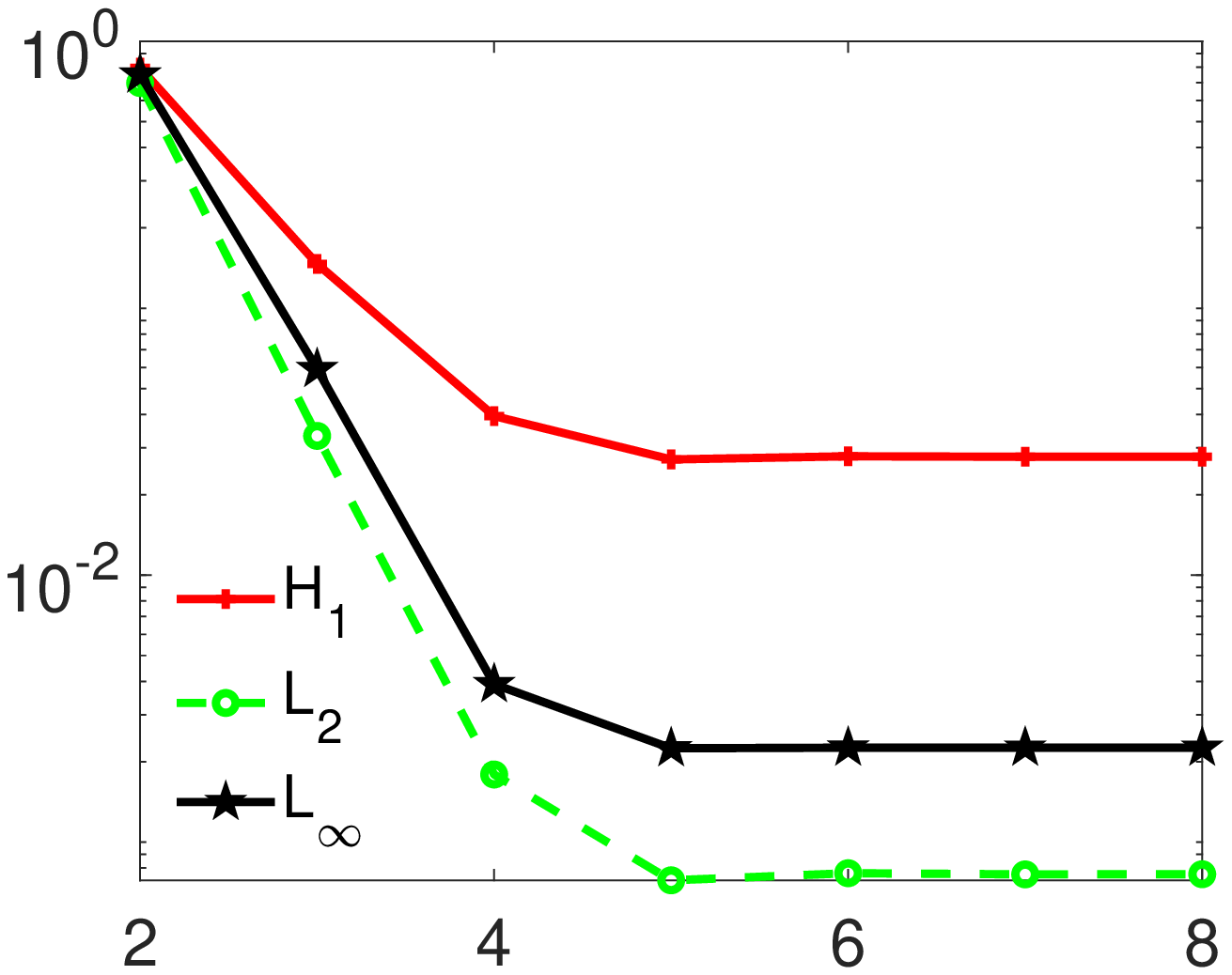}}
	\quad
	
	\subfigure[GRPS-D, $H=2^{-3}$ \label{fig:PD1}]
	{\includegraphics[width=0.24\textwidth]{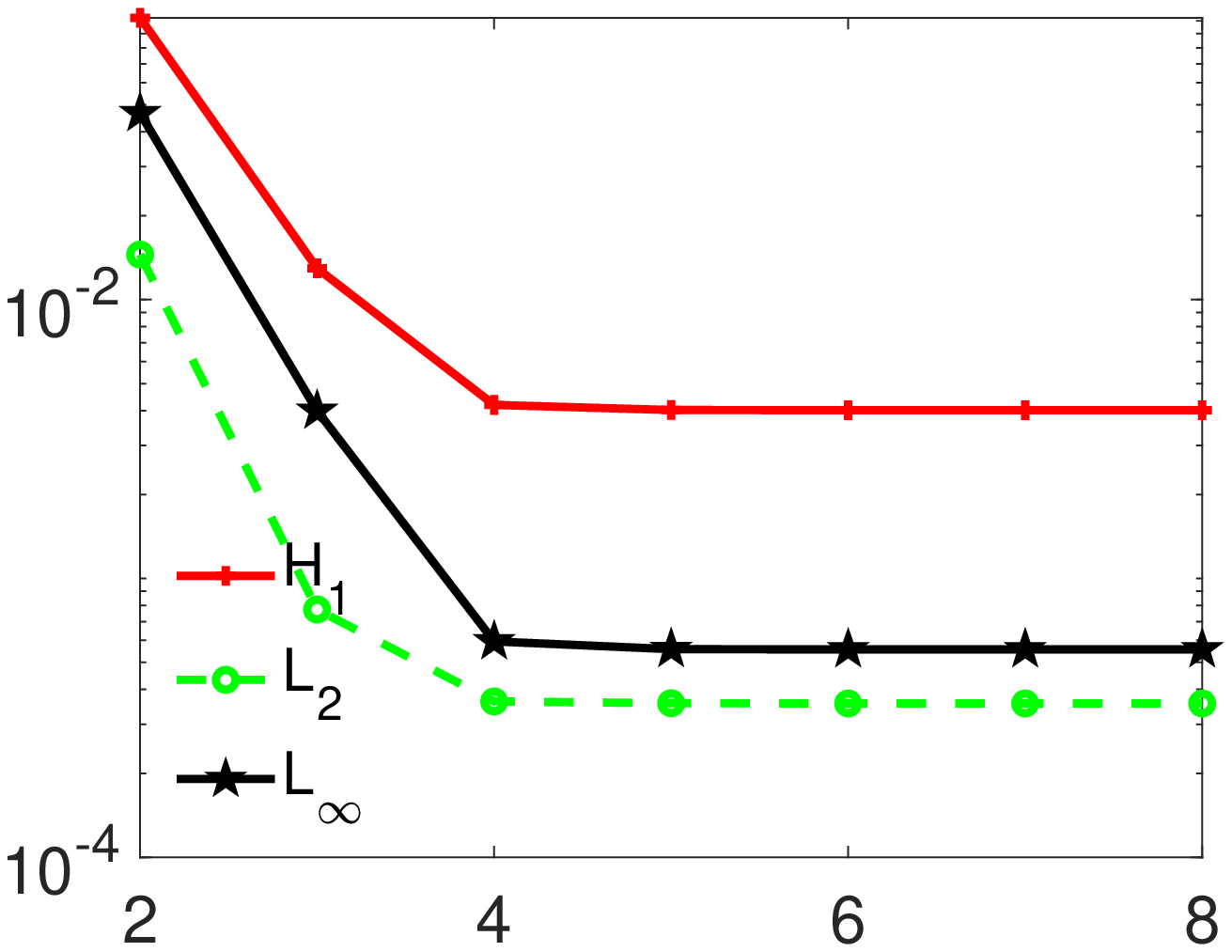}}
	\quad
	\subfigure[GRPS-D, $H=2^{-4}$ \label{fig:PD2}]
	{\includegraphics[width=0.24\textwidth]{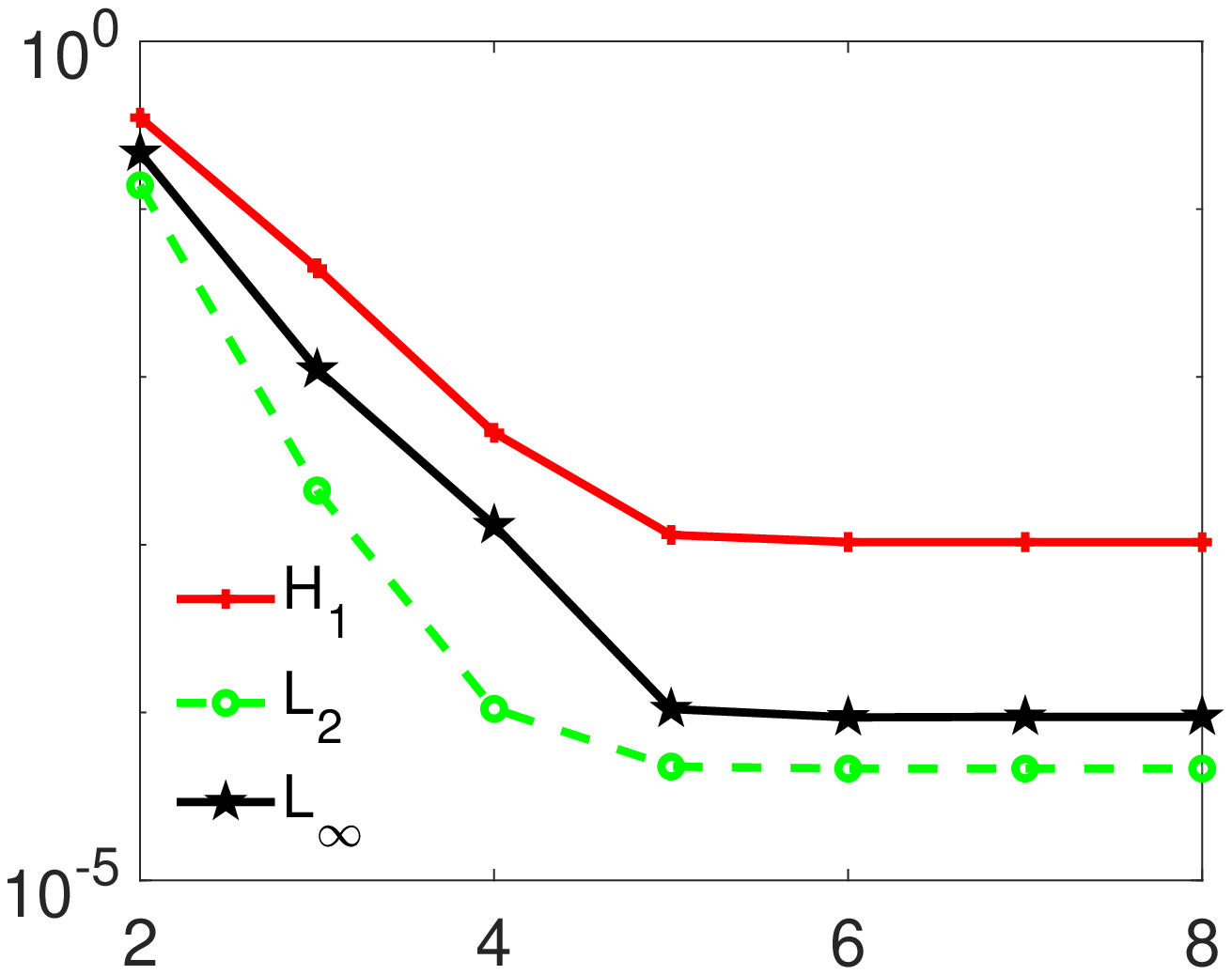}}
	\quad
	\subfigure[GRPS-D, $H=2^{-5}$ \label{fig:PD3}]
	{\includegraphics[width=0.24\textwidth]{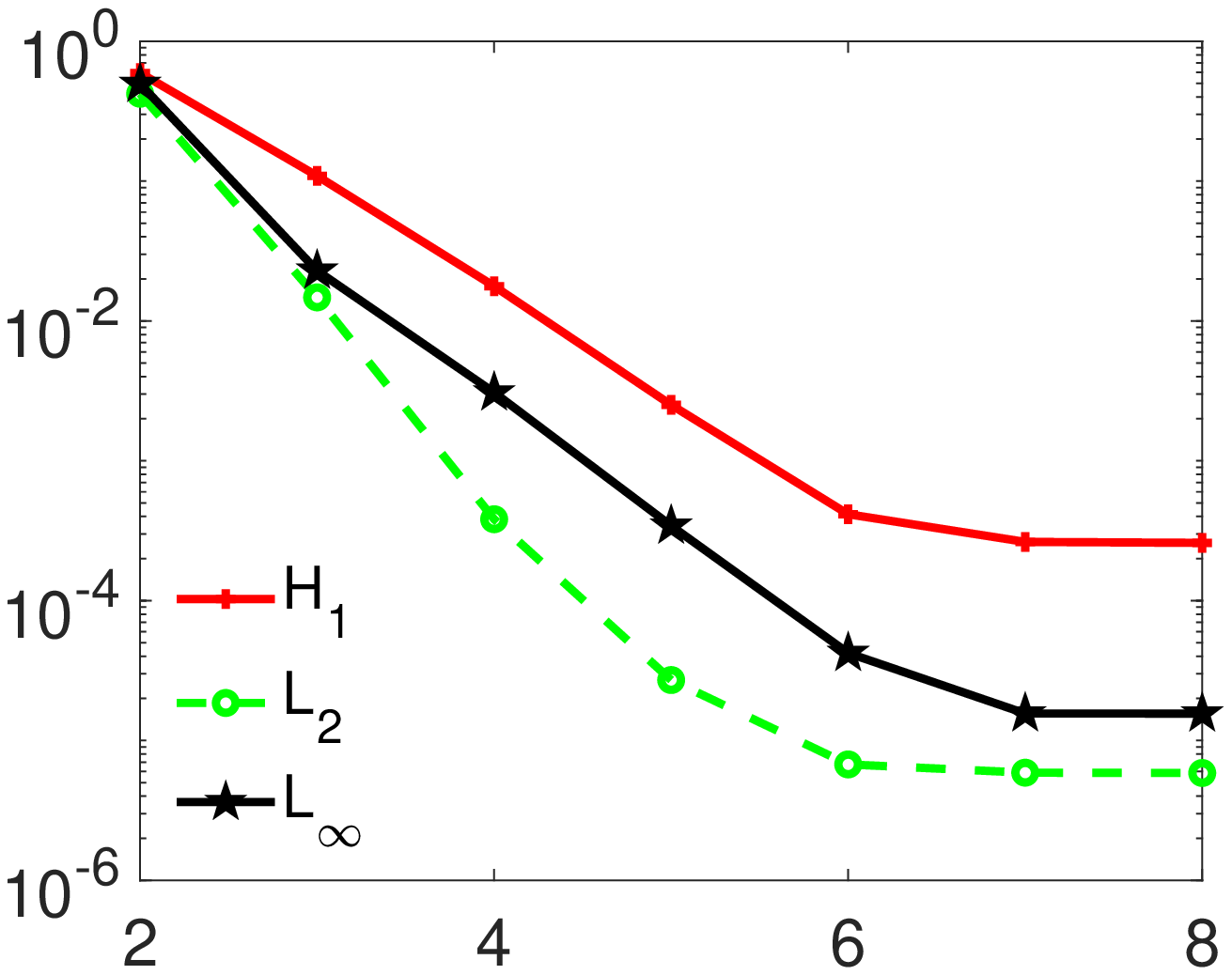}}	
	\caption{Comparison of the convergence with respect to the localization level $\ell$, between RPS, GRPS-V, GRPS-E and GRPS-D. The x-axis stands for the localization level $\ell$ and the y-axis stands for the relative error $\|u_h-u_H^{\ell}\|_{H^1_0(\Omega)}/\|u_h\|_{H^1_0(\Omega)}$ in the $\log_{10}$-scale. } 
	\label{fig:loclevl}
\end{figure}

In Figure \ref{fig:loclevl}, we show the convergence of GRPS bases with respect to the localization level $\ell$ and the coarse mesh size $H$. For a fixed $H$, the accuracy improves with increasing $\ell$, until it reaches the saturation level $O(H)$. 

In Figure \ref{fig:comp_phi}, we demonstrate that, for the fixed localization level $\ell=6$, GRPS-D has an approximately second order convergence rate. Compared with other bases, GRPS-D has the smallest computational cost to achieve a given approximation error, in terms of the number of layers and coarse degrees of freedom .

\begin{figure}[H]
	\centering
	{\includegraphics[width=0.7\textwidth]{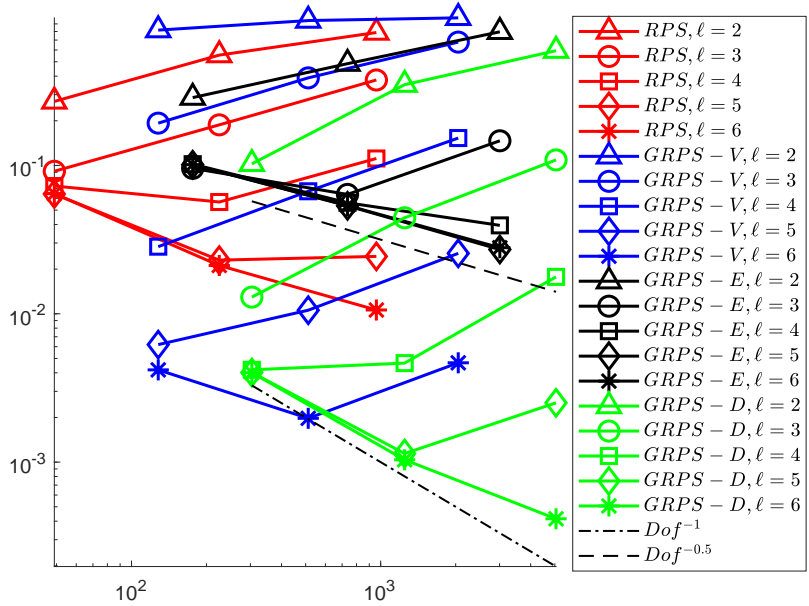}}
	\caption{Convergence curves (error vs. dof) for the \textit{mstrig} example, with fixed localization levels $\ell=2,\,3,\,4,\,5,\,6$, respectively. The x-axis stands for the degrees of freedom in the $\log_{10}$-scale and the y-axis stands for the relative error $\|u_h-u_H^{\ell}\|_{H^1_0(\Omega)}/\|u_h\|_{H^1_0(\Omega)}$ in the $\log_{10}$-scale. }
	\label{fig:comp_phi}
\end{figure}


\subsection{SPE10}
\label{sec:numerics:spe10}

The second example is the  \texttt{SPE10} benchmark problem \footnote{\texttt{http://www.spe.org/web/csp/}}, which is a prototypical example with high contrast heterogeneous coefficient. The physical domain is the rectangular cuboid  $[0,220] \times [0,60] \times[0,85]$, with piecewise constant coefficient $\kappa(x_1,x_2,x_3)$ given on grid points. We select the coefficients over layer 63 with respect to $z$-axis for our two dimensional test problems, and illustrate its contour in Figure \ref{fig:spe10}. Layer 63 possesses the so called channel features, which makes the problem more challenging. $g(x_1,x_2) = \sin(x_1)$.

\begin{figure}[H]
	\centering
	{\includegraphics[width=0.85\textwidth]{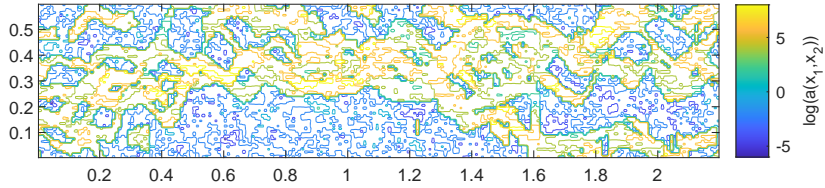}}
	\caption{SPE10 layer 63, $\mathrm{contrast}\approx  10^{16}$.}
	\label{fig:spe10}
\end{figure}

In Figure \ref{fig:comp_phi_2}, we show the performance of RPS, GRPS-V, GRPS-E, and GRPS-D bases for the SPE10 example. For the same degrees of freedom, GRPS-D has the best accuracy compared to others, and it achieves a stable convergences rate in less number of layers $\ell$. 

\begin{figure}[H]
	\centering
	\includegraphics[width=0.80\textwidth]{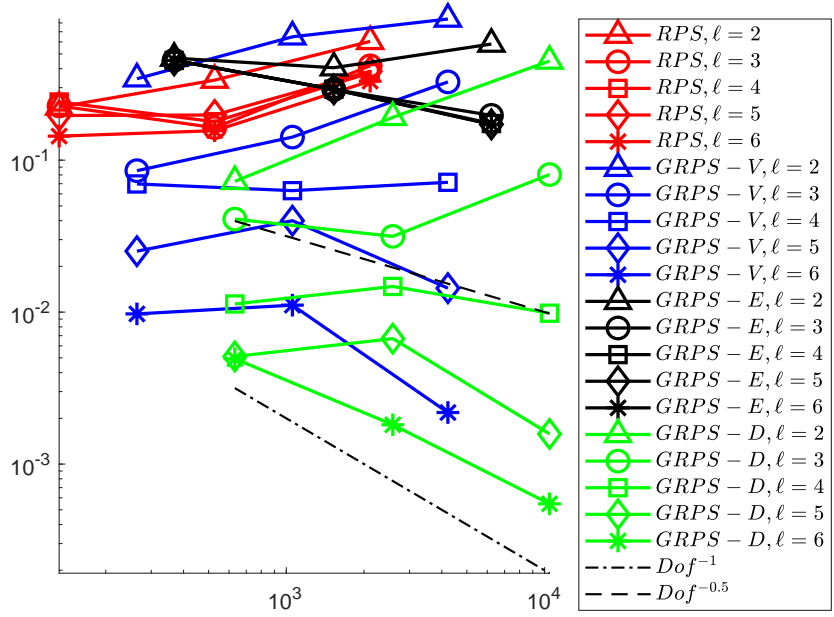} 
	\caption{Convergence curves(error in vs. dof) for SPE10's 63 layer, with fixed localization levels $\ell=2,\,3,\,4,\,5,\,6$, respectively. The x-axis stands for the degrees of freedom in the $\log_{10}$-scale and the y-axis stands for the error relative error $\|u_h-u_H^{\ell}\|_{H^1_0(\Omega)}/\|u_h\|_{H^1_0(\Omega)}$ in the $\log_{10}$-scale. }
	\label{fig:comp_phi_2}
\end{figure}

\subsection{Wave equation in heterogeneous media}

Now we investigate the convergence property of GRPS bases for the wave equation in stationary heterogeneous media,
\begin{equation}
\begin{cases}
&u_{tt}-\nabla \cdot \kappa(x) \nabla u =f(x,t),  x\in \Omega \\
& u(x,0)=u_0(x)\, \, \text{ on } \Omega, \\
& u_t(x,0)=v_0(x)\,\,\text{ on } \Omega. \\
& u(x,t)=0, \,\, \text{ on } \partial \Omega \times[0,T].
\end{cases}
\label{eqn:wave}
\end{equation}
where $\kappa(x)$ is taken as the \textit{mstrig} coefficient in \eqref{eqn:benchmark1}, $\Omega=[0,1]\times [0,1]$, $u(x,t)=0$, $u_t(x,0)=\sin(2\pi x_1)\sin(2\pi x_2)$, and $f(x,t)=0$. 

We employ the temporal discretization in \cite{OZ11,owhadi2008homogenization}. To be more precise, let $M \in \mathbb{N}$, $\Delta t = T/M$, $\left(t_{n}=n \Delta t\right)_{0 \leqslant n \leqslant M}$ be a discretization of $[0, T]$. We write the trial space $\Psi_{T}^{\ell}$ as
\begin{align*}
\Psi_{T}^{\ell} := &\{ w \in L^2(0,T; H^1_0(\Omega))| w(x, t)=\sum_{i} c_{i}(t) \psi_{i}^{\ell}(x), \\
& c_{i}(t) \text{ are linear on } \left(t_{n}, t_{n+1}\right] \text{ and continuous on } [0,T] \}.
\end{align*}
We denote $u_H^{\ell} \in \Psi_{T}^{\ell}$ as the finite element solution of the following implicit weak form, such that, for any $v \in \Psi^{\ell}$, with $\Psi^\ell$ the (localized) GRPS space,  
\begin{equation}
[v, \partial_{t} u_H^{\ell}(t_{n+1})]-[v, \partial_{t} u_H^{\ell}(t_{n})]
= -\int_{t_{n}}^{t_{n+1}} a(v, u_H^{\ell}) \dt +\int_{t_{n}}^{t_{n+1}}v f \dt 
\label{eqn:wave_weak_form}
\end{equation}
In \eqref{eqn:wave_weak_form}, $\partial_{t}  u_H^{\ell}(t) \text { stands for } \lim _{\epsilon_{\downarrow} 0}\left( u_H^{\ell}(t)- u_H^{\ell}(t-\epsilon)\right) / \epsilon$. Once we know the values of $u_H^{\ell}$ and $\partial_{t} u_H^{\ell}$ at $t_n$, \eqref{eqn:wave_weak_form} is a linear system for the unknown coefficients of $\partial_{t} u_H^{\ell}(t_{n+1})$ in $\Psi^{\ell}$.  By continuity of $u_H^{\ell}$ in time, we obtain $u_H^{\ell}(t_{n+1})$ by
\begin{equation}
	u_H^{\ell}(t_{n+1}) = u_H^{\ell}(t_{n}) + \partial_{t} u_H^{\ell}(t_{n+1})\Delta t.
	\label{eqn:wave_weak_form2}
\end{equation}

We use a fixed fine mesh with $h=2^{-8}$ and coarse meshes with $H=2^{-3},\,2^{-4},\,2^{-5}$. We denote $u_h(x,t)$ as the reference finite element solution obtained to the same weak formulation \eqref{eqn:wave_weak_form} and \eqref{eqn:wave_weak_form2} on the fine mesh. We take $\Delta = 1/200$ and compute the solution up to time $T = 1$. 


The approximation error is measured by
$$\|u_h-u_H^{\ell}\|_{L^2(0,T;H^1_0(\Omega))}:= (\int_0^T \|u_h(\cdot,t)-u_H^{\ell}(\cdot,t)\|^2_{H^1_0(\Omega)} \dt)^{1/2}.$$  
Figure \ref{fig:wave_1} illustrates the convergence behavior with respect to the coarse mesh resolution. The convergence is nearly linear when $\ell$ is large enough and GRPS-D achieves a stable convergences rate in less number of layers $\ell$ compared with others. 

\begin{figure}[H]
	\centering
	{\includegraphics[width=0.80\textwidth]{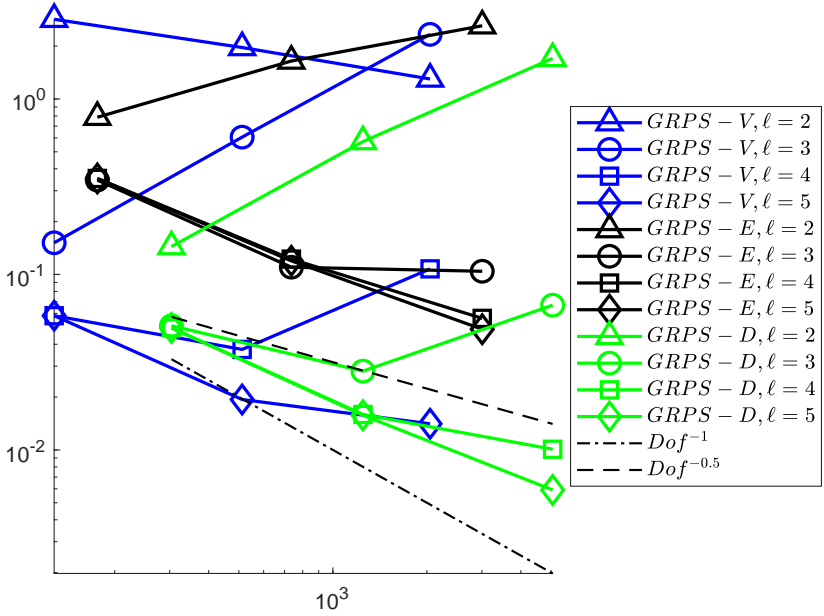}}\quad
    \caption{Convergence curves of wave equation with GRPS bases with fixed localization levels $\ell==2,\,3,\,4,\,5,\,6$, respectively. The x-axis stands for the coarse degrees of freedom in the $\log_{10}$-scale and the y-axis stands for the relative error $\|u_h-u_H^{\ell}\|_{L^2(0,T;H^1_0(\Omega))}/\|u_h\|_{L^2(0,T;H^1_0(\Omega))}$ in the $\log_{10}$-scale. }
    \label{fig:wave_1}
\end{figure}




\section{Conclusion}
\label{sec:conclusion}

In this paper, we generalize the RPS and Gamblet bases for numerical homogenization within the Bayesian framework. We propose to use the edge and first order derivative measurements to construct new generalized rough polyharmonic splines (GRPS) basis. Such a generalization requires some new techniques to prove the localization and convergence properties. Theoretical results on these new GRPS bases are developed and numerical justifications are provided. It seems that those bases are efficient for certain multiscale PDEs.

In this paper, we only consider the case of the second order elliptic operator, it is important to note that the framework works for general integro-differential operators \cite{owhadi2019operator}. For example, we can apply the method to heterogeneous elastic-plastic dynamics \cite{zhang2010global}, and furthermore, to nonlinear multiscale equations \cite{liu2021iterated}. 

\section*{Acknowledgement}

This work is partially supported by the National Natural Science Foundation of China (NSFC 11871339, 11861131004). Dr Zhu's research is further supported by Foundation of LCP (No.6142A05180501), BNU-HKBU United International College(UIC) Start-up Research Fund (No.R72021114) and NSFC (No.11771002, 11571047, 11671049, 11671051, 6162003, and 11871339). 

\bibliographystyle{spmpsci}      

\bibliography{nh}
\appendix
\section{Appendix}
\subsection{Proof of Proposition \ref{prop:span_eq}} \label{sec:prop:span_eq}
\begin{proof}
Remark \ref{rem:deriv} implies that $ \operatorname{span}\{\phi_{\tau,\alpha}\}_{\tau\in \TH, \alpha \in \mathcal{A}} \subset  \operatorname{span}\{\phi_{e}\}_{e\in \EH}$. For the opposite direction, without loss of generality, we start with an element $\tau_1$, with one edge on the boundary and two edges in the interior of the domain, as illustrated in Figure \ref{fig:illu1}. Two derivative measurement functions over $\tau_1$ are linearly independent, and they span the same space as two edge measurement functions on the interior edges of $\tau_1$, due to the Dirichlet boundary condition and \eqref{eqn:deriv}.

We can continue the argument by adding neighboring elements (sharing common edges), for example, $\tau_2$ or $\tau_3$ in \ref{fig:illu1}, until we cover the whole domain $\Omega$. Each time we add an element, by \eqref{eqn:deriv}, the newly added edge measurement functions in the element can be linear represented by derivative measurement functions in the new element and the edge measurement function on the shared edge. The latter can again be linearly represented by derivative measurement functions in the existing elements. Therefore, we have  $ \operatorname{span}\{\phi_{\tau,\alpha}\}_{\tau\in \TH, \alpha \in \mathcal{A}} \supset  \operatorname{span}\{\phi_{e}\}_{e\in \EH}$. 

We note that the number of independent edge measurement functions may increase by two ($\tau_2$ in Fig. \ref{fig:illu1}), or one ($\tau_5$ in Fig.\ref{fig:illu2}), or even zero ($\tau_4$ in Fig.\ref{fig:illu2}) when a neighboring element is added. This implies $\{\phi_{\tau,\alpha}\}_{\tau\in \TH, \alpha \in \mathcal{A}}$ can be linearly dependent.  

\begin{figure}[H]
		\centering
		\subfigure[Extended region\label{fig:illu1}]{\includegraphics[width=0.35\textwidth]{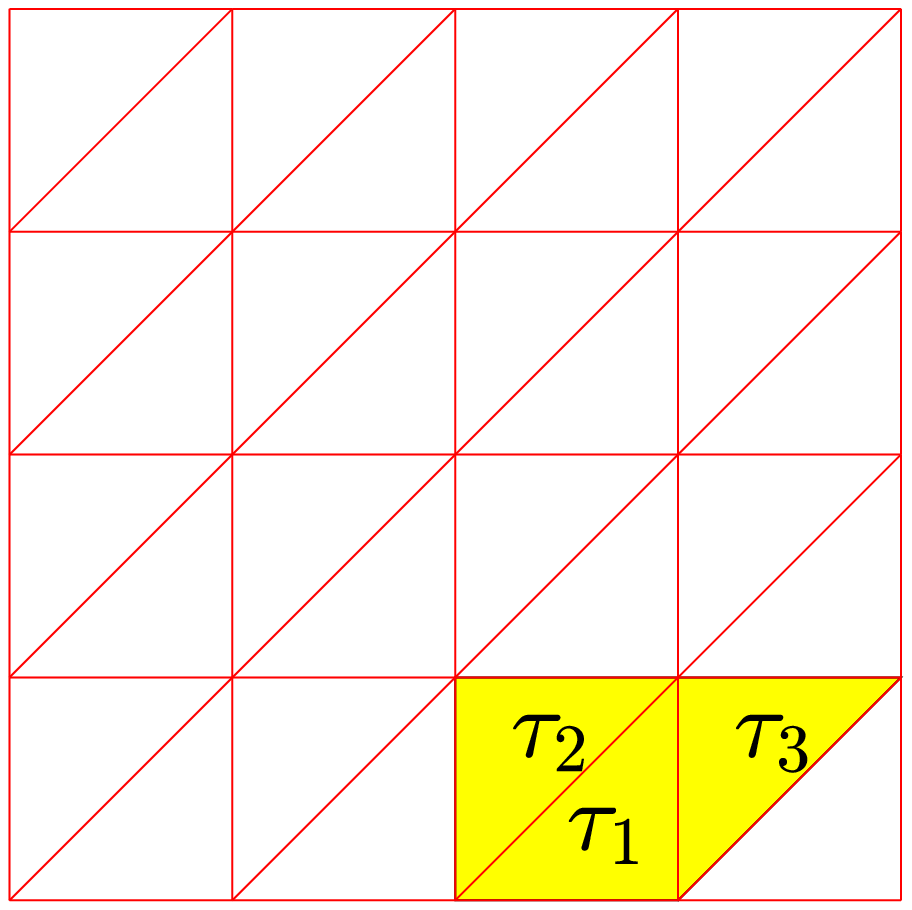}}
		\subfigure[Extended region by adding another layer \label{fig:illu2}] {\includegraphics[width=0.35\textwidth]{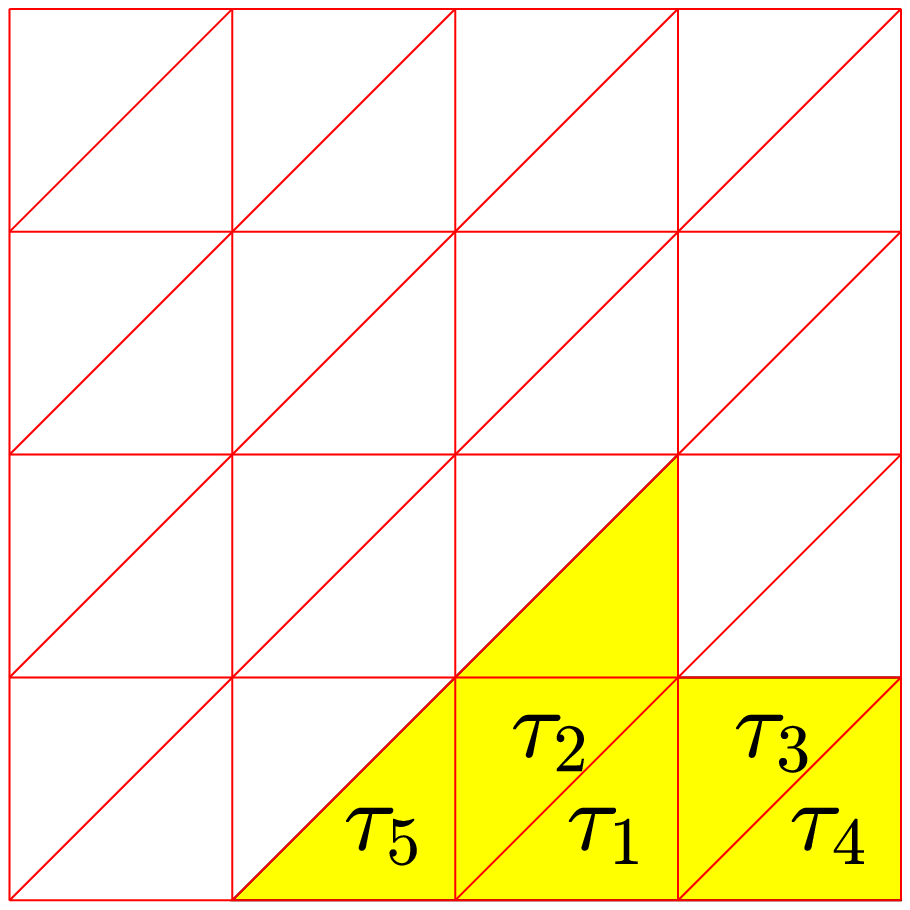}}	
		\caption{illustration of proposition \ref{prop:span_eq} }
		\label{fig:illu}
	\end{figure}
\end{proof}

\begin{lemma}	\label{lem:uniq}
	If $\lambda_{min}(P) > 0 $, then $\xi_i:=\psi_i^0-\psi_i$ is the unique solution in $\Phi^{\perp}$ such that 
	\begin{equation}
	P\xi_i = P\psi_i^0  ,
	\label{eqn:solvingchi}
	\end{equation}
	where $\psi_i^0$ is defined in \eqref{eqn:psi_i^0}.
\begin{proof}
	The definition of $\psi_i^0$ implies that $\xi_i \in \Phi^{\perp} $. Since $\psi_i$ is the unique minimizer of \eqref{eqn:psi},  we obtain that $\psi_i$ is a-orthogonal to subspace $\Phi^{\perp}$, i.e. $P\psi_i=0$, hence $P\xi_i = P\psi_i^0$.  The uniqueness is implied by $\lambda_{min}(P) >0$. 
\end{proof}
\end{lemma}
\subsection{Proof of Lemma \ref{prop:iteration}} \label{Sec:lemma subspace decomposition}
\begin{proof}
	In view of Lemma \ref{lem:uniq}, we consider the approximation error  $\|\xi_{i,{\ell}}-\xi_i\|$, instead of analyzing $\|\psi_i^{\ell}-\psi_i\|$ directly. Let $\psi_{i,0}=\psi_i^0$,  $\xi_{i,0}=0$,  and $\xi_{i,{\ell}}$ is constructed by induction
	\begin{equation}
	\xi_{i,{\ell}+1} = \xi_{i,{\ell}} + \beta P(\psi_{i,0}-\xi_{i,{\ell}}), \label{interation}
	\end{equation}
	where $\beta$ is a parameter to be identified later. Recalling that applying $P$ expands the support by one layer, we have $\xi_{i,{\ell}} \in \Phi^{\perp}\cap H^1_0(\Omega_i^{\ell})$. The iteration scheme (\ref{interation}) can be viewed as a Richardson iteration method to solve $\xi_i$ from \eqref{eqn:solvingchi}, with the iteration matrix $I-\beta P$. Taking $\beta=2/(\lambda_{max}(P)+\lambda_{min}(P))$, we deduce  that $\|I-\beta P\| \leq \big(\frac{\mathrm{cond}(P)-1}{\mathrm{cond}(P)+1}\big) $, therefore  it holds true that
	\begin{equation}
	\|\xi_{i,{\ell}}-\xi_i\|\leq \big(\frac{\mathrm{cond}(P)-1}{\mathrm{cond}(P)+1}\big)^{\ell} \|\xi_i\| .
	\end{equation}
	Let $\psi_{i,{\ell}}:=\psi_{i,0}-\xi_{i,{\ell}}$, we have
	\begin{equation}
	\|\psi_{i,{\ell}}-\psi_i\| = \|\xi_i-\xi_{i,{\ell}}\| \leq  \big(\frac{\mathrm{cond}(P)-1}{\mathrm{cond}(P)+1}\big)^{\ell} \|\xi_i\| . 
	\label{eqn:approximation}
	\end{equation}
	Since $\psi_{i,{\ell}}-\psi_i^{\ell} \in \Phi^{\perp}\cap H^1_0(\Omega_i^{\ell}) $, the variational property of $\psi_i$ and $\psi_i^{\ell}$ implies that $a(\psi_{i,{\ell}}-\psi_i^{\ell}, \psi_i^{\ell}-\psi_i)=0$, therefore,
	\begin{equation}
	\begin{aligned}
	\|\psi_i^{\ell}-\psi_i\|  &\leq  \|\psi_{i,{\ell}}-\psi_i^{\ell}\|+\|\psi_i^{\ell}-\psi_i\| = \|\psi_{i,{\ell}}-\psi_i\|\\
	& \leq \big(\frac{\mathrm{cond}(P)-1}{\mathrm{cond}(P)+1}\big)^{\ell} \|\xi_i\| .
	\label{eqn:approx}
	\end{aligned}
	\end{equation} 
	The a-orthogonality of $\psi_i$ to $\Phi^{\perp}$ also implies,
	\begin{equation}
	\|\psi_i^0\|^2 = \|\psi_i\|^2 +\|\xi_i\|^2 ,\ \|\xi_i\| \leq \|\psi_i^0\|.
	\label{eqn:psi_i^0approx}
	\end{equation}

	With \eqref{eqn:approx} and \eqref{eqn:psi_i^0approx}, we finally draw the conclusion that 
	
	\begin{equation}
	\|\psi_i^{\ell}-\psi_i\|\leq \big(\frac{\mathrm{cond}(P)-1}{\mathrm{cond}(P)+1}\big)^{\ell} \|\psi_i^0\| .
	\end{equation} 
\end{proof}

\subsection{Proof of Lemma \ref{lemma:predecomposition} } \label{Sec:lemma:predecomposition}
\begin{proof}
	Combining the gradient estimate for $\eta_{\hat{\dotlessi}}$
	\begin{equation}
	|D\eta_{\hat{\dotlessi}}|_{L^{\infty}} \leq C_{\gamma}H^{-1},
	\label{eqn:partion_stability}
	\end{equation}
    the $H^1$ seminorm estimate for $v_{\hat{\dotlessi}}$
	\begin{equation*}
	|v_{\hat{\dotlessi}}|_1^2\leq  |D\eta_{\hat{\dotlessi}}|_{L^{\infty}} \|\chi\|_{L^2(\omega_{\hat{\dotlessi}})}^2+\|\nabla\chi\|_{L^2(\omega_{\hat{\dotlessi}})}^2,
	\end{equation*}
	and the Poincar\'{e}'s inequality, we have 
	\begin{equation}
	\|v_{\hat{\dotlessi}}\|_{H_0^1(\omega_{\hat{\dotlessi}})} ^2\leq C_{\gamma}(H^{-1} \|\chi\|_{L^2(\omega_{\hat{\dotlessi}})}^2+\|\nabla\chi\|_{L^2(\omega_{\hat{\dotlessi}})}^2 ).
	\label{eqn:stability}
	\end{equation}
	
	Recalling the definition of the overlapping number $n_{max}$ in Lemma \ref{lemma:cond(P)}, it follows that 
	\begin{equation}
	\sum_{\hat{\dotlessi}\in \Ih}\|\chi\|_{H^1(\omega_{\hat{\dotlessi}})}^2 \leq  n_{max}\|\chi\|_{H^1(\Omega)}^2,\, \text{ and } \, \sum_{\hat{\dotlessi}\in \Ih}\|\chi\|_{L^2(\omega_{\hat{\dotlessi}})}^2 \leq  n_{max}\|\chi\|_{L^2(\Omega)}^2. \label{eqn:overlap}
	\end{equation}
	Combining  \eqref{eqn:stability}, \eqref{eqn:overlap} and Poincar\'{e}'s inequality \eqref{eqn:poincare}, we have 
	\begin{equation}
	\sum_{\hat{\dotlessi}\in \Ih} \|v_{\hat{\dotlessi}}\|_{H_0^1(\omega_{\hat{\dotlessi}})} ^2\leq C  \|\chi\|_{H_0^1(\Omega)}^2 ,
	\end{equation}
	which further implies the following inequality due to the equivalence of $\|\cdot\|$ and $\|\cdot\|_{H_0^1(\Omega)}$, 
	\begin{equation}
	\sum_{\hat{\dotlessi}\in \Ih} \|v_{\hat{\dotlessi}}\| ^2\leq C \kappa_{max}\kappa_{min}^{-1} \|\chi\|^2 ,
	\end{equation}
	where $C$ only depends on $\gamma$, $d$.
	\end{proof}

\subsection{Proof of Lemma \ref{lemma:stability}} \label{sec:lemma:stability}
\begin{proof}
	For three cases, it suffices to show 
	\begin{equation}
	 \|\tilde{P}v_{\hat{\dotlessi}}\| \leq C \|v_{\hat{\dotlessi}}\|,
	 \label{eqn:stable}
	\end{equation} where $C$ only depends on $d$ and $\gamma$.
	  
	{Case V:}
	By Poincar\'{e} inequality, 
	$$
	[\phi_i, v_{\hat{\dotlessi}}]\leq \|\phi_i\|_{L^2(\omega_{\hat{\dotlessi}})}\|v_{\hat{\dotlessi}}\|_{L^2(\omega_{\hat{\dotlessi}})}\leq CH\|v_{\hat{\dotlessi}}\|_{H_0^1(\omega_{\hat{\dotlessi}})} ,
	$$ 
	where $C$ only depends on  $\gamma$ and $d$. Again, by the shape regularity, $\max_{\hat{\dotlessi}}\#\{\tau_i|\ \tau_i\subset\omega_{\hat{\dotlessi}}\}$ is bounded by a constant independent of $H$. By Lemma \ref{lemma:psi_i^0} we have $\|\psi^0_i\| \leq C H^{-1}$, where $\psi^0_i$ corresponds to $\tau_i\in \TH$, for $i\in\I$. Therefore 
	\begin{equation*}
	\|\tilde{P}v_{\hat{\dotlessi}}\| 
	\leq a_{max}\|\sum_{\tau_i \subset \omega_{\hat{\dotlessi}}} \psi^0_i [\phi_i, v_{\hat{\dotlessi}}] \|_{H_0^1(\omega_{\hat{\dotlessi}})} 
	\leq a_{max}\sum_{\tau_i \subset \omega_{\hat{\dotlessi}}} \|\psi^0_i [\phi_i, v_{\hat{\dotlessi}}] \|_{H_0^1(\omega_{\hat{\dotlessi}})} 
	\leq  C\|v_{\hat{\dotlessi}}\|_{H_0^1(\omega_{\hat{\dotlessi}})} .
	\label{eqn:InterpStability1}
	\end{equation*}
	{Case E:}
	For any $v\in H_0^1(\omega_{\hat{\dotlessi}}) \text{ and any } e_i \subset \omega_{\hat{\dotlessi}}$, by Lemma \ref{lem:trace}, we have $\|v\|_{L^2(e_i)} \leq C H^{1/2}\|v\|_{H_0^1(\omega_{\hat{\dotlessi}})}$.  
	Therefore, 
	$$
	[\phi_i, v_{\hat{\dotlessi}}] = \int_{e_i}|e_i|^{\frac{2-d}{2(d-1)}} v_{\hat{\dotlessi}} \ds  
	\leq |e_i|^{\frac{1}{2(d-1)}} \|v_{\hat{\dotlessi}}\|_{L^2(e_i)}
	\leq C H\|v_{\hat{\dotlessi}}\|_{H_0^1(\omega_{\hat{\dotlessi}})}, 
	$$
	where $C$ only depends on $\gamma$ and $d$.  Again, by the shape regularity, $\max\limits_{\hat{\dotlessi}\in \Ih}\#\{e_i|\ \omega_{e_i}\subset\omega_{\hat{\dotlessi}}\}$ is bounded by a constant depending on $\gamma$. Therefore, by Lemma \ref{lemma:psi_i^0} we have 
	\begin{equation*}
	\begin{aligned}
	\|\tilde{P}v_{\hat{\dotlessi}}\| 
	\leq a_{max}\|\sum_{\omega_{e_i} \subset \omega_{\hat{\dotlessi}}}  [\phi_i, v_{\hat{\dotlessi}}]\psi^0_i \|_{H_0^1(\omega_{\hat{\dotlessi}})} 
	&\leq a_{max}\sum_{\omega_{e_i} \subset \omega_{\hat{\dotlessi}}} \|[\phi_i, v_{\hat{\dotlessi}}]\psi^0_i  \|_{H_0^1(\omega_{\hat{\dotlessi}})} 
	&\leq  C\|v_{\hat{\dotlessi}}\|_{H_0^1(\omega_{\hat{\dotlessi}})} .
	\end{aligned}
	\label{eqn:InterpStability2}
	\end{equation*}
	{Case D:}
	We obtain $\|\tilde{P}v_{\hat{\dotlessi}}\| \leq  C\|v_{\hat{\dotlessi}}\|_{H_0^1(\omega_{\hat{\dotlessi}})} $ by combining the results in case V and case E. 
	
	It follows from \eqref{eqn:stable} and Lemma \ref{lemma:predecomposition} that 
	\begin{equation}
	\sum_{\hat{\dotlessi}\in \Ih}\|\tilde{P}v_{\hat{\dotlessi}}\|^2 \leq C  \|\chi\|^2.
	\label{eqn:InterpStable}
	\end{equation}
	
	Lemma \ref{lemma:predecomposition} and \eqref{eqn:InterpStable}  indicate that 
	\begin{equation}
	\sum_{\hat{\dotlessi}\in \Ih}\|\chi_{\hat{\dotlessi}}\|^2\leq 2C \|\chi\|^2 ,
	\end{equation}
	where C depends on $ \gamma,  \kappa_{min}, \kappa_{max}$ but not on $H$. 
	\end{proof}
    
\subsection{Proof of Lemma \ref{lemma:psi_i^0}}\label{sec:lemma:psi_i^0} 
\begin{proof}
	Case V can be referred to Lemma 15.24 in \cite{owhadi2019operator}. 
	
	For case E, we can use the scaling argument. Consider $\Omega_{ref}:=[0,1]^2$ consisting of two reference triangles $T_{ref_{1}}$  and $T_{ref_{2}}$ with a common edge $E_1$, as illustrated in Figure \ref{fig:Tref}.
    \begin{figure}[H]
		\centering
		\includegraphics[width=0.4 \textwidth]{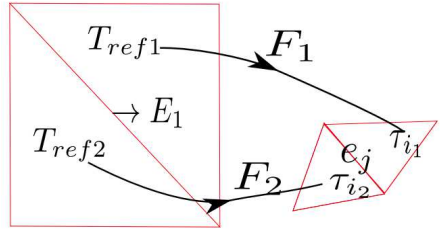}
		\caption{Illustration of reference triangle}
		\label{fig:Tref}
	\end{figure}
    Let	
    \begin{equation}
	\hat{\psi}_{\Omega_{ref}}:=\argmin \limits_{v\in H_0^1(\Omega_{ref}),\int_{E_{1}} \phi_{E_1} v ds=1}\|v\|_{H_0^1(\Omega_{ref})},
	\end{equation}	
	where $\phi_{E_1}$ is the corresponding edge measurement on $E_1$. We have $\|\hat{\psi}_{\Omega_{ref}}\| \leq C$, where $C$ only depends on $\kappa_{max}$ and $d$. We denote the affine mapping $F_k: T_{ref_k} \rightarrow \tau_{i_k}$,  $F_k \hat{x}= B_k \hat{x}+b_k$, $k = 1, 2$, where $F_1(E_1)=F_2(E_1)=e_{j}$. By the shape regularity of $\TH$, we have $B_1$, $B_2$ nonsingular, and $\|B_1\|,\|B_2\| \sim \mathcal{O}(H)$. 
	
	We define $\psi_{i_1}\in H^1(\tau_{i_1})$ by $\psi_{i_1}(x):=\sqrt{2}|e_j|^{-\frac{d}{2(d-1)}}\hat{\psi}(F_1^{-1}{x}),\forall x\in \tau_{i_1}$. By the definition of edge measurement $\phi_{e_j}$ in \eqref{eqn:edge}, we have
	$
	[\psi_{i_1},\phi_{e_j}] = 1,
	$
	and
	\begin{equation}
	|\psi_{i_1}|_1 \leq \|B_1^{-1}\| |e_j|^{-\frac{d}{2(d-1)}} |\operatorname{det}(B_1)|^{1/2}|\hat{\psi}_{\Omega_{ref}}|_1 \leq  C H^{-1},
	\end{equation}
	where $C$ only depends on the shape regularity parameter $\gamma$, $\kappa_{max}$ and $d$. Similarly, we can obtain $\psi_{i_2} $ on $T_{ref_2}$ and patch them together to form a function $\psi_{i_1,i_2}\in H^1_0(\omega_{e_j}) $, which satisfies $\|\psi_{i_1,i_2}\| \leq C H^{-1}$ and  $[\psi_{i_1,i_2},\phi_{e_j}] = 1$. Since $\psi^0_{j}$ is the minimizer in $H^1_0(\omega_{e_j})$, we have $\|\psi^0_{j}\| \leq \|\psi_{i_1,i_2}\| \leq C_1 H^{-1}$, where $C_1$ only depends on $\gamma$, $\kappa_{max}$ and $d$. 
	
	Case D can be proved similarly.
\end{proof}

\end{document}